%% file: revised.tex
\RequirePackage{fix-cm}
\documentclass[smallextended]{svjour3}       % onecolumn (second format)
\smartqed  % flush right qed marks, e.g. at end of proof
\usepackage{graphicx}
\usepackage{multirow}
\usepackage{booktabs,tabularx} 
\usepackage{lscape}

\input{ex_shared}

\begin{document}

\title{Distributed Nonconvex   Constrained Optimization over Time-Varying Digraphs%\thanks{Grants or other notes
}


\author{Gesualdo Scutari         \and
       Ying Sun  
 }


\institute{Scutari and Ying are with the 
              School of Industrial Engineering, Purdue University, West-Lafayette, IN, USA. Emails: \texttt{<gscutari,sun578>@purdue.edu}. 
              This work was supported by the USA National Science Foundation (NSF) under Grants CIF 1564044  and CAREER Award No. 1555850; and the Office of Naval Research (ONR) Grant N00014-16-1-2244.               Part of this work has been presented at the 2016 Asilomar Conference on System, Signal, and Computers \cite{SunScutariPalomar-C'16} and the 2017 IEEE  %International Conference on Acoustics, Speech and Signal Processing (ICASSP)
              ICASSP Conference \cite{SunScutariICASSP17}. %and the 2107 IEEE International Conference of Acoustic Speech and Signal Processing (ICASSP) \textcolor{red}{REF}. % \\
             % Tel.: +123-45-678910\\
             % Fax: +123-45-678910\\
             % \email{fauthor@example.com}           %  \\
           %\and
          % S. Author \at
           %   second address
\vspace{-4.9cm}}
 
\date{
Received: June 3, 2017; Revised June 5, 2018. \vspace{-.2cm}
 }

\maketitle

\begin{abstract}
This paper considers  nonconvex distributed constrained optimization over  networks, modeled as directed (possibly  time-varying) graphs. We introduce the first algorithmic framework for the minimization of the sum of a  smooth nonconvex (nonseparable) function--the agent's sum-utility--plus a Difference-of-Convex (DC)  function (with nonsmooth convex part). This general formulation arises in many  
applications, from statistical machine learning to  engineering.  {The proposed distributed method combines  successive convex approximation techniques with a judiciously designed  perturbed push-sum consensus  mechanism that aims to track  locally the gradient of the (smooth part of the) sum-utility.}  Sublinear convergence rate is proved when  a fixed step-size (possibly different among the agents) is employed whereas asymptotic convergence to stationary solutions  is proved using a diminishing step-size.   Numerical results show that our algorithms  {compare favorably with } 
current schemes on both convex and nonconvex problems.

%\keywords{Distributed algorithms \and Nonconvex  optimization \and Time-varying digraphs.}
\vspace{-0.4cm}
\end{abstract}

\section{Introduction}\vspace{-0.2cm}
\label{intro}

This paper focuses on the following (possibly) \emph{nonconvex} multiagent composite optimization problem:\vspace{-0.2cm}
\begin{equation} 
\min_{\bf{x}\in \mathcal K}\,U\left(\bf x\right)\triangleq \underset{F(\mathbf{x})}{\underbrace{\sum_{i=1}^I f_i\left(\bf x\right)}} + \underset{G\left(\mathbf{x}\right)}{\underbrace{G^{+}\left(\bf x\right) - G^{-}\left(\bf x\right)}},
\tag{P}\label{eq: P}\vspace{-0.2cm}
\end{equation}
where  $f_i:\mathbb{R}^m\to \mathbb{R}$ is the cost function of agent $i$, assumed to be smooth  (possibly) \emph{nonconvex};    $G:\mathbb{R}^m\to \mathbb{R}$ is a DC  function, whose concave part $-G^-$ is   smooth; and %and $G^+$ (possibly) non; and % a regularizer composed by a convex (possibly nonsmooth) function $G^{+}:\mathbb{R}^m\to \mathbb{R}$ and a smooth but (possibly) nonconvex function $G^{-}:\mathbb{R}^m\to \mathbb{R}$;  
  $\mathcal{K}$ is a closed convex subset of $\mathbb{R}^m$. The function $G$ is generally used to promote some extra structure on the solution, like sparsity. Note that, differently from most of the papers in the literature, we do not require   the (sub)gradient of $f_i$, $G^-$ or $G^+$ to be (uniformly) bounded on $\mathcal K$. Agents are connected through a communication network, modeled as a directed graph, possibly  time-varying. Moreover,  each agent $i$ knows only its own function  $f_i$ (as well as $G$ and $\mathcal K$). In this setting, the agents want to cooperatively solve Problem\eqref{eq: P} leveraging local communications with their immediate neighbors.\\ % want to cooperatively minimize  a global objective   by means of actions taken by each agent and local coordination between neighboring nodes.  
\indent Distributed \emph{nonconvex} optimization  in the form   \eqref{eq: P} has found a wide range of applications in several areas, 
including   network  information processing, %(e.g., parameter estimation,
%detection, and localization), 
 telecommunications, %(e.g., resource allocation
%in interference peer-to-peer/multi-cellular systems), 
 multi-agent control,
%(e.g., distributed learning, and flock control), 
 and machine learning. %; see, e.g., \textcolor{red}{[CIME]}.
In particular, Problem~\eqref{eq: P}   is a key enabler of many emerging \emph{nonconvex} ``big data'' analytic tasks,  including   nonlinear
least squares, dictionary learning,  principal/canonical component analysis,     low-rank approximation, and matrix completion \cite{friedman2001elements},
just to name a few. Moreover, the DC structure of   $G$  allows  to accommodate in an unified fashion     convex and  nonconvex sparsity-inducing surrogates  of the $\ell_0$ cardinality function (cf.~Sec.~\ref{sec:prob}). Time-varying communications arise, for instance, in mobile wireless networks (e.g., ad-hoc networks), wherein  nodes are mobile and/or communicate throughout fading channels. Moreover, since  nodes generally transmit    at different power    and/or communication channels are not symmetric,  directed  links are   a natural assumption.\\% transmitter and receivers are geographically collocated (e.g., think of ad-hoc networks).%examples are  the $\ell_p$  ($p\geq 1$), $\ell_{1,2}$ norm, the total variation penalty, the SCAD function, the logarithmic/exponential functions, and the $\ell_p$ norm with $0<p<1$ (see, e.g., \textcolor{red}{[Pang]} and references therein)  
\indent  In most of the  above scenarios,  data processing and optimization need to be performed   in a distributed but collaborative manner by the agents within the network. For instance, this is the case in 
data-intensive (e.g., sensor-network) applications wherein the sheer volume and spatial/temporal disparity of scattered    data  render centralized
processing and storage infeasible or inefficient. \\%This happens, for instance, whenever the volume of data overwhelms
%the storage capacity of a single computing device. 
%Moreover, collecting sensor-network data, which
%are observed across a large number of spatially scattered centers/servers/agents, and routing all this local
%information to centralized processors, under energy, privacy constraints and/or link/hardware failures, is often
%infeasible or inefficient. 
\indent  While distributed  methods for convex optimization  have been widely studied in the
literature, there are no such schemes  for \eqref{eq: P} (cf.~Sec.\,\ref{related_works}).  We propose   the first  family of distributed algorithms that converge to  stationary solutions of \eqref{eq: P} over   time-varying (directed) graphs. %\textcolor{blue}{The proposed algorithmic framework hinges on Successive Convex Approximation (SCA) techniques coupled with a novel consensus mechanism as well as tracking scheme estimating locally the gradient of $F$}. %These protocols are implementable over directed graphs. 
Asymptotic convergence is proved, under the use of either constant uncoordinate step-sizes from the agents or diminishing ones. When a constant step-size is employed, the algorithms are showed to achieve sublinear convergence rate. Furthermore,  the technical tools we introduce are of independent interest. 
 Our analysis hinges on  a descent technique   technique valid for nonconvex, nonsmooth, constrained problems based on a novel {Lyapunov}-like function (see Sec.\,\ref{sec:contributions} for the  list of contributions).%: the value of an  appropriate stationarity measure becomes smaller than $\epsilon>0$ after a number of iterations proportional to $1/\epsilon$. 
  \vspace{-0.5cm}

%signal processing, machine learning,  networking, and decentralized control. Common to these problems is a network of agents--processors, computers of a cluster, nodes of a sensor network, vehicles, or UAVs--that want to cooperatively solve an optimization problem  by means of actions taken by each agent and local coordination between neighboring nodes. 

%Problem~\eqref{eq: P}  encompasses a variety of \emph{convex} and \emph{nonconvex}    applications, including    target localization, resource allocation, least squares/logistic regression,    maximum likelihood  estimation,      principal component analysis,  canonical component analysis, and  low-rank approximation \cite{friedman2001elements}, just to name a few.   
%Our goal  is  developing solution methods for the \emph{nonconvex} Problem~\eqref{eq: P} in the following distributed setting: i) Each agent $i$ knows only its own function  $f_i$ (as well as $G$ and $\mathcal K$); and ii) the communication topology connecting the agents  is \emph{time-varying} and \emph{directed} and it is not known to the agents.  Time-varying communication topologies arise, for instance, in mobile  wireless networks, wherein the nodes are mobile and/or communicate throughout  (fast-)fading channels. Directed communication links are also a natural assumption as in many cases there is no reason to expect different nodes to transmit at the same power level or that transmitter and receivers are geographically collocated (e.g., think of ad-hoc networks). 

\subsection{Related works}\label{related_works}\vspace{-0.3cm}

The design of  distributed algorithms for \eqref{eq: P}   faces the following challenges:    
 (i) $U$ is \emph{nonconvex} and \emph{nonseparable}; (ii) $G^+$  is  \emph{nonsmooth};   (iii)  there are \emph{constraints}; (iv) the graph is \emph{directed} and \emph{time-varying}, with \emph{no specific} structure;  and (v) the (sub)gradient of $U$  is not assumed to be bounded on  $\mathcal K$.  We are not aware of any distributed design    addressing  (even a subset of) challenges (i)-(v), as documented next.   {Since the focus of this work is on distributed algorithms working on  general network architectures, we omit to discuss the vast literature of schemes implementable on {\it specific} topologies, such as   hierarchical networks (e.g., master-slave  or shared memory systems);    see, e.g.,  \cite{wright2015coordinate,palomarTACdistributed,ScuFacSonPalPan2014,facchinei2015parallel,scutari_PartI,scutari_PartIII,Curtis-SIAM18}   and references therein for an entry point of this literature.}
%Besides the distributed approaches we are going to summarize, many efforts have been made to solve convex and nonconvex optimization problems over networks with {\it specific} topology, including multilayer hierarchical networks (e.g., master-slave architectures and shared memory systems); see, e.g.,  \cite{ScuFacSonPalPan2014,facchinei2015parallel,scutari_PartI,scutari_PartII,scutari_PartIII} [add Daniel] and references therein. However these schemes reuire some degree of centralization to be implemented, due to the use of a master node. Since the focus of this work is on distributed algorithms implementable on general network architecture, we omit to discuss the literature on algorithms for  

 %......Parallel algorithms handling nonconvex problems include \cite{ScuFacSonPalPan2014,facchinei2015parallel,scutari_PartI,scutari_PartII,scutari_PartIII} (see also references therein). However   these schemes require some degree of  centralization to be implemented. For instance, they are applicable over hierarchical networks wherein there is a master node able to collect information from all the other agents; or % are not applicable to \eqref{eq: P}:  either they require  the graph to be complete (i.e., at each iteration, agents must be able to exchange information with \emph{all} the others) or 
%they require full knowledge of $F$ from the agents. 

%........

 \noindent \textbf{Distributed \emph{convex} optimization:} Although the focus of this paper is mainly on nonconvex optimization, we begin overviewing  the much abundant literature of  distributed algorithms for \emph{convex} problems. We show in fact  that, even   in this simpler setting,   some of the challenges (ii)-(v) remain unaddressed.
 
 %therein cannot they represent the majority of the literature of distributed algorithms.  We show that %We group   representative literature in primal- and dual-based distributed methods and show that,  %and It is convenient to cast the current literature in two groups, namely: dual  and primal-based distributed methods. 
 %While these schemes have complementary features,  we show  
 %none of them can address all challenges ii)-v). %, even in the convex setting.   %To support this statement, we some of the issues listed above remain unsolved. 

\noindent $-$\textit{Primal methods}: While substantially different, primal methods  can be generically abstracted as a combination of a       local (sub)gradient-like   step  and a subsequent   consensus-like update (or multiple consensus updates); examples include \cite{nedic2015distributed,nedic2009distributedsubgradient,jakovetic2014fast,shi2015extra, shi2015proximal}.     
Algorithms for adaptation and learning tasks based on in-network diffusion techniques were proposed in \cite{cattivelli2010diffusion,chen2012diffusion,sayed2014adaptation}. Schemes in \cite{nedic2009distributedsubgradient,jakovetic2014fast,shi2015extra,shi2015proximal,cattivelli2010diffusion,chen2012diffusion,sayed2014adaptation} are applicable only to 
\emph{undirected} graphs;     \cite{nedic2009distributedsubgradient,cattivelli2010diffusion}   require the consensus matrices to be  \emph{double-stochastic} whereas  \cite{chen2012diffusion,sayed2014adaptation} use only \emph{row-stochastic} matrices but are applicable only to strongly convex agents' cost functions  having  a \emph{common} minimizer. When the graph is \emph{directed}, double-stochastic   weight matrices compliant to the graph  might not exist or are not easy to be constructed in a distributed way \cite{gharesifard2010does}. This requirement    was removed in  \cite{nedic2015distributed} where the authors combined the  sub-gradient algorithm \cite{nedic2009distributedsubgradient} with   push-sum  consensus   \cite{kempe2003gossip}.  Other schemes applicable to digraphs are \cite{xi2015linear,XiKha-J'16}.  However,    \cite{nedic2009distributedsubgradient,xi2015linear,XiKha-J'16}  cannot handle constraints. In fact, up until this work (and the associated conference papers \cite{SunScutariPalomar-C'16,SunScutariICASSP17}) it was not clear  how to leverage push-sum-like protocols to deal with constraints  over \emph{digraphs}.  Finally,  as far as challenge (v) is concerned, only  recent proposals \cite{nedich2016achieving,Xu2015augmented,qu2016harnessing,xi2015linear,XiKha-J'16,shi2015extra,shi2015proximal} removed the  assumption that  the (sub-)gradient of $U$ has to be bounded; however \cite{nedich2016achieving,Xu2015augmented,qu2016harnessing,xi2015linear,XiKha-J'16,shi2015extra}    
   can  handle only \emph{smooth and unconstrained} problems while  \cite{Xu2015augmented,qu2016harnessing,XiKha-J'16,shi2015extra,shi2015proximal} are not implementable over digraphs.

 \noindent $-$\textit{Dual-based methods}: This class of algorithms is based on a different approach: slack variables are first introduced to decouple the sum-utility function while forcing consistency among these local copies by adding consensus equality constraints (compliant with the graph topology). Lagrangian dual variables are then  introduced to deal with such coupling constraints. The resulting algorithms build on 
 primal-dual updates,  aiming at converging to a saddle point of the    (augmented) Lagrangian function. %....  Another approach to tackle Problem \eqref{eq: P}  (matching the graph topology). is to reformulate it as a
   %problem with separable objective and consensus-forcing equality constraints matching the graph topology, 
 %  and updating the primal-dual variable pair to obtain a saddle point of the (augmented) Lagrangian.
%    The agents update primal variable locally and exchange the dual variables among the neighbors. 
   Examples of such algorithms include ADMM-like methods   \cite{chang2014proximal,wei20131,jakovetic2011cooperative,shi2015proximal} as well as  inexact primal-dual instances   \cite{chang2015multi,mokhtari2016decentralized,mokhtari2015dqm}.     All these algorithms can handle only \emph{static and undirected} graphs.  Their extensions to time-varying graphs or digraphs seem  not possible, because it is not clear how to enforce consensus via equality constraints over time-varying or directed networks.  Furthermore, all the above schemes but \cite{shi2015proximal,chang2014proximal} require $U$  
  to be   \emph{smooth} and \eqref{eq: P} to be  \emph{unconstrained}.

 {In summary, even restricting to convex instances of \eqref{eq: P}, there exists no distributed algorithm  in the  literature that can deal with either constraints [issue (iii)] or nonsmooth $U$ [issue (ii)] with  nonbounded (sub-)gradient [issue (v)] over (time-varying) digraphs.}      
Also, it is not clear how to extend the convergence analysis  developed  in   the above papers when $U$ is no longer  convex.
  
\noindent \textbf{Distributed \emph{nonconvex} optimization:} %For nonconvex \eqref{eq: P}, results are scarse. 
 Distributed algorithms dealing with special instances of Problem~\eqref{eq: P} are scarce; they include   primal methods  \cite{bianchi2013convergence,tatarenko2015non,LorenzoScutari-J'16,Wai2017DeFW}  and dual-based  schemes \cite{zhu2013approximate,hong2016decomposing}. The key features of these algorithms are summarized in Table \ref{nonconvex_table} and  discussed next.  

\noindent $-$\textit{Primal methods}: 
The scheme in   \cite{bianchi2013convergence} combines the distributed  stochastic projection algorithm, employing a diminishing step-size, with the random gossip protocol. It    
 can handle \emph{smooth} objective functions over \emph{undirected static} graphs; no rate analysis of the scheme is known. In   \cite{tatarenko2015non}, the authors showed that the (randomly perturbed)   push-sum gradient algorithm with diminishing (square summable) step-size, earlier proposed for convex objectives in \cite{nedic2015distributed}, converges also when applied to nonconvex \emph{smooth unconstrained}  problems.  Asymptotic convergence   and  a  sublinear convergence rate were proved (the latter under the assumption that the set of stationary points of $U$ is finite).  The first, to our knowledge,  provably convergent distributed scheme   for  \eqref{eq: P}, with $G^+\neq 0$ and constraints $\mathcal K$, over time-varying  graphs is NEXT  \cite{LorenzoScutari-J'16}. The algorithm   requires the consensus matrices to be doubly-stochastic. %; it is thus not applicable over digraphs with arbitrary topology. 
 Asymptotic convergence was proved, when a diminishing step-size is employed;   no rate analysis was provided. A special instance of NEXT was studied  in \cite{Wai2017DeFW}, where the authors considered  {\it smooth} (possibly nonconvex) $U$ over \emph{undirected static} graphs. Under a diminishing step-size (and further  technical assumptions on the set of stationary solutions), a sublinear convergence rate %of $O(1/\sqrt{n})$ is proved, %for the decay of a suitably defined stationary gap along the average of the agents' iterates, 
% where $n$ is the iteration index.
is proved. Finally, all the algorithms discussed above require that  the (sub)gradient of $U$ is \emph{bounded} on $\mathcal K$ (or $\mathbb{R}^m$). This is a key assumption to prove convergence: in the analysis of descent, it permits to  treat the optimization and consensus steps \emph{separately}, with the consensus error being  a  \emph{summable} perturbation. \\
  $-$\textit{Dual-based methods}:  In \cite{zhu2013approximate}  
 a distributed approximate dual subgradient algorithm, coupled with a consensus scheme (using double-stochastic weight matrices),  is introduced to solve \eqref{eq: P} over time-varying graphs. 
Assuming zero-duality gap, the algorithm is   proved to asymptotically find a pair of primal-dual solutions  of an auxiliary problem%(associated to the original one)
, which however    might not be   stationary for the original problem; also, consensus is not guaranteed. No rate analysis is provided.   In \cite{hong2016decomposing}, a proximal primal-dual algorithm is proposed to solve an \emph{unconstrained}, \emph{smooth} instance of \eqref{eq: P} over {\it undirected static} graphs.  The algorithm employs either a constant or increasing penalty parameter (which plays the role of the step-size); a global sublinear convergence rate is proved. %for a merit function measuring distance from stationarity and consensus disagreement. 
The algorithm can also  deal with nonsmooth convex regularizes and norm constraints when it is applied to  some  distributed matrix factorization problems.  \\
      %use the objective function as Lyapunov function; and to use standard descent arguments while treating optmizaiton and consensus steps separately 
 % and invoke standard descent arguments on  the objective function by which ; and ii) consensus and optimization can be \emph{decoupled}, convergence is proved  invoking  standard descent arguments by treating the  consensus steps as additive summable error on the induction e  optimization and consensus are not coupled in the descent, in the sense that the consensus just appears as  the descent can be studied \emph{separately}. In fact,   optimization and consensus updates  can be  \emph{decoupled}:  these schemes, since it allows one to  by which because  ........ 
  % As their convex counterparts, the dual-based schemes  \cite{zhu2013approximate,hong2016decomposing}     cannot deal with {\it time-varying or directed} graphs. 
\input{intro_table}  
\noindent\textbf{Gradient-tracking:}   The proposed algorithmic framework leverages the idea of gradient tracking: each agent updates its own local variables along  a direction that is a proxy of the sum-gradient $\sum_{i=1}^I \nabla f_i$ at the current iteration, an information that is   not locally available.  The idea of  tracking the gradient averages through the use of consensus coupled with distributed optimization was independently introduced in  \cite{Lorenzo2015NEXT-CAMSAP15,Lorenzo2015NEXT-ICASSP16,LorenzoScutari-J'16} (NEXT     framework)  for     constrained, nonsmooth, nonconvex instances of \eqref{eq: P} over time-varying graphs and in \cite{Xu2015augmented}    for the  case of strongly convex, unconstrained, smooth optimization over static undirected graphs.  This   tracking protocol was extended to  arbitrary (time-varying) digraphs (without requiring doubly-stochastic weight matrices) in our conference work  \cite{SunScutariPalomar-C'16}.    
A convergence rate analysis of  the scheme in \cite{Xu2015augmented}   was later developed in \cite{nedich2016achieving,qu2016harnessing}, with   \cite{nedich2016achieving} considering (time-varying)  directed   graphs. We refer the reader to Sec.~\ref{sec:avg_consensus} for a  more detailed discussion on this topic.    %\marginpar{\textcolor{red}{ few sentences to explain our contribution and mention the perturbed consensus push-sum}}

 \vspace{-0.5cm}
\subsection{Summary of  contributions}\label{sec:contributions}\vspace{-0.3cm}
%Prior to this work, an open question was how to device provably distributed algorithms for Problem~\eqref{eq: P}, under challenges i)-v). This paper fills this gap and proposed the first 
 We summarize our major contributions as follows; see also Table~\ref{nonconvex_table}.  \vspace{-0.1cm}
\begin{enumerate}
\item \textbf{Novel algorithmic framework:}  We propose the first provably convergent distributed algorithmic framework for the general class of    Problem~\eqref{eq: P}, addressing   \emph{all} challenges  (i)-(iv). The proposed  approach hinges on Successive Convex Approximation (SCA) techniques, coupled with a    judiciously designed  perturbed push-sum consensus  mechanism that aims to track  locally the gradient  of $F$. Both communication and tracking  protocols  are implementable on arbitrary \emph{time-varying} undirected or \emph{directed} graphs, and in the latter case only  column-stochasticity of the weight matrices is required. Also,   \emph{feasibility} of the iterates is preserved at each iteration.  Either constant   or diminishing step-size rules can be used  in the same scheme, and  
convergence to  stationary solutions
of Problem~\eqref{eq: P} is established. % under  diminishing or constant step-size rules. 

\item \textbf{Iteration complexity:} We prove that the proposed scheme has sublinear convergence rate as long as the positive step-size is smaller than an explicit upper bound;  different step-sizes   among the agents  can also be used. To the best of our knowledge, this is the first convergence/complexity result of  distributed algorithms employing a constant step-size for  nonconvex (constrained) optimization over (time-varying) digraphs.   
\item \textbf{New  {Lyapunov}-like  function and  {descent} technique:} We improve upon  existing convergence  techniques and introduce new ones. Current analysis of distributed algorithms has trouble  handling \emph{nonconvex, nonsmooth, constrained} optimization. Moreover, in the presence of unbounded (sub-)gradients of the objective function,  {descent} on the objective function while treating  optimization and consensus errors separately no  longer works.   A new  convergence analysis is introduced to overcome this difficulty based on a novel ``{Lyapunov}''-like  function that  properly combines suitably defined weighted average dynamics,  consensus and tracking disagreements. %In addition, we use several other techniques that are tailored to relax our convergence assumptions as much as possible. 
\item   \textbf{Broader class of problems and convergence results:} The proposed algorithmic framework and convergence results are  applicable to a significantly larger class of (constrained)  optimization problems  and network topology  than current distributed schemes, including several instances  arising from machine learning, signal processing, and data analytic applications (cf. Sec.\ref{sec:examples}). Moreover, we contribute to the theory of distributed algorithms also for convex  problems, being our schemes the first able to   provably  deal with either constraints [issue (iii)] or nonsmooth $U$ [issue (ii)] with  nonbounded (sub-)gradient [issue (v)] over (time-varying) digraphs.      
 %issues     iii)
 %and iv); ii) and v); or iii) and v).   
 Finally, our algorithm contains  as special cases several recently   gradient-based algorithms whose convergence was proved under more restrictive assumptions on the optimization problem  and network topology (cf. Sec.~\ref{sec:connection}). \vspace{-0.4cm} \end{enumerate} 
  
Finally, preliminary numerical
results show that the proposed schemes compare favorably with 
%centralized and decentralized 
state-of-the-art algorithms.
% The outline is not required, but we show an example here.

 The rest of the  paper is organized as follows. The problem setting is discussed in 
Sec.~\ref{sec:prob} along with some motivating applications. Some preliminary results, including a perturbed push-sum consensus scheme over time-varying digraphs, 
are introduced in Sec.~\ref{sec:avg_consensus}.  Sec.~\ref{sec:alg} describes the proposed algorithmic framework along with its convergence properties, whose  proofs  are given in Sec.~\ref{sec:convergence proof}. Finally, some numerical results are presented in Sec.~\ref{sec:experiments}.\\ \noindent \textbf{Notation.} The set of nonegative (resp. positive) natural number is denoted by $\mathbb{N}_+$ (resp. $\mathbb{N}_{++}$). A vector $\mathbf{x}$  is viewed as a column vector;   matrices are denoted by bold letters.  We  work with the space $\mathbb{R}^m$, equipped with the standard Euclidean norm, which is denoted by $\|\bullet\|$; when the argument of $\|\bullet\|$ is a matrix, the default norm is the spectral norm.  When some other (vector or matrix) norms are used, such as $\ell_1$-norm, or infinity-norm, we will use the notation  $\|\bullet\|_p$ with the corresponding value of $p$. The transpose of a vector $\mathbf{x}$ is denoted by $\mathbf{x}^\top$. The Kronecker product is denoted by  $\otimes$. We use $\mathbf{1}$ to denote a vector with all entries equal to 1, and $\mathbf{I}$ to denote the identity matrix; With some abuse of notation, the dimensions of $\mathbf{1}$ and $\mathbf{I}$ will not be given explicitly but   understood within the context. Given $I\in \mathbb{N}_{++}$, we define $[I]\triangleq \{1,\dots, I\}$.  \vspace{-0.4cm} %\textcolor{red}{what else?}% while conclusions are drawn in Sec.~\ref{sec:conclusions}.
\vspace{-0.4cm}
 \section{Problem Setup and Motivating Examples}\label{sec:prob}\vspace{-0.2cm}
We study Problem~\eqref{eq: P} under the following assumptions.\vspace{-0.1cm}
\begin{assumption}[On Problem~\eqref{eq: P}]\label{assumption:P}  Given  Problem~\eqref{eq: P}, suppose that\vspace{-0.1cm}
	\begin{enumerate}[leftmargin=.8cm,label=(\theassumption\arabic*)]
		\item[A.1] The set $\mathcal{K}\subseteq\mathbb{R}^m$ is (nonempty) closed and convex;\label{assump:P1}\smallskip
		\item[A.2] Each $f_i:\mathcal{O}\to \mathbb{R}$ is $C^{1}$, where $\mathcal{O}\supseteq \mathcal{K}$ is an open set, and $\nabla f_i$ is $L_{i}$-Lipschitz on $\mathcal{K}$;\label{assump:P2}\smallskip
		\item[A.3]  $G^{+}:\mathcal{K}\to \mathbb{R}$ is convex (possibly nonsmooth), and $G^{-}:\mathcal{O}\to \mathbb{R}$ is $C^{1}$ with  $\nabla G^{-}$ being $L_{G}$-Lipschitz on $\mathcal{K}$;  \label{assump:P3}\smallskip
		\item[A.4]  {$U$ is lower bounded on $\mathcal{K}$.}%coercive on $\mathcal K$, i.e., $\lim_{\mathbf{x}\in\mathcal{K},\,\|\mathbf{x}\|\to \infty} U\left(\mathbf{x}\right)= +\infty$.
	\end{enumerate}
\end{assumption}
%Recall that 
We also made   the blanket assumption that each agent $i$ knows only its on function $f_i$ and the regularizer $G$ but not the functions of the other agents. %(constraints are shared among all the agents). 

Assumptions A.1  A.2 and A.4 are quite standard and satisfied by    several problems of practical interest. We remark that, as a major departure from most of the literature on distributed algorithms, we do not assume that the gradient of $F$ (and $G^-$) is bounded on the feasible set $\mathcal K$. This, together with the nonconvexity of $G$ as stated in A.3, opens the way to design for the first time  distributed algorithms for a gamut of new applications, including several big-data problems in statistical learning; see Sec.~\ref{sec:examples} for details.  
 \smallskip

\noindent\textbf{On the network topology:} %Next, we introduce the network topology through which the agents communicates. 
 Agents communicate through a (possibly) time-varying network. Specifically, time is slotted with $n$ denoting the iteration index (time-slot); in each time-slot $n$, the communication network of agents is  modeled as a  (possibly) time-varying digraph  $\mathcal{G}^{n}=\left([I],\mathcal{E}^{n}\right)$, where  $[I]= \{1,\ldots, I\}$ denotes  the set of agents--the vertices of the graph--and the set of edges $\mathcal{E}^{n}$ represents the agents' communication links; we use $(i,j)\in \mathcal{E}^{n}$ to indicate that the link is directed from node $i$ to node $j$. The  in-neighborhood of agent $i$ at  time   $n$   is defined as  $\mathcal{N}_i^{\rm in}[n]=\{j\,|\,(j,i)\in\mathcal{E}^{n}\}\cup\{i\}$  (we included in the set  node $i$ itself, for notational simplicity); it represents the set of agents  which node $i$ can receive information from. The out-neighborhood  of agent $i$ is  $\mathcal{N}_i^\textrm{out}\left[n\right]=\{j\,|\,\left(i,j\right)\in\mathcal{E}^{n}\}\cup\{i\}$--the set of agents receiving information from node $i$ (including node $i$ itself). The out-degree of agent $i$ is defined as $d_i^{n} \triangleq  \left|\mathcal{N}_i^\textrm{out}\left[n\right]\right|$.  To let information propagate over the network,  we assume that  the graph sequence $\{\mathcal{G}^{n}\}_{n\in\mathbb{N}_+}$ possesses some ``long-term'' connectivity property, as formally stated next.
\begin{assumption}[On graph connectivity]\label{assumption:G}
	The graph sequence $\{\mathcal{G}^{n}\}_{n\in\mathbb{N}_+}$ is $B$-strongly connected, i.e., there exists a finite integer $B > 0$  %(possibly unknown to the agents) 
	such that the graph with edge set  {$\cup_{t=k}^{k + B-1} \mathcal{E}^{t}$} is strongly connected,  for all $k\geq0$.
\end{assumption}

%\del{ We are interested in designing distributed algorithms converging, in the above setting, to d(irectional)-stationary solutions of  Problem~\eqref{eq: P}.
%\begin{definition}
%	Given Problem~\eqref{eq: P},   $\mathbf{x}^\ast\in \mathcal{K}$ is a d-stationary solution of \eqref{eq: P} if \begin{equation}
 %	U'\left(\mathbf{x}^\ast; \mathbf{y}-\mathbf{x}^\ast \right)\geq 0,\quad \forall \mathbf{y}\in \mathcal{K},\label{def:stationarity} 
 %\end{equation}where $U'(\mathbf{x}; \mathbf{d})$ denotes the directional  derivative of $U$ at $\mathbf{x}$ along the direction $\mathbf{d}\in \mathbb{R}^m$.
%\end{definition}}
%\mage{We are interested in designing distributed algorithms converging, in the above setting, to stationary solutions of  Problem~\eqref{eq: P}.}
%To the best of our knowledge, this is among the weakest conditions known in the literature, under which distributed algorithms have been studied. 

We conclude this section discussing some instances of Problem~\eqref{eq: P} in the context of statistical learning. \vspace{-0.5cm}

\subsection{Distributed sparse statistical learning}\label{sec:examples} \vspace{-.2cm}
We consider two distributed nonconvex problems in statistical learning, namely: i) a nonconvex sparse linear regression problem; and ii) the sparse Principal Component Analysis (PCA) problem. \\
\begin{comment}
 In fact,  under A.3, one can incorporate in the objective function not only a wide range of convex   regularizers, such as   $\ell_{p}$ ($p\geq 1$) norms, the $\ell_{1,2}$ norm, and the nuclear norm, but more importantly also a vast class of  nonconvex surrogates of the $\ell_{0}$ cardinality function, including  the SCAD \cite{fan2001variable} \textcolor{red}{[ref]}, the ``transformed'' $\ell_1$ \cite{ZhangXin16}, the logarithmic \cite{weston2003use}, and the exponential \cite{bradley1998feature} \textcolor{red}{[ref]}. It is well documented that such nonconvex regularizers  outperform the widely used $\ell_1$ norm in enhancing the
sparsity of the solution.  Quite interestingly, all the aforementioned  nonconvex surrogates enjoy an equivalent  DC (Difference
of Convex)  structure (whose smooth part does not have in general a bounded gradient on $\mathbb{R}^m$) that satisfies A.3; %, which is captured in Problem~\eqref{eq: P} by a regularizer having exactly the structure of  $G$ in  A.3; 
 see Sec.~\ref{sec:examples} for more  details.
\end{comment}
\noindent\textbf{Nonconvex Sparse Linear Regression.} Consider the problem of retrieving a sparse signal $\mathbf{x}\in \mathbb{R}^m$ from the observations $\{\mathbf{b}_i\}_{i=1}^{I}$, where each $\mathbf{b}_i = \mathbf{A}_i\mathbf{x}$ is a linear measurement of the signal acquired  by agent $i$. A mainstream approach in the literature is to solve the following optimization
problem
\vspace{-0.2cm}
 \begin{equation}
 \begin{aligned}
 &\underset{\mathbf{x}}{\textrm{min}}& & \sum_{i=1}^{I}\norm{\mathbf{b}_i - \mathbf{A}_i\mathbf{x}}^2 + \lambda \cdot G\left(\mathbf{x}\right),
 \end{aligned}\label{p: sparse regression}\vspace{-0.2cm}
 \end{equation}
where  the quadratic term  measures the model fitness whereas the
regularizer $G$ is used to promote sparsity in the solution, and  $\lambda>0$ is chosen to balance the trade-off between the model fitness and
solution sparsity. Problem \eqref{p: sparse regression} is clearly an instance of \eqref{eq: P}. %, with    %$\mathcal{D}_i \triangleq  \{\left(\mathbf{A}_i,\mathbf{b}_i\right)\}$ and 
%$f_i\left(\mathbf{x},\mathcal{D}_i\right) \triangleq \norm{\mathbf{b}_i - \mathbf{A}_i\mathbf{x}}^2$.
 Note   that each agent knows only its own function $f_i$  (since $\mathbf{b}_i$ is own only by agent $i$). Also, $\nabla f_i$ is not bounded on $\mathbb{R}^m$.

To promote sparsity on the solution, the ideal choice for $G$ would be the cardinality of $\mathbf{x}$ (a.k.a. $\ell_0$ ``norm'' of $\mathbf{x}$). However, its combinatorial nature makes the resulting optimization problem
numerically intractable as the variable dimension $m$ becomes large. Several convex and, more recently, also nonconvex surrogates of the $\ell_0$ function have been proposed in the literature. The structure of $G$, as stated in Assumption A.3,  captures  either choices.  For instance,  one can choose as regularizer in \eqref{p: sparse regression}, the  $\ell_2$ or $\ell_1$ norm of $\mathbf{x}$ (and thus $G^-=0$), which leads to  the  ridge and LASSO regression problems, respectively. Moreover, a vast class of  nonconvex surrogates can also be considered, including  the SCAD \cite{fan2001variable}, the ``transformed'' $\ell_1$ \cite{ZhangXin16}, the logarithmic \cite{weston2003use}, and the exponential \cite{bradley1998feature}; see Table~\ref{tab:ncvx_regularizer}.   It is well documented that   nonconvex regularizers  outperform the   $\ell_1$ norm in enhancing solution
sparsity.  Quite interestingly, all the widely used  nonconvex surrogates listed in Table~\ref{tab:ncvx_regularizer} enjoy the following  separable  DC  structure   (see, e.g.,  \cite{le2015dc,pang2016PartI} and references therein)\vspace{-0.2cm}
 \begin{equation}
G(\mathbf{x}) = \sum_{i=1}^m g(x_i),\quad\text{with}\quad g\left(x_i\right)=\underset{\triangleq g^{+}\left(x_i\right)}{\underbrace{\eta\left(\theta\right)\left\vert x_i\right\vert }}-\underset{\triangleq g^{-}\left(x_i\right)}{\underbrace{\big(\eta\left(\theta\right)\left\vert x_i\right\vert -g\left(x_i\right)\big)}},\label{eq: dc decompose}
\end{equation}
  where the   expression of $g:\mathbb{R}\to\mathbb{R}$ is given in Table~\ref{tab:ncvx_regularizer}; and $\eta\left(\theta\right)$ is a fixed given  function,  defined in Table~\ref{tab:gradient} for each of  the surrogate $g$ listed in Table~\ref{tab:ncvx_regularizer}. The parameter $\theta$ controls the tightness of the approximation of the $\ell_0$ function: in fact, it holds that  $\lim_{\theta\to +\infty} g(x_i)=1$ if $x_i\neq 0$, otherwise $\lim_{\theta\to +\infty} g(x_i)=0$. Note that  $g^-$ is convex and has Lipschitz continuous first derivative $dg^-/dx$ \cite{le2015dc}, whose closed form is given in Table~\ref{tab:gradient}. %Also, for the \textcolor{red}{...list which $g^-$ ...}  $dg^-/dx$ is  not bounded on $\mathbb{R}$.  
  \\
 \indent  It is not difficult to check that Problem~\eqref{p: sparse regression}, with any of the regularizers discussed above,  is an instance of \eqref{eq: P} and satisfies Assumption~A. Also, note that the gradient of the smooth part is not bounded on $\mathbb{R}^m$. 

\noindent\textbf{Sparse PCA.} Consider  finding the sparse principal component of a distributed data set given by the rows of a set of matrices $\mathbf{D}_i$'s (each $\mathbf{D}_i$ is own by agent $i$). The problem can be formulated as \vspace{-0.2cm}
 \begin{equation}
 \begin{aligned}
 &\underset{\norm{\mathbf{x}}_2 \leq 1}{\textrm{max}}& & \sum_{i=1}^{I}\norm{\mathbf{D}_i\mathbf{x}}^2 - \lambda \cdot G\left(\mathbf{x}\right),\\
 \end{aligned} \label{p: sPCA}\vspace{-0.2cm}
 \end{equation}
 where $G$ can be any of the sparse-promoting regularizers discussed in the previous example. Clearly, Problem \eqref{p: sPCA} is another (nonconvex) instance of  Problem~\eqref{eq: P} (satisfying Assumption A).%, with  $\mathcal{D}_i \triangleq \{\mathbf{D}_i\}$ and $f_i \triangleq -\norm{\mathbf{D}_i\mathbf{x}}^2$.
 % We are not aware of any  provably convergent distributed algorithm for  Problems \eqref{p: sparse regression} and  \eqref{p: sPCA}. \marginpar{\mage{inaccurate statement}}%
 \vspace{-0.4cm}

  \begin{table}
\setlength{\tabcolsep}{8pt}
\renewcommand{\arraystretch}{2}
	\centering
	\caption{Examples of nonconvex surrogates of the $\ell_0$ function  having a DC structure [cf.\,(\ref{eq: dc decompose})]}
	\label{tab:ncvx_regularizer}
\resizebox{.8\columnwidth}{!}{
	\begin{tabular}{ll}\toprule
		Penalty function & Expression   \\ \midrule
		Exp \cite{bradley1998feature} & $g_{\text{exp}}(x)=1-e^{-\theta |x|}$   \\
		$\ell_p(0<p<1)$ \cite{fu1998penalized} & $g_{\ell_p^{+}}(x)= (|x|+\epsilon)^{1/\theta}$,  \\
		
		$\ell_p(p<0)$ \cite{rao1999affine} & $g_{\ell_p^{-}}(x)= 1-(\theta |x|+1)^{p}$ \smallskip\\

		SCAD \cite{fan2001variable}& $g_{\text{scad}}(x)\!= \!\begin{cases}
		
		\frac{2 \theta}{a+1} |x|,  & 0 \leq |x|\leq \frac{1}{\theta}  \smallskip \\
		
		\frac{-\theta^2 |x|^2+2 a \theta |x|-1}{a^2-1},& \frac{1}{\theta} < |x|\leq \frac{a}{\theta}  \smallskip \\
		
		1, & |x|> \frac{a}{\theta}  \end{cases}\smallskip$  \\
		
		Log \cite{weston2003use} & $g_{\log}(x)=\frac{\log(1 +\theta |x|)}{\log(1 +\theta)}$ \\
		\bottomrule
	\end{tabular}} 
\end{table}

 \begin{table}
 	\centering
 	\setlength{\tabcolsep}{15pt}
 	\renewcommand{\arraystretch}{2}
 		\caption{ Explicit expression of $\eta(\theta)$  and   $dg^{-}/dx$ [cf.\,(\ref{eq: dc decompose})]   }\label{tab:gradient}
\resizebox{.8\columnwidth}{!}{
	\begin{tabular}{ l l  l}
		\toprule	 
		$g$ & $\eta(\theta)$ & {$dg^{-}_{\theta}/dx$ }   \\ \midrule
$g_{\textrm{exp}}$ & $\theta$ & $\text{sign}(x)\cdot\theta\cdot(1-e^{-\theta |x|})$ \\		
$g_{\ell_p^{+}}$ &  $\frac{1}{\theta}\epsilon^{1/\theta-1}$ & $\frac{1}{\theta}\,\text{sign}(x)\cdot[\epsilon^{\frac{1}{\theta}-1}-(|x|+\epsilon)^{\frac{1}{\theta}-1}]$  \\
$g_{\ell_p^{-}}$ & $-p \cdot\theta$  & $-\text{sign}(x)\cdot p\cdot \theta\cdot [1-(1+\theta |x|)^{p-1}]$  \smallskip\\
$g_{\text{scad}}$ & $\frac{2 \theta}{a+1}$  & $  \begin{cases}
0,  & |x|\leq \frac{1}{\theta}   \smallskip \\
\text{sign}(x)\cdot \frac{2 \theta(\theta |x|-1)}{a^2-1}, & \frac{1}{\theta} < |x|\leq \frac{a}{\theta}  \smallskip \\
\text{sign}(x)\cdot\frac{2 \theta}{a+1}, & \text{otherwise}
\end{cases} $ \smallskip\\
$g_{\log}$ &    $\frac{\theta}{\log(1+\theta)}$ & $\text{sign}(x)\cdot \frac{\theta^2 |x|}{\log(1+\theta)(1+\theta |x|)}$  \\
		\bottomrule
	\end{tabular}}
	\end{table}

 \section{Preliminaries: The perturbed condensed push-sum algorithm} \label{sec:avg_consensus}\vspace{-0.2cm}

 The proposed algorithmic framework combines local optimization based on SCA  with  constrained consensus   and tracking of gradient averages over digraphs.  % The consensus protocol needs to be design   protocol whose goal is to  and  one step of  tracking of gradient averages over digraphs. 
% builds on the idea of SCA techniques coupled
%with suitably designed message passing protocols (compatible with the local
%agent knowledge of the network) aiming at disseminating information among the
%nodes as well as locally estimating  from each agent
 %hinges on the idea of tracking the average of agents' gradients  over digraphs. Standard consensus algorithms--such as the renowned   push-sum scheme   \cite{kempe2003gossip}--cannot be directly applied, as they would converge to the average of the initial values of the gradients--they cannot track time-varying signals.  
 %network. Since the gradients of the agents changes over at each iteration, more generally, the problem boils down to tracking the average of time-varying signals over a  problem can be formulated as In this section   we  preliminarily introduce a  perturbed push-sum consensus algorithm along with its convergence properties, which is instrumental to accomplish this task.  %\mage{I'm missing here the connection between this section and the big picture. Readers at this point may not get why we need to study consensus. A high level description of the algorithm idea might be useful.}\vspace{-0.3cm} %our developments. In particular, we discuss a revisited  \emph{constrained, perturbed} average consensus algorithm on  time-varying digraphs, which will be a building block of the proposed distributed   algorithm. 
 %\subsection{The perturbed condensed push-sum algorithm}\vspace{-0.2cm}
The    consensus problem over   graphs has been widely studied in the literature; a renowned  distributed scheme  solving this problem   over (possibly time-varying) digraphs  is  the   so-called push-sum algorithm   \cite{kempe2003gossip}.      A perturbed version of the push-sum scheme has been introduced  in  \cite{nedic2015distributed}   to solve   \emph{unconstrained} optimization problems over (time-varying) digraphs. However, it is not clear how to leverage the push-sum update and extend these optimization schemes to deal with constraints. %when it comes to solve distributed \emph{constrained} optimization problems, it is no longer clear how to leverage the   push-sum idea.  
    In this section,  we introduce a  reformulation of the \emph{perturbed} push-sum protocol \cite{nedic2015distributed}--termed \emph{perturbed condensed } push-sum--that is more suitable for the integration with \emph{constrained} optimization.  This scheme will be then used to  build  the {gradient tracking}   and constrained consensus mechanisms embedded in   the proposed algorithmic framework (cf.~Sec.~\ref{sec:alg}). %\mage{Same comment as before, reader has no idea about what gradient tracking and consensus mean at this point}% In fact, it is not clear how to extend  current works     in the literature \cite{nedic2015distributed,tatarenko2015non} using push-sum in the context of distributed optimization cannot : they all focused on   \emph{unconstrained} problems.   
  %protocol in the form \cite{kempe2003gossip} would lead to distributed schemes whose feasibility of their iterates is not preserved.   We are aware of a few works in the literature \cite{nedic2015distributed,tatarenko2015non} using push-sum in the context of distributed optimization: they all focused on   \emph{unconstrained} problems.   
   
 % when it comes to  solve distributively  \emph{constrained} optimization problems, %over digraphs,  
 %it is not clear how to leverage the push-sum to deal with constraints.

 %In this section, we  provide  a revisited consensus protocol which builds on   the push-sum idea \cite{kempe2003gossip} but enjoys the  property of preserving the feasibility of the iterates  {(when applied to consensus problems, this means that if the initial point of the consensus algorithm belongs to a closed convex set, then all the sequent iterates stay in the set)}.  While our scheme does not add any new feature to the push-sum protocol when applied to solve the average consensus problem, it instead provides a key new building block for the design of  distributed algorithms solving \emph{constrained} optimization problems; see Sec.~\ref{sec:alg}. 
  
 Consider a network of $I$ agents, as introduced in Sec. \ref{sec:prob},  communicating over a time-varying digraph (cf. Assumption~\ref{assumption:G}).   Each agent $i$ controls a vector of variables $\mathbf{x}_{(i)}\in \mathbb{R}^m$ as well as a scalar $\phi_i$ that are  iteratively updated,  based upon  the information received from  its immediate neighbors. Let   $\mathbf{x}_{(i)}^n$ and $\phi_i^n$ denote the values of $\mathbf{x}_{(i)}$  and $\phi_i$  at iteration $n\in\mathbb{N}_+$. We let  agents' updates   be subject to a(n adversarial)  perturbation; we denote by $\error_i^n\in \mathbb{R}^m$  the perturbation injected in the update of agent $i$ at iteration $n$. Given   $\mathbf{x}_{(i)}^n$ and  $\phi_i^n$,  the perturbed condensed consensus algorithm reads:\vspace{-0.2cm} 
  \begin{subequations}\label{eq:push_sum_consensus} 
  \begin{align}
   \phi_i^{n+1} &=\sum_{j=1}^I\,a_{ij}^n\phi_j^n,\label{eq:mixing phi}\\
   \mathbf{x}_{(i)}^{n+1} &=\dfrac{1}{\phi_i^{n+1}}\sum_{j=1}^I\,a_{ij}^n\phi_j^n\mathbf{x}_{(j)}^n+\error_i^{n+1},\label{eq:mixing_x}
   \end{align} 
  \end{subequations}
 for all $n\in \mathbb{N}_+$ and $i\in [I]$, where $\mathbf{x}_{(i)}^0$ are arbitrarily chosen  and $\phi_i^0$ are positive scalars such that $\sum_{i=1}^I \phi_i^0=I$; and    $\mathbf{A}^{n}\triangleq(a_{ij}^n)_{i,j=1}^I$ is a  (possibly) time-varying matrix of weights  whose nonzero pattern is compliant with  the topology of  the graph $\mathcal{G}^n$, in the sense of the assumption below. 
   \begin{assumption}[On the weight matrix $\mathbf{A}^{n}$]\label{assumption:A}
 	Each   $\mathbf{A}^{n}\triangleq (a_{ij}^n)_{i,j=1}^I$ is compliant with $\mathcal{G}^n$, that is,
 	\begin{enumerate}[leftmargin=.8cm]
 		\item[C1.]  $a_{ii}^n\geq \kappa>0$, for all $i\in [I]$;\smallskip
 		\item[C2.] $a_{ij}^n\geq \kappa>0$, if $\left(j,i\right)\in \mathcal{E}^{n}$; and $a_{ij}^n=0$ otherwise.
 		%\item[C3.] $\mathbf{A}^{n}$ is column stochastic, i.e., $\mathbf{1} ^\top\mathbf{A}^{n}= \mathbf{1}^{\top}$.
 	\end{enumerate} 
 \end{assumption}	
Under  Assumption~\ref{assumption:A},  the protocol (\ref{eq:push_sum_consensus}) is implementable in a distributed fashion: each  agent $i$ updates its own variables using  only the information $\phi_j^n\,\mathbf{x}_{(j)}^n$ and $\phi_j^n$ received from its current in-neighbors (and its own). We study  convergence   of (\ref{eq:push_sum_consensus}) under  the following further  (standard) assumption on    $\mathbf{A}^{n}$. 
  \begin{assumption}[Column stochasticity]\label{assumption:A_bis}
 	Each matrix $\mathbf{A}^{n}$   is column stochastic, that is,    $\mathbf{1} ^\top\mathbf{A}^{n}= \mathbf{1}^{\top}.$
 \end{assumption}

The role of the extra variables $\phi_i$ is to  dynamically rebuild  the row stochasticity of the equivalent weight matrix governing variables' updates, which is a key condition to   lock  consensus.   This can be easily seen   rewriting the dynamics (\ref{eq:mixing_x}) in terms of the equivalent weights  $\mathbf{W}^n\triangleq (w_{ij})_{i,j=1}^I$: \vspace{-0.2cm} %is to  dynamically rebuild  row stochastic matrix which is essential to force the consensus of local variable $\mathbf{x}_{(i)}^n$. To be precise, 
%we define $w_{ij}^n$  as 
 \begin{equation}\label{eq:def w}
 \mathbf{x}_{(i)}^{n+1}  = \sum_{j=1}^I w_{ij}^n\, \mathbf{x}_{(j)}^n,\qquad w_{ij}^n \triangleq \dfrac{a_{ij}^n\phi_j^n}{\phi_i^{n+1}}. \vspace{-0.2cm}
 \end{equation}
It is not difficult to check  that, under  Assumption~\ref{assumption:A_bis},   $\mathbf{W}^n$ is   row-stochastic.

To state  the main convergence result in compact form, we introduce the following notation. Let  \vspace{-0.3cm}
 \begin{subequations} \label{eq:x_stack}
  \begin{align} 
   \mathbf{x}^{n}  & \triangleq [\mathbf{x}^{n\,\top}_{(1)}, \ldots, \mathbf{x}^{\,\top}_{(I)}]^\top,\\
   \boldsymbol{\phi}^n  & \triangleq   \,\left[\phi_{1}^n,\ldots,\phi_{I}^n\right]^{\top},\\
  \error^n  & \triangleq   \,[\error_{1}^{n\,\top},\ldots,\error_{I}^{n\,\top}]^{\top}. \end{align}
  \end{subequations} 
Noting that, in the absence of perturbation (i.e., $\error^{n}=\mathbf{0}$), the weighed sum    $\sum_{i=1}^I\phi_i^{n} \mathbf{x}_{(i)}^{n}$ is an invariant of \eqref{eq:push_sum_consensus}, that is, $\sum_{i=1}^I\phi_i^{n+1} \mathbf{x}_{(i)}^{n+1}$ $=\cdots =\sum_{i=1}^I\phi_i^{0} \mathbf{x}_{(i)}^{0}$, we define the consensus disagreement at iteration $n$ as the deviation of each $\mathbf{x}_{(i)}^n$ from the weighted average   $({1}/{I})\sum_{i=1}^I \phi_i^n \mathbf{x}_{(i)}^n$: \vspace{-0.1cm}
\begin{equation}\label{eq:consensus-disagreement}
	 \var{x}{n}\triangleq \mathbf{x}^n-\mathbf{1}\otimes \frac{1}{I}\sum_{i=1}^I \phi_i^n \mathbf{x}_{(i)}^n.\vspace{-0.1cm}
\end{equation}
  The dynamics of the error $\var{x}{n}$ are studied in the following proposition (whose proof is postponed to Sec.~\ref{sec:proof-Prop1}).  
 
 \begin{prop}\label{prop:error_decay}
 	Let $\{\mathcal{G}^{n}\}_{n\in\mathbb{N}_+}$ be a sequence of digraphs satisfying  Assumption~\ref{assumption:G}, and let   $\{(\boldsymbol{\phi}^n, \mathbf{x}^n)\}_{n\in \mathbb{N}_+}$ be the sequence generated by the perturbed condensed push-sum protocol (\ref{eq:push_sum_consensus}), for a given perturbation sequence  $\{\error^n\}_{n\in \mathbb{N}_+}$ and   weight  matrices  $\{\mathbf{A}^{n}\}_{n\in\mathbb{N}_+}$    satisfying Assumptions~\ref{assumption:A}-\ref{assumption:A_bis}.  Then, there hold: 
 	\begin{itemize}
 	\item[(i)] \emph{[\texttt{Bounded $\{\boldsymbol{\phi}^n\}_{n\in \mathbb{N_+}}$}]:}  	 \vspace{-0.1cm}  \begin{equation}\label{consensus-error-bound-pert-cons} 
\begin{aligned} 
    \inf_{n\in\mathbb{N}_+} \min_{i\in [I]} \phi_{i}^n &  \geq \phi_{lb},\qquad  \phi_{lb}\triangleq\kappa^{2(I-1)B},\medskip \\
 \sup_{n\in\mathbb{N}_+} \max_{i\in [I]} \phi_{i}^n &     \leq \phi_{ub},\qquad   \phi_{ub}  \triangleq I - \kappa^{2(I-1)B},
\end{aligned}
\end{equation}   with $B\geq 1$ and $\kappa\in (0,1)$ defined in  Assumption~\ref{assumption:G} and  Assumption~\ref{assumption:A}, respectively; 
\end{itemize}

  	\item[(ii)] \emph{[\texttt{Error decay}]:} For all $n,\,k \in \mathbb{N}_+$,   $n\geq k$, \vspace{-0.1cm}
  	\begin{align}\label{error-decay-perturbed-consensus}
  	\| \var{x}{n  }\| \leq  \lambda^k \,\| \var{x}{n-k}\| + \lambda^t \cdot \sum_{t=0}^{k-1 }\| \error^{n-t} \|,
  	\end{align}
  	where  \vspace{-0.1cm} $$\lambda^t \triangleq \min \Big\{ \sqrt{2}\,I , 2\, c_0\, I\, (\rho)^{\big\lfloor \frac{t }{(I-1)B} \big\rfloor} \Big\}, \vspace{-0.1cm}$$ and \vspace{-0.1cm} %and   $m = \lfloor \frac{t }{(I-1)B} \rfloor$ and 
\begin{equation}\label{definitions_kappa_c0_rho_0}
	  c_{0} \triangleq  2\left(1+\tilde{\kappa}^{-\left(I-1\right)B}\right),\quad 
 	\rho \triangleq  1-\tilde{\kappa}^{\left(I-1\right)B},  \quad \tilde{\kappa} \triangleq \kappa^{2 (I-1) B+1}/I. 
\end{equation}  
Furtheremore, there exists a finite $\bar{B} \in \mathbb{N}_+$ such that $\rhoopt \triangleq 2 c_0 I (\rho)^{\big\lfloor \frac{\bar{B} }{(I-1)B} \big\rfloor} \!<1$.
 \end{prop}

\begin{remark}\label{remark_consensus}  The perturbed consensus algorithm~(\ref{eq:push_sum_consensus}) was mainly designed for  digraphs. However, when the graph is undirected, one can  choose the weight matrix $\mathbf{A}^n$ to be double stochastic and get rid of the auxiliary variables $\boldsymbol{\phi}^n$, just setting in  (\ref{eq:push_sum_consensus})     $\boldsymbol{\phi}^0=\mathbf{1}$. As a consequence,  $\mathbf{\boldsymbol{\phi}}^n \equiv   \mathbf{1} $ and $\mathbf{W}^n \equiv \mathbf{A}^n$, for all $n\in\mathbb{N}_+$. 
% \del{Algorithm~(\ref{eq:push_sum_consensus}) then reduces to the standard average consensus scheme: $\mathbf{x}^{n+1} = \mathbf{W}^n \mathbf{x}^n$. Moreover, using \cite[Lemma 9]{nedic2009distributedaveraging},  \eqref{eq:B bar step consensus} can be tightened as: 
%    	%\begin{equation}\label{eq:B  step consensus}
% 	$\|\widehat{\mathbf{W}}^{n+B-1:n} - \mathbf{J}\|_{2} \leq \sqrt{1-\kappa/\left(2I^{2}\right)},$ for all $n\geq 0.$}
 	 {In this case, using \cite[Lemma 9]{nedic2009distributedaveraging}, the expression of $\lambda^t$  in   Proposition~\ref{prop:error_decay} can be tightened by letting $\lambda^t \triangleq \min \{ 1, (\rho)^{\lfloor  t/ B \rfloor}\}$, with $\rho \triangleq \sqrt{1-\kappa/\left(2I^{2}\right)}$.
 	}\vspace{-0.5cm} 
 \end{remark}

\subsection{Discussion}\label{sec:discuss}\vspace{-0.3cm}
Proposition~\ref{prop:error_decay} provides a unified set of convergence conditions of the   perturbed condensed push-sum scheme that are  applicable to   any given perturbation sequence $\{\error^n\}_{n\in \mathbb{N}_+}$. We discuss next two special cases, namely:  the plain average  consensus problem and the distributed tracking of   time-varying signals. \smallskip

%\begin{enumerate}
%	\item 
\noindent \textbf{1. (Weighted) average consensus:} Setting in (\ref{eq:push_sum_consensus}) $\error^n=\mathbf{0}$, for all $n\in \mathbb{N}_+$, (\ref{eq:push_sum_consensus}) reduces to the plain (condensed) push-sum scheme. 
 whose geometric convergence to the (weighted) average of the initial values, $(1/I)\,\sum_{i=1}^I \phi_i^0\,\mathbf{x}_{(i)}^0$, follows readily from Proposition~\ref{prop:error_decay}. More specifically, using $\sum_{i=1}^I\phi_i^{n+1} \mathbf{x}_{(i)}^{n+1}$ $=\cdots =\sum_{i=1}^I\phi_i^{0} \mathbf{x}_{(i)}^{0}$,   \eqref{error-decay-perturbed-consensus}
 yields  
	\begin{equation}
\begin{aligned}
\norm{\mathbf{x}^{n+1} - \mathbf{1}\otimes \frac{1}{I}\sum_{i=1}^I \phi_i^0 \mathbf{x}_{(i)}^0}   \leq 2\,c_{0}\,I\,(\rho)^{\lfloor \frac{n+1}{(I-1)B}\rfloor }\, \norm{\var{x}{0}},\quad   n\in \mathbb{N}_+.\smallskip
\end{aligned}
\end{equation}
 	Note that, since the weight matrix $\mathbf{W}^n$ in \eqref{eq:def w} is row  stochastic, if the initial values $\mathbf{x}_{(i)}^0$ all belong  to a common set $\mathcal{K}$, then $\mathbf{x}_{(i)}^n\in \mathcal{K}$, for all $n\in \mathbb{N}_{++}$; that is feasibility of the iterates is preserved.  
 
	\noindent \textbf{2. Tracking of time-varying signals' averages:} Consider the problem of tracking distributively the average of time-varying signals. At each iteration $n\in \mathbb{N}_+$, each agent $i$  evaluates (or generates) a signal sample $\mathbf{u}_i^n\in \mathbb{R}^m$ from the (time-varying)  sequence $\{\mathbf{u}_i^n\}_{n\in\mathbb{N}_+}$. The goal is to design a distributed algorithm obeying the communication structure of the graphs $\mathcal G^n$ that tracks the average of   the signals  $\{\mathbf{u}_i^n\}_{n\in\mathbb{N}_+}$, that is, \vspace{-0.3cm}\begin{equation}\label{eq:tacking}
	\lim_{n\to \infty}\left\|\mathbf{x}^n-\mathbf{1}\otimes  \bar{\mathbf{u}}^n\right\|=0,\quad \bar{\mathbf{u}}^n\triangleq \frac{1}{I}\sum_{i=1}^I  \mathbf{u}_{i}^n.
\end{equation}
 
The perturbed condensed push-sum algorithm   (\ref{eq:push_sum_consensus})  can be readily used to accomplish this task by setting \vspace{-0.1cm}
\begin{equation}\label{eq:tracking-perturbation}
\error^{n+1}_i= \dfrac{1}{\phi_i^{n+1}}\,\big(\mathbf{u}_i^{n+1}-\mathbf{u}_i^{n}\big),\quad \quad i\in [I], \,\,n\in\mathbb{N}_+,\end{equation}
and $\mathbf{x}_i^0=\mathbf{u}_i^0$,   $i\in [I]$. Convergence of this scheme is stated next.  
   \begin{cor}
   Let $\{\mathbf{u}_i^n\}_{n\in\mathbb{N}_+}$ be a given sequence such that $\lim_{n\to \infty} \|\mathbf{u}^{n+1}_i-\mathbf{u}_i^{n}\|=0$,  for all $i\in [I]$.	Consider the perturbed condensed push-sum protocol (\ref{eq:push_sum_consensus}), under the assumptions of Proposition~\ref{prop:error_decay}; and set  $\error^{n+1}_i$ as in (\ref{eq:tracking-perturbation})  and $\mathbf{x}_i^0=\mathbf{u}_i^0$,  for all  $i\in [I]$. Then, (\ref{eq:tacking}) holds.  \vspace{-0.1cm}
   \end{cor}
   \begin{proof}
   	The proof follows readily from   Proposition~\ref{prop:error_decay} and the following two facts: i) $(1/I)\sum_{i=1}^I\phi_i^{n+1} \mathbf{x}_{(i)}^{n+1}=\bar{\mathbf{u}}^{n+1}$; and ii) \cite[Lemma~7]{nedic2010constrained} \vspace{-0.2cm}$$\lim_{n\to \infty} \|\error^{n}\|=0\,\,\Rightarrow \lim_{n\to \infty}\sum_{t=0}^{n-1} (\rho)^{\lfloor \frac{t}{(I-1)B}\rfloor }\,\|\error^{n-t}\|=0.\vspace{-0.2cm}$$\hfill$\square $\vspace{-0.4cm}
   \end{proof}
   % \end{enumerate}

\subsection{Proof of Proposition~\ref{prop:error_decay}} \label{sec:proof-Prop1}\vspace{-0.2cm} To prove Proposition~\ref{prop:error_decay}, it is convenient to rewrite the perturbed consensus protocol (\ref{eq:push_sum_consensus}) in a vector-matrix form. To do so, let us introduce the following quantities: given the weight matrix $\mathbf{A}^n$ compliant with $\mathcal G^n$ (cf.~Assumption~\ref{assumption:A}) and $\mathbf{W}^n$ defined in (\ref{eq:def w}), let\vspace{-0.2cm}
 \begin{subequations} 
  \begin{align}\label{def_matrices}
	%\boldsymbol{\phi}^n   \triangleq & \,\left[\phi_{1}^n,\ldots,\phi_{I}^n\right]^{\top},\quad 
	\mathbf{D}_{\boldsymbol{\phi}^n}   & \triangleq \textrm{Diag}\left(\boldsymbol{\phi}^n\right),\\
	\widehat{\mathbf{D}}_{\boldsymbol{\phi}^n}  & \triangleq    \mathbf{D}_{\boldsymbol{\phi}^n} \otimes \mathbf{I}, \\ 
	\widehat{\mathbf{A}}^n   &\triangleq \mathbf{A}^n\otimes \mathbf{I}, \\ 
	\widehat{\mathbf{W}}^n  & \triangleq \mathbf{W}^n\otimes \mathbf{I},
\end{align}
\end{subequations}
where   ${\rm Diag}(\bullet)$ denotes a diagonal matrix whose diagonal entries are the  elements of the vector argument, and $
\mathbf{I}$ is the $m\times m$ identity matrix.  Under the  column stochasticity of $\mathbf{A}^n$,  it is not difficult to check that the following holds: %relationship exists between $\mathbf{W}^n$ and $\mathbf{A}^n$ (and $\widehat{\mathbf{W}}^n$ and $\widehat{\mathbf{A}}^n$):
\begin{equation}\label{eq:def W}
 		 \mathbf{W}^{n}=\left(\mathbf{D}_{\boldsymbol{\phi}^{n+1}}\right)^{-1}\mathbf{A}^{n}\,\mathbf{D}_{\boldsymbol{\phi}^n}\quad \text{and}\quad %\label{eq:def W}\\
 		 \widehat{\mathbf{W}}^{n}  =  \left(\widehat{\mathbf{D}}_{\boldsymbol{\phi}^{n+1}}\right)^{-1}\widehat{\mathbf{A}}^{n}\,\widehat{\mathbf{D}}_{\boldsymbol{\phi}^{n}}.%\label{eq: A-W} 
 	\end{equation}

Using the above notation and (\ref{eq:x_stack}), the perturbed push-sum protocol   \eqref{eq:push_sum_consensus}  can be rewritten in matrix-vector  form as
\begin{equation}\label{eq:consensus_mat_form}
\begin{aligned}
\boldsymbol{\phi}^{n+1}   = \mathbf{A}^n \boldsymbol{\phi}^n\quad \text{and}\quad
\mathbf{x}^{n+1}  = \widehat{\mathbf{W}}^n \mathbf{x}^n + \error^{n+1}.
\end{aligned}
\end{equation}
To study convergence of  \eqref{eq:consensus_mat_form}, it is convenient to introduce  the following matrix products: given $n, k\in \mathbb{N}_+$, with $n\geq k$,  
\begin{equation}
\begin{aligned}
\mathbf{A}^{n:k}  \triangleq &
\begin{cases}
\mathbf{A}^n\mathbf{A}^{n-1}\cdots \mathbf{A}^t, & \text{if } n>k,\\
\mathbf{A}^n, &\text{if } n=k,
\end{cases}\\
\mathbf{W}^{n:t}  \triangleq &
\begin{cases}
\mathbf{W}^n\mathbf{W}^{n-1}\cdots \mathbf{W}^k, & \text{if } n>k,\\
\mathbf{W}^n, &\text{if } n=k,
\end{cases}
\end{aligned}\vspace{-0.2cm}
\end{equation}
and 
\begin{equation}
\begin{aligned}
\widehat{\mathbf{A}}^{n:k} \triangleq \mathbf{A}^{n:k} \otimes \mathbf{I},\quad
\widehat{\mathbf{W}}^{n:k} \triangleq \mathbf{W}^{n:k} \otimes \mathbf{I}.
\end{aligned}
\end{equation}

   %We prove next that system \eqref{eq:consensus_mat_form} geometrically reaches a consensus. %The first question to address is on the value of such a consensus.  It is not difficult to check that 
 %Since the weighted  sum $\sum_{i=1}^I\phi_i^{n} \mathbf{x}_{(i)}^{n}$ is an invariant of \eqref{eq:consensus_mat_form}, that is, $\sum_{i=1}^I\phi_i^{n+1} \mathbf{x}_{(i)}^{n+1}$ $=\sum_{i=1}^I\phi_i^{n} \mathbf{x}_{(i)}^{n}=\cdots =\sum_{i=1}^I\phi_i^{0} \mathbf{x}_{(i)}^{0}$,  if a consensus is reached for \eqref{eq:consensus_mat_form}, it must be  $\lim_{n\to\infty}\|\mathbf{x}^n_{(i)} - (1/I)\,\sum_{j=1}^I\phi_j^{0} \mathbf{x}_{(j)}^{0}\|=0$, for all $i=1,\ldots, I$, which is proved next. 
 Define the   weight-averaging matrix 
  \begin{equation}
  	\mathbf{J}_{\boldsymbol{\phi}^n} \triangleq \dfrac{1}{I} \left(\mathbf{1} \,(\boldsymbol{\phi}^{n})^\top\right)\otimes \mathbf{I},%\quad \text{with}\quad \mathbf{J}\triangleq \mathbf{J}_{\mathbf{1}_I},
  \end{equation}
so that 
 $\mathbf{J}_{\boldsymbol{\phi}^n}\,\mathbf{x}^n = \mathbf{1}\otimes \frac{1}{I}\sum_{i=1}^I \phi_i^n \mathbf{x}_{(i)}^n$.
Also, it is not difficult to check the following chain of equalities hold among   $\mathbf{J}_{\boldsymbol{\phi}}^{n}$, $\widehat{\mathbf{W}}^{n:t}$, and  $\widehat{\mathbf{A}}^{n:t}$:  for   $n, k\in \mathbb{N}_+$, with $n\geq k$,  % can be verified that the following equalities  involving  $\mathbf{J}_{\boldsymbol{\phi}}^{n}$, $\widehat{\mathbf{W}}^{n:t}$, and  $\widehat{\mathbf{A}}^{n:t}$ hold:
\begin{equation}\label{eq: JW}
 	\mathbf{J}_{\boldsymbol{\phi}^{n+1}}\,\widehat{\mathbf{W}}^{n:k} \overset{(a)}{=} 
 	% =\mathbf{J}_\mathbf{1}\widehat{\mathbf{A}}^{n:k}\,\widehat{\mathbf{D}}_{\boldsymbol{\phi}^{k}}
 	 \mathbf{J}_\mathbf{1} \widehat{\mathbf{D}}_{\boldsymbol{\phi}^{k}}
   = 	\mathbf{J}_{\boldsymbol{\phi}^{k} } \overset{(b)}{=}  \widehat{\mathbf{W}}^{n:k}\, \mathbf{J}_{\boldsymbol{\phi}^{k}},
\end{equation}
where in (a) we used the definition of $\widehat{\mathbf{W}}^n$ [cf.~(\ref{eq:def W})], $\mathbf{J}_{\boldsymbol{\phi}^{n+1}}$ [cd.~(\ref{eq: JW})], and the column stochasticity of $\widehat{\mathbf{A}}^n$; and (b)   is due to the row stochasticity of $\widehat{\mathbf{W}}^{n:k}$.

%Moreover, the invariance of the weighted  average $(1/I)\cdot\sum_{i=1}^I\phi_i^{n} \mathbf{x}_{(i)}^{n}$ reads [cf.~\eqref{eq:consensus_mat_form} and \eqref{eq: JW}]\vspace{-0.4cm}
% \begin{equation}
% \mathbf{J}_{\boldsymbol{\phi}^{n+1}}\mathbf{x}^{n+1} = \mathbf{J}_{\boldsymbol{\phi}^{n+1}}\,\widehat{\mathbf{W}}^{n:t}\,\mathbf{x}^t =  \mathbf{J}_{\boldsymbol{\phi}^{t}}\,\mathbf{x}^{t}. 
% \end{equation}
%That is, the weighted sum $\sum_{i=1}^I\phi_i^{n} \mathbf{x}_{(i)}^{n}$ is time invariant.

%To show that the $\mathbf{x}_{(i)}^n$s are consensual, it   is then sufficient to prove  \vspace{-0.2cm}
%\begin{equation}
%\begin{aligned}
%\lim_{n\to\infty}\norm{\mathbf{x}^n - \mathbf{J}_{\boldsymbol{\phi}^n}\,\mathbf{x}^n}_2  = \lim_{n\to\infty}\norm{\left(\widehat{\mathbf{W}}^{n:0} - \mathbf{J}_{\boldsymbol{\phi}^0}\right)\mathbf{x}^0}_2=0.\end{aligned}
%\end{equation}

\noindent The consensus error $\var{x}{n}$ in \eqref{eq:consensus-disagreement} can be rewritten as $\mathbf{e}_{x}^n=(\mathbf{I} -\mathbf{J}_{\boldsymbol{\phi}^n}) \mathbf{x}^n$.

%\mage{To study the evolution of $\var{x}{n}$, we apply  the $x$-update recursively and  obtain 
%  \begin{equation}\label{eq:x_recursion}
%\mathbf{x}^{n+\bar{B}} =       \widehat{\mathbf{W}}^{n+\bar{B} - 1:n}\,\mathbf{x}^{n} + \sum_{t=1}^{\bar{B}-1} \widehat{\mathbf{W}}^{n+\bar{B} - 1:n+t} \error^{n+t} +  \error^{n+\bar{B}}.
%\end{equation}
%Using (\ref{eq: JW}) and \eqref{eq:x_recursion}, the weighted average   $\Vwavg{x}{n+\bar{B}} $ can be written as \begin{equation}\label{eq:avg_dynamic}
%\Vwavg{x}{n+\bar{B}}  =    \Vwavg{x}{n}+ \sum_{t=1}^{\bar{B}-1} \Vwavg{\error}{n+t}+ \Vwavg{\error}{n+\bar{B}}.
%\end{equation}
%
%Subtracting \eqref{eq:avg_dynamic}  from   \eqref{eq:x_recursion}   and  using  $\big(\widehat{\mathbf{W}}^{n+\bar{B} - 1:n} - \mathbf{J}_{\boldsymbol{\phi}^{n}}\big) \mathbf{J}_{\boldsymbol{\phi}^{n}} = \mathbf{0}$ [cf.~(\ref{eq: JW})], we can bound the consensus error $\var{x}{n+\bar{B}}$ as 
% \begin{equation}\begin{aligned}  
%  		\|\mathbf{e}_x^{n+\bar{B}}\| &\leq  \Big\| \widehat{\mathbf{W}}^{n+\bar{B} - 1:n} - \mathbf{J}_{\boldsymbol{\phi}^{n}} \Big\|\,\|\mathbf{e}_x^{n-k}\| + \sum_{t=1}^{\bar{B}-1} \Big\|\widehat{\mathbf{W}}^{n+\bar{B} - 1:n+t}  -\mathbf{J}_{\boldsymbol{\phi}^{n+t}}\Big\|\,\|\mathbf{e}_x^{n-t}\|  \\&\quad + \Big\|{\mathbf{I}}  -\mathbf{J}_{\pmb{\phi}^{n+1}}\Big\|\, \|\error^{n+1}\|.
%  	\end{aligned}\label{eq:upper-bound-error_proof}  \end{equation}
%}

{To study the evolution of $\var{x}{n}$, we apply  the $x$-update \eqref{eq:consensus_mat_form} recursively and  obtain 
  \begin{equation}\label{eq:x_recursion}
\mathbf{x}^{n} =       \widehat{\mathbf{W}}^{n-1:n-k}\,\mathbf{x}^{n-k} + \sum_{t=1}^{k-1} \widehat{\mathbf{W}}^{n-1:n-t} \error^{n-t} +  \error^{n}.
\end{equation}
Using (\ref{eq: JW}) and \eqref{eq:x_recursion}, the weighted average   $\Vwavg{x}{n} $ can be written as \begin{equation}\label{eq:avg_dynamic}
\Vwavg{x}{n}  =    \Vwavg{x}{n-k}+ \sum_{t=1}^{k-1} \Vwavg{\error}{n-t}+ \Vwavg{\error}{n}.
\end{equation}

Subtracting \eqref{eq:avg_dynamic}  from   \eqref{eq:x_recursion}   and  using  $\big(\widehat{\mathbf{W}}^{n-1:n-k} - \mathbf{J}_{\boldsymbol{\phi}^{n-k}}\big) \mathbf{J}_{\boldsymbol{\phi}^{n-k}} = \mathbf{0}$ [cf.~(\ref{eq: JW})], we can bound the consensus error $\var{x}{n+1}$ as 
 \begin{equation}\begin{aligned}  
  		\|\mathbf{e}_x^{n}\| &\leq  \Big\| \widehat{\mathbf{W}}^{n-1:n-k} - \mathbf{J}_{\boldsymbol{\phi}^{n-k}} \Big\|\,\|\mathbf{e}_x^{n-k}\| + \sum_{t=1}^{k-1} \Big\|\widehat{\mathbf{W}}^{n-1:n-t}  -\mathbf{J}_{\boldsymbol{\phi}^{n-t}}\Big\|\,\|\error^{n-t}\|  \\&\quad + \Big\|{\mathbf{I}}  -\mathbf{J}_{\pmb{\phi}^{n}}\Big\|\, \|\error^{n}\|.
  	\end{aligned}\label{eq:upper-bound-error_proof}  \end{equation}
}
	Convergence of the perturbed consensus protocol reduces to studying the  dynamics of the matrix product  $\|\widehat{\mathbf{W}}^{n:k}- \mathbf{J}_{\boldsymbol{\phi}^k}\|$, as done in  the lemma below.    	
  \begin{lem}\label{cor:averaging matrix column stochastic}
Let $\{\mathcal{G}^{n}\}_{n\in\mathbb{N}_+}$ be a sequence of digraphs satisfying  Assumption~\ref{assumption:G}; let  $\{\mathbf{A}^{n}\}_{n\in\mathbb{N}_+}$ be  a sequence of weight  matrices  satisfying Assumptions~\ref{assumption:A}-\ref{assumption:A_bis};   and let  $\{\mathbf{W}^{n}\}_{n\in\mathbb{N}_+}$ be the sequence of row stochastic matrices related to  $\{\mathbf{A}^{n}\}_{n\in\mathbb{N}_+}$ by   \eqref{eq:def W}. There holds:   \vspace{-0.1cm} 	\begin{equation}\label{eq:averaging matrix column stochastic}
 	\norm{\widehat{\mathbf{W}}^{n:k} - \mathbf{J}_{\boldsymbol{\phi}^{k}}} \leq \min \Big\{ \sqrt{2}\,I, 2\,c_0 I (\rho)^{\big\lfloor \frac{n-k +1}{ (I-1)B} \big\rfloor} \Big\} ,\quad n,\,k\in \mathbb{N}_+,\,n\geq k,\vspace{-0.1cm}
 	\end{equation}
 	where  $c_0$ and $\rho$ are defined in Proposition~\ref{prop:error_decay}. 
 \end{lem}
 \begin{proof}
 	See Appendix~\ref{app:proof_lemma_1}.\hfill $\square$
 \end{proof}
  	The error decay law \eqref{error-decay-perturbed-consensus} comes readily from \eqref{eq:upper-bound-error_proof}, Lemma~\ref{cor:averaging matrix column stochastic}, and  the following fact: $\big\|{\mathbf{I}}  -\mathbf{J}_{\boldsymbol{\phi}^{n}}\big\|\leq \sqrt{2\,I} \leq \lambda^0 \triangleq \min \{ 2 c_0 I, \sqrt{2} I\}$, which is proved below.    Let $\mathbf{z}\in \mathbb{R}^{I\cdot m}$ be an arbitrary vector;  let us  partition $\mathbf{z}$ as $\mathbf{z}=[\mathbf{z}_1^\top,\ldots, \mathbf{z}_I^\top]^\top$, with each  $\mathbf{z}_i\in \mathbb{R}^{m}$.
Then, 
\begin{align}\label{eq:I_minus_J}
\begin{split}
\norm{(\mathbf{I} - \mathbf{J}_{\boldsymbol{\phi}^{n}})\,\mathbf{z}}
 \leq & \,\norm{\mathbf{z} - \mathbf{J}_{\mathbf{1}}\,\mathbf{z}} + \norm{\mathbf{J}_{\mathbf{1}}\mathbf{z} - \mathbf{J}_{\boldsymbol{\phi}^{n}}\mathbf{z}}
 \stackrel{(a)}{\leq}  \norm{\mathbf{z}} + \frac{\sqrt{I}}{I}\norm{\sum_{i=1}^I \mathbf{z}_i - \sum_{i=1}^I \phi_i^{n} \mathbf{z}_i}\\
 %\leq & \norm{\mathbf{z}} + \frac{\sqrt{I}}{I} \, \sqrt{\sum_{i=1}^I\big(1-\phi_i^k\big)^2}\,\norm{\mathbf{z}} 
%\stackrel{(b)}{\leq} 
\leq & \norm{\mathbf{z}} + \frac{\sqrt{I}}{I} \, \sqrt{I^2 - I} \,\norm{\mathbf{z}} \leq \sqrt{2\,I} \,\norm{\mathbf{z}},
\end{split}
\end{align}
 where in (a) we used $\|\mathbf{I} - \mathbf{J}_{\mathbf{1}}\| = 1$.    \hfill $\square$\vspace{-0.2cm}

\section{Algorithmic Design} 
\label{sec:alg} \vspace{-0.2cm}

We are  ready to introduce the proposed distributed algorithm for Problem~\eqref{eq: P}. To shed light on the core idea of the novel framework, we begin introducing an informal and constructive description of the  algorithm (cf.~Sec.~\ref{SONATA-informal}), followed by its formal statement along with its convergence properties (cf.~Sec.~\ref{SONATA_formal}).\vspace{-0.4cm}

\subsection{SONATA at-a-glance}\label{SONATA-informal}\vspace{-0.2cm}  
%Each agent $i$ maintains  a  copy of the common variables $\mathbf{x}$, denoted by $\mathbf{x}_{(i)}$, and updates $\mathbf{x}_{(i)}$ iteratively by performing a local optimization step  (Step 1 below) based on SCA, followed by a communication step (Step 2 below) leveraging the protocol introduced in Sec.~\ref{sec:avg_consensus}. The aforementioned two steps are designed so that   each  $\mathbf{x}_{(i)}$   reaches asymptotically   a d-stationary solution  of Problem~\eqref{eq: P} while being consensual with  the other agents' local variables.  

%......
Each agent $i$ maintains and updates iteratively a    {\it local copy} $\mathbf{x}_{(i)}$ of the global variable $\mathbf{x}$, along with an auxiliary variable $\mathbf{y}_{(i)}\in \mathbb{R}^m$;  let $\mathbf{x}_{(i)}^n$ and $\mathbf{y}_{(i)}^n$ denote the values of $\mathbf{x}_{(i)}$ and $\mathbf{y}_{(i)}$ at iteration $n$, respectively. Roughly speaking, 
%whose goal is to track locally the average of the gradients $(1/I)\cdot \sum_{i=1}^I \nabla f_i$, an information that  is not available at the agent's side; let $\mathbf{x}_{(i)}^n$ and $\mathbf{y}_{(i)}^n$ denote the values of $\mathbf{x}_{(i)}$ and $\mathbf{y}_{(i)}$ at iteration $n$, respectively. 
the  update of these variables is designed so that    all the $\mathbf{x}_{(i)}^n$ will be asymptotically consensual, converging to a stationary solution of \eqref{eq: P}; and   each  $\mathbf{y}_{(i)}$   tracks locally the average of the gradients $(1/I)\cdot \sum_{i=1}^I \nabla f_i$, an information that  is not available at the agent's side. More specifically, the
following two steps are performed iteratively and in parallel
across the agents.

\smallskip

\noindent\textbf{Step 1: Local SCA.} The nonconvexity  of $f_i$ together with the lack of knowledge of  $\sum_{j\neq i}f_j$ in $F$,
 prevent agent $i$ to solve Problem~\eqref{eq: P} directly. To cope with
these issues, we leverage SCA techniques: at each iteration  $n$,  given the current iterate $\mathbf{x}_{(i)}^n$ and $\mathbf{y}_{(i)}^n$, 
   agent $i$ solves instead    a convexification of  \eqref{eq: P}, having the following form: \vspace{-0.2cm} \begin{equation}
 \widetilde{\mathbf{x}}_{(i)}^n\triangleq  \argmin_{\mathbf{x}_{(i)}\in \mathcal{K}} \,\widetilde{F}_i\left(\mathbf{x}_{(i)};\mathbf{x}_{(i)}^n,\mathbf{y}_{(i)}^n\right) +  G^+\left(\mathbf{x}_{(i)}\right), \label{eq: x_tilde}
 \end{equation}
and updates its $\mathbf{x}_{(i)}$ according to \vspace{-0.2cm}
 \begin{equation}\label{wi}
 \mathbf{x}_{(i)}^{n+1/2} =\mathbf{x}_{(i)}^n + \alpha^n\left(\widetilde{\mathbf{x}}_{(i)}^n- \mathbf{x}_{(i)}^n\right),\vspace{-0.2cm}
  \end{equation}
 where  $\alpha^{n}\in \left(0,1\right)$ is a step-size   (to be properly chosen). In \eqref{eq: x_tilde},  $\widetilde{F}_i(\bullet; \mathbf{x}_{(i)}^n,\mathbf{y}_{(i)}^n)$ is 
 chosen as:
\begin{equation}\label{surrogate_F}
	\begin{aligned}
\widetilde{F}_i\big(\mathbf{x}_{(i)};\mathbf{x}_{(i)}^n,\mathbf{y}_{(i)}^n\big) &\triangleq \widetilde{f_i}\big(\mathbf{x}_{(i)};\mathbf{x}_{(i)}^n\big) - \nabla G^{-}\big(\mathbf{x}_{(i)}^{n}\big)^{\top}\big(\mathbf{x}_{(i)} - \mathbf{x}_{(i)}^n\big)	 \smallskip \\&\quad + \big(I\cdot \mathbf{y}_{(i)}^{n}-\nabla f_i\big(\mathbf{x}_{(i)}^n\big)\big)^\top  \big(\mathbf{x}_{(i)} - \mathbf{x}_{(i)}^n\big),\end{aligned}
\end{equation}
where   $\widetilde{f}_i(\bullet;\mathbf{x}_{(i)}^n)$ is a strongly
convex approximation  of $f_i$ at the current iterate $\mathbf{x}_{(i)}^n$ (see Assumption \ref{assumption:f tilde} below); the second term is the linearization of the  smooth nonconvex function $-G^-$; and $\mathbf{y}_{(i)}^n$, as anticipated,      aims at  tracking  the gradient average $(1/I)\,\sum_{j=1}^I\nabla f_j(\mathbf{x}_{(i)}^n)$, that is,    $\lim_{n\to \infty}\|\mathbf{y}_{(i)}^n-$ $(1/I)\,\sum_{j=1}^I\nabla f_j(\mathbf{x}_{(i)}^n)\|=0$. This sheds light on the role of the last term in (\ref{surrogate_F}): under the claimed tracking properties of $\mathbf{y}_{(i)}^n$, there would hold: 
\begin{equation}\label{eq:asyntotic_track}
	\lim_{n\to \infty}\Big\|   
	\big(I\cdot \mathbf{y}_{(i)}^{n}-\nabla f_i\big(\mathbf{x}_{(i)}^n\big)\big)- \sum_{j\neq i}\nabla f_j(\mathbf{x}_{(i)}^n) 	\Big\|=0.
\end{equation}
Therefore, the last term in (\ref{surrogate_F}) can be seen as a  proxy of   the   gradient sum $\sum_{j\neq i}\nabla f_j(\mathbf{x}_{(i)}^n)$, which is not available at  agent $i$'s site.  
Building on the perturbed condensed push-sum protocol introduced in Sec.~\ref{sec:avg_consensus}  we will show   in Step 2 below how to update  $\mathbf{y}_{(i)}^n$ so that  (\ref{eq:asyntotic_track}) holds, using only  \emph{local} information.
\\ \indent The surrogate function $\widetilde{f}_i$ satisfies the following assumption.  \vspace{-0.1cm}
  
  \begin{assumption}[On surrogate function $\widetilde{f}_i$]\label{assumption:f tilde} Let $\widetilde{f}_i:{\mathcal K}\times {\mathcal K}\rightarrow \mathbb{R}$ be a $C^1$ function with respect to its first argument, and such that 
\begin{enumerate}[leftmargin=.8cm,label=(\theassumption\arabic*)]
	\item[D1.] $\nabla \widetilde{f}_i\left(\mathbf{x};\mathbf{x}\right)=\nabla f_i\left(\mathbf{x}\right)$, for all $\mathbf{x}\in \mathcal{K}$;
	\item[D2.] $\widetilde{f}_i\left(\bullet;\mathbf{y}\right)$ is uniformly strongly convex  on $\mathcal{K}$, with constant $\tau_i$;
	\item[D3.] $\nabla \widetilde{f}_{i}\left(\mathbf{x};\bullet\right)$ is uniformly Lipschitz continuous on $\mathcal{K}$, with constant $\tilde{L}_{i}$;
\end{enumerate}
where $\nabla \widetilde{f}_i(\mathbf{x};\mathbf{y})$ denotes the partial gradient of $\widetilde{f}_i$ with respect to the first argument, evaluated at $(\mathbf{x},\mathbf{y})$.\vspace{-0.2cm}
\end{assumption}

Conditions D1-D3 are quite natural: $\widetilde{f}_i$  should be regarded as
a (simple) convex, local, approximation of $f_i$ at $\mathbf{x}$ that
preserves the first order properties of $f_i$.  A gamut  of choices for $\widetilde{f}_i$ satisfying Assumption \ref{assumption:f tilde} are available; some representative examples are discussed in Sec.~\ref{sec:discussion}.

%To complete the picture, we introduce next the information mixing step based on  , aiming at guarantee that $\widetilde{\boldsymbol{\pi}}_i[n]$ tracks the missing gradient and the $\mathbf{x}_{(i)}^n$'s are consensual asymptotically.
\vspace{.2cm}

\noindent\textbf{Step 2: Information mixing and gradient tracking.}  To complete the description of the algorithm,  we need   to introduce  
a mechanism to ensure that i) the local estimates  $\mathbf{x}_{(i)}^n$'s  asymptotically converge to a common value; and ii) each  $\mathbf{y}_{(i)}^n$  tracks  the gradient sum  $\sum_{j\neq i}\nabla f_j(\mathbf{x}_{(i)}^n)$. To this end, we leverage the perturbed condensed push-sum protocol introduced in Sec.~\ref{sec:avg_consensus}. Specifically, given $\mathbf{x}_{(j)}^{n+1/2}$'s, each $\mathbf{x}_{(i)}$ is updated according to [cf.~(\ref{eq:push_sum_consensus})]\vspace{-0.2cm}
\begin{equation}\label{eq:mixing x} 
  \phi_i^{n+1} =\sum_{j=1}^I a_{ij}^n\phi_j^n, \qquad 
   	\mathbf{x}_{(i)}^{n+1} = \sum_{j=1}^I a_{ij}^n \mathbf{x}_{(j)}^{n+1/2},\vspace{-0.2cm}
\end{equation}
where the $a_{ij}^n$ are chosen to satisfy Assumption~\ref{assumption:A}.  
Note that, the updates in \eqref{eq:mixing x} can be performed in a distributed way: each agent $j$ only needs to select the set of weights $\{a_{ij}^n\}_{i=1}^I$ and send $a_{ij}^n\phi_j^n$ and $a_{ij}^n\phi_j^n\mathbf{x}_{(j)}^{n+1/2}$ to its out-neighbors while  summing up the information received from its in-neighbors.

 To update the $\mathbf{y}_{(i)}^n$'s we leverage again the perturbed condensed push-sum scheme  (\ref{eq:push_sum_consensus}), with     with $\boldsymbol{\epsilon}_i^{n+1}=(1/\phi_i^{n+1})\,\big(\nabla f_i(\mathbf{x}_{(i)}^{n+1}) - \nabla f_i(\mathbf{x}_{(i)}^{n})\big)$ [cf.~(\ref{eq:tracking-perturbation})].   The resulting gradient tracking mechanism reads \vspace{-0.3cm}
\begin{equation}\label{eq: update y}
	\mathbf{y}_{(i)}^{n+1}= \dfrac{1}{\phi_i^{n+1}}  \sum_{j=1}^I\,a_{ij}^n\phi_j^n  \,\mathbf{y}_{(j)}^n +\dfrac{1}{\phi_i^{n+1}}  \left (\nabla f_i\big(\mathbf{x}_{(i)}^{n+1}\big) - \nabla f_i\big(\mathbf{x}_{(i)}^{n}\big)\right), 
\end{equation}
with   $\mathbf{y}_{(i)}^0 = \nabla f_i(\mathbf{x}_{(i)}^0)$. %Given $\mathbf{y}_{(i)}^n$, the update of $\widetilde{\boldsymbol{\pi}}_{(i)}^{n}$ reads
	%\begin{align}
%\widetilde{\boldsymbol{\pi}}_{i}^{n+1} = I \cdot \mathbf{y}_{(i)}^{n+1} - \nabla f_i\left(\mathbf{x}_{(i)}^{n+1}\right).\label{eq: update pi}
	% \end{align}
Note that the update of $\mathbf{y}_{(i)}^n$    can be   performed locally by agent $i$, with the same
signaling of that of  \eqref{eq:mixing x}. \vspace{-0.3cm}	%The intuition behind the update of local decision variable $\mathbf{x}_i$ and  tracking vector $\mathbf{y}_i$ is to construct on-the-fly a row stochastic  matrix using $\phi_i$, thus forcing consensus among the agents. Moreover, the column stochasticity of $\mathbf{A}[n]$ ensures $\frac{1}{I}\sum_{i=1}^{I}\phi_i[n]\mathbf{y}_i[n] = \overline{\nabla f_i}[n]$ and $\sum_{i=1}^{I}\phi_i[n] = I$. 
	%One can show that if $\mathbf{x}_{(i)}^n$'s and $\mathbf{y}_{(i)}^n$'s are consensual (a fact that is proved in Theorem~\ref{thm:convergence}),  $\|\widetilde{\boldsymbol{\pi}}_i^n- \sum_{j\neq i}\nabla f_j\left(\mathbf{x}_{(i)}^n\right)\|\underset{n\rightarrow\infty}{\longrightarrow}0$. 
	
	\subsection{The SONATA algorithm}\label{SONATA_formal}\vspace{-0.4cm}
	 We can now formally introduce the proposed algorithm, SONATA, just combining   steps \eqref{eq: x_tilde},\eqref{wi},\eqref{eq:mixing x}, and  \eqref{eq: update y}--see   Algorithm 
	 \ref{alg:SONATA}.\vspace{-0.2cm} % Recall that $\widehat{\mathbf{W}}^n$ is related to $\mathbf{A}^n$ by \eqref{eq:def W}.
{\fontsize{10}{10}\selectfont
\begin{algorithm}[h]	
	\caption{SONATA}\label{alg:SONATA}
	\textbf{Data}: $\mathbf{x}^{0}_{(i)}\in \mathcal{K}$, for all $i$; $\boldsymbol{\phi}^{0}=\mathbf{1}$; $\mathbf{y}^{0}=\mathbf{g}^0$. Set $n=0$.
	\vspace{0.2cm}
	
	\texttt{[S.1]} If $\mathbf{x}^n$	satisfies termination criterion: STOP;\vspace{0.1cm}\\
	\texttt{[S.2] [Distributed Local Optimization]} Each agent $i$ \medskip\\ \phantom{\texttt{[S.2]}}\quad  Compute locally 	$\widetilde{\mathbf{x}}_{(i)}^n$ solving    problem~\eqref{eq: x_tilde}; \smallskip\\ 
	\phantom{\texttt{[S.2]}}\quad Update its local variable $\mathbf{x}_{(i)}^{n+1/2}   \triangleq \mathbf{x}_{(i)}^n + \alpha^n (\widetilde{\mathbf{x}}_{(i)}^n - \mathbf{x}_{(i)}^n) $;\vspace{0.1cm}

	 \texttt{[S.3] [Information Mixing]} Each agent $i$ compute \smallskip\\
	 \phantom{\texttt{[S.2]}} (a) \texttt{Consensus}  \vspace{-0.3cm}
	 \begin{align}  
 \phi_i^{n+1} & =\sum_{j=1}^I a_{ij}^n\phi_j^n \label{eq: update phi}\\
	\mathbf{x}_{(i)}^{n+1} &= \dfrac{1}{\phi_i^{n+1}}\sum_{j=1}^I  a_{ij}^n\, \phi_j^n \, \mathbf{x}_{(j)}^{n+1/2};
		\label{eq: update x}\end{align}  
	 \phantom{\texttt{[S.2]}} (b) \texttt{Gradient tracking}\vspace{-0.2cm}
	\begin{align}
	\mathbf{y}_{(i)}^{n+1}= \dfrac{1}{\phi_i^{n+1}} \sum_{j=1}^I\,a_{ij}^n\phi_j^n  \,\mathbf{y}_{(j)}^n + \dfrac{1}{\phi_i^{n+1}} \left(\nabla f_i\big(\mathbf{x}_{(i)}^{n+1}\big) - \nabla f_i\big(\mathbf{x}_{(i)}^{n}\big)\right);\label{eq: update y}%\\
	% \widetilde{\boldsymbol{\pi}}^{n+1} &= I \cdot \mathbf{y}^{n+1} -  \mathbf{g}^{n+1};\label{eq: update vec pi}
	 \end{align}
	 \texttt{[S.4] } $n\longleftarrow n+1$, go to \texttt{[S.1] }
\end{algorithm}\vspace{-0.3cm} 
}

Note that the algorithm is  distributed. Indeed, in Step~2, the optimization (\ref{eq: x_tilde}) is performed locally by each agent $i$, computing its own $\widetilde{\mathbf{x}}^{n}_{(i)}$. To do so,   agent $i$ needs to know the current ${\mathbf{x}}^{n}_{(i)}$ and $\mathbf{y}_{(i)}^{n}$, which are both available locally. There are then two consensus steps (Step~3) whereby agents transmit/receive information only to/from their out/in neighbors: one is on the optimization variables ${\mathbf{x}}^{n}_{(i)}$  (and the auxiliary scalars $\phi^n_i$)--see (\ref{eq: update phi})-(\ref{eq: update x})--and one is on the variables  ${\mathbf{y}}^{n}_{(i)}$--see (\ref{eq: update y}). \vspace{-0.3cm}

\subsection{Convergence and complexity analysis of SONATA}\vspace{-0.2cm} To prove  convergence, in addition to  Assumptions \ref{assumption:P}-\ref{assumption:f tilde}, one needs some conditions on the step-size $\alpha^n$. Since line-search methods are not practical in a distributed environment, there are two other options, namely: i) a fixed (sufficiently small) step-size; and ii) a diminishing step-size. We prove convergence using either choices. Recalling the definition of the network parameters $c_0$, $\bar{B}$, $\rhoopt$, $\LBphi$, and $\UBphi$ as given in Proposition~\ref{cor:averaging matrix column stochastic} [see also \eqref{definitions_kappa_c0_rho_0}] and  introducing the problem parameters  [cf.~Assumptions \ref{assumption:P}]\vspace{-0.2cm}
\begin{equation}\label{eq:deinitions_c_L}
\begin{aligned}
L \triangleq \sum_{i=1}^I L_i, \quad \tilde{L}_{\rm{mx}} \triangleq \max\limits_{1\leq i\leq I}\tilde{L}_{i}+L_G,\quad L_{\rm{mx}}\triangleq \max\limits_{1\leq i\leq I}L_{i},\\ c_{\tau}\triangleq \min\limits_{1\leq i \leq I}\tau_{i},\quad c_{L}\triangleq  \left(L\sqrt{I} + L_{\rm{mx}}+\tilde{L}_{\rm{mx}}\right)/I,
\end{aligned}
\end{equation}
%\begin{equation}\label{eq:deinitions_c_L}
%\tilde{L} \triangleq \max\limits_{1\leq i\leq I}\tilde{L}_{i},\quad L\triangleq \max\left\{\max\limits_{1\leq i\leq I}L_{i}, L_G\right\},\quad c_{\tau}\triangleq \min\limits_{1\leq i \leq I}\tau_{i},\quad c_{L}\triangleq L + \frac{L+\tilde{L}}{I},
%\end{equation}  
the step-size can be chosen as follows.
\begin{assumption}\label{assumption:alpha}
	The step-size  $\left\{\alpha^n\right\}_{n\in \mathbb{N}_+}$ satisfies either one of the following conditions:
	\begin{enumerate}[leftmargin=.8cm,label=(\theassumption\arabic*)]
	%	\item $\alpha^n \in \left(0,1\right)$ for all $n \in \mathbb{N}$;\label{assump:E1}
		\item[F1.] \emph{(diminishing):} $(0,1]\ni\alpha^n \downarrow 0$ and $ \sum_{n=0}^{\infty} \alpha^n = \infty$;\label{assump:E3}\smallskip 
		\item[F2.] \emph{(fixed):} \label{assump:E2} $\alpha^n \equiv \alpha$, for all  $n \in \mathbb{N}_+$, with
%		\begin{equation}\hspace{-0.6cm}\label{constant_step}
%		\alpha \leq\min\left\{\frac{\left(1-\rho\right)\sigma}{\sqrt{2}\,c\,\bar{B}},\frac{2c_\tau\LBphi}{I\UBphi}\left(\frac{L+L_G}{I}+2c\left(c_{L}+L_{\rm{mx}}\LBphi^{-1}\right) \frac{\bar{B}}{1-\rho} \sqrt{\frac{2}{1-\sigma^2}}\right)^{-1}\right\},
%		\end{equation}
		\begin{multline}\hspace{-1.1cm}\label{constant_step}
		\alpha \leq\min\left\{\frac{\left(1- \rhoopt \right)\sigma}{\sqrt{2}\,c\,\bar{B}},\right.\\
		\hspace{-1.1cm}\left.\frac{2c_\tau\LBphi}{I\UBphi}\left(\frac{L+L_G}{I}+ \frac{2c_{L}\bar{B}c}{1-\rhoopt} \sqrt{\frac{2}{1-\sigma^2}} + \frac{12 L_{\rm{mx} } \LBphi^{-1}\bar{B}^2 c^2}{(1-\rhoopt)^2}  \sqrt{\frac{1}{1-\sigma^2}} \right)^{-1}\right\},
		\end{multline}
	\noindent where  $\sigma$ is an arbitrary constant $\sigma\in (0,1)$ and  $c=I\,\sqrt{2I}$. \end{enumerate} In addition, if all $\mathbf{A}^n$ are double stochastic, the upper bound in \eqref{constant_step} holds with  $c = 1$, $\bar{B}=B$, $\LBphi = \UBphi = 1$, and $\rhoopt = \big(1-\kappa/(2\,I^{2}) \big)^{1/2}$.   
\end{assumption}

We can now state the convergence results of the proposed algorithm, postponing all the  proofs to Sec.\,\ref{sec:convergence proof}.
Given $\{\mathbf{x}^n\triangleq (\mathbf{x}_{(i)}^n)_{i=1}^{I}\}_{n\in\mathbb{N}_+}$    generated by Algorithm~\ref{alg:SONATA}, convergence  is stated measuring the distance of the average sequence $\bar{\mathbf{x}}^n \triangleq (1/I)\cdot \sum_{i=1}^I \mathbf{x}_{(i)}^n$ from optimality and well as the consensus disagreement among the local variables $\mathbf{x}_{(i)}^n$'s. Distance from stationarity is measured by   the following function:\vspace{-0.2cm} %$J:\mathbb{R}^m\to \mathbb{R}$,  
\begin{equation}\label{eq:def_J}\begin{aligned}
& J(\bar{\mathbf{x}}^n)\triangleq \\ &\left\| \bar{\mathbf{x}}^n - \underset{\mathbf{z}\in \mathcal{K}}{\text{argmin}} \left\{\Big(\nabla F(\bar{\mathbf{x}}^n)-\nabla G^{-}\big(\bar{\mathbf{x}}^n\big)\Big)^{\top} (\mathbf{z}-\bar{\mathbf{x}}^n)+ \dfrac{1}{2}\|\mathbf{z}-\bar{\mathbf{x}}^n\|^2+G(\mathbf{z})^+\right\} \right\|.
 \end{aligned}\end{equation}
Note that $J$  is a valid measure of stationarity because it is continuous and $J(\bar{\mathbf{x}}^\infty)=0$ if and only if $\bar{\mathbf{x}}^\infty$ is a  d-stationary solution of Problem~\eqref{eq: P} \cite{facchinei2015parallel}.   
The consensus disagreement at iteration $n$ is defined as $$ D (\mathbf{x}^n) \triangleq \| \mathbf{x}^n - \mathbf{1}_I \otimes \bar{\mathbf{x}}^n\|.$$
Note that $D$   is equal to $0$ if and only if all the $\mathbf{x}_{(i)}^n$'s are consensual.
We combine the metrics $J$ and $D$ in a single merit function, defined as  
\begin{equation*}
M(\mathbf{x}^n) \triangleq \max\big\{ J (\bar{\mathbf{x}}^n) ^2,\, D (\mathbf{x}^n)^2 \big\}.
\end{equation*}
%which captures the progresses of the algorithm  towards optimality and consensus. 

We are now ready to state the main convergence results  for Algorithm~\ref{alg:SONATA}. 

\vspace{-0.1cm}

	\begin{theorem}[asymptotic convergence]\label{thm:convergence}
	Given Problem~\eqref{eq: P} and Algorithm~\ref{alg:SONATA}, suppose that Assumptions \ref{assumption:P}-\ref{assumption:alpha} are satisfied; and  let  $\{\mathbf{x}^n\}_{n\in \mathbb{N}_+}$  be the sequence generated by the algorithm. Then,  there holds %\begin{equation}
 $\lim_{n\to \infty} M(\mathbf{x}^n)=0$.%\label{eq:SONATA_convergence_Asympt}	
 %\end{equation}
%the following hold:
%	\begin{enumerate}[leftmargin=.8cm,label=(\theassumption\arabic*)]
%	\item[i)]\emph{(convergence):}	 {$\{U\left(\bar{\mathbf{x}}^n\right)\}_{n\in\mathbb{N}_+}$ converges,} $\{\bar{\mathbf{x}}^n\}_{n\in \mathbb{N}_+}$ is bounded, and every limit point of $\{\bar{\mathbf{x}}^n\}_{n\in \mathbb{N}_+}$ is a d-stationary  solution of Problem~\eqref{eq: P};   \smallskip 
	%\item[ii)]\emph{(consensus):} $\|\mathbf{x}_{(i)}^n-\bar{\mathbf{x}}^n\|_2\to 0$ as $n\to +\infty$, for all $i=1,\ldots I$.
	% \end{enumerate}
\end{theorem}
%\begin{proof}
%See Sec.~\ref{sec:convergence proof}.
%\end{proof}

%Roughly speaking, Theorem~\ref{thm:convergence} states two results:\,1)\,the  average $\bar{\mathbf{x}}^n$ of the $\mathbf{x}_i$'s converges to a d-stationary solution of  Problem~\eqref{eq: P}; 2) the $\mathbf{x}_i$'s  asymptotically agree on the common value $\bar{\mathbf{x}}^n$.  

Under a constant step-size (Assumption \ref{assump:E2}.2), the next theorem provides an upper bound on the number of iterations needed to decrease $M({\mathbf{x}}^n)$ below a   given accuracy $\epsilon>0$.   %We will use the norm of the following quantity as a measure of optimality:
%\begin{equation*}
%M_U(\mathbf{x})\triangleq \mathbf{x} - \underset{\mathbf{y}\in \mathcal{K}}{\text{argmin}} \left\{\left(\nabla F(\mathbf{x})-\nabla G^-(\mathbf{x})\right)^{\top}\!\!\!(\mathbf{y}-\mathbf{x})+G^+(\mathbf{yni2
%})+\dfrac{1}{2}\|\mathbf{y}-\mathbf{x}\|^2_2 \right\}.
%\end{equation*}
%This is a valid measure of stationarity because $M_U(\mathbf{x})$ is continuous and $M_U(\mathbf{x})=0$ if and only if $\mathbf{x}$ is a  d-stationary solution of Problem~\eqref{eq: P}.  We thus study the rate of decrease of $\|M_U(\bar{\mathbf{x}}^n)\|^2_2$. Specifically, the next theorem provides an upper bound on the number of iterations needed to decrease $\|M_U(\bar{\mathbf{x}}^n)\|^2_2$ below a   given accuracy $\epsilon>0$.

%\bigskip 
%\textcolor{red}{define the necessary constants here} 
%\bigskip
 
\begin{theorem}[complexity] \label{thm:convergence_complexity}
	%Given Problem~\eqref{eq: P} and Algorithm~\ref{alg:SONATA}, 
	Suppose that Assumptions \ref{assumption:P}-\ref{assumption:f tilde} are satisfied; and let  $\{\mathbf{x}^n\}_{n\in \mathbb{N}_+}$  be the sequence generated by Algorithm~\ref{alg:SONATA}, with a constant step-size $\alpha^n=\alpha$, satisfying Assumption~\ref{assump:E2}.2. Given $\epsilon>0$,  let $T_\epsilon$  be the first iteration $n$ such that $M({\mathbf{x}}^n)\leq \epsilon$. Then  
	 $T_\epsilon=\mathcal{O}({1}/{\epsilon})$.
\end{theorem}	
 
%Note that the inverse proportionality   between the desired
%optimality achieved tolerance  and the required number of iterations, i.e. $T_\epsilon \varpropto \frac{1}{\epsilon}$,  is consistent with  previous complexity results available in the literature for special cases of SONATA, applicable  only to \emph{convex, smooth} (mostly unconstrained) problems.

\begin{remk}[generalizations]\label{remark_step-size}  Theorems~\ref{thm:convergence} and \ref{thm:convergence_complexity} can be established with minor modifications under the setting wherein each agent $i$ uses different constant step-size $\alpha_i$. Also the assumption on the strongly convexity of the surrogate function $\widetilde{f}_i$ (Assumption~\ref{assumption:f tilde}.2) can be weakened to just convexity, if the feasible set $\mathcal{K}$ is compact. With mild additional assumptions on $G^+$--see \cite{LorenzoScutari-J'16}--we can extend convergence results in Theorem~\ref{thm:convergence} to the case wherein agents solve their subproblems (\ref{eq: x_tilde}) inexactly.  We omit further details because of space limitation.
\end{remk}\vspace{-0.6cm}
 \subsection{Discussion}\label{sec:discussion} \vspace{-0.2cm}

%\noindent \textbf{On the convergence.} %Theorem~\ref{thm:convergence} and Theorem~\ref{thm:convergence_complexity} represent a major contribution of this work. Specifically, i) 
Theorem~\ref{thm:convergence} (resp. Theorem~\ref{thm:convergence_complexity}) provides the first  convergence (resp. complexity) result of distributed algorithms for \emph{constrained} and/or \emph{composite} optimization problems over time-varying (undirected or directed) graphs, which   significantly enlarges the class of convex and nonconvex problems  which distributed algorithms can be applied to with convergence guarantees. %We remark that  convergence is proved without requiring that  the (sub)gradients of $F$ or $G$ is bounded; this is a major improvement with respect to current distributed methods for  nonconvex problems \cite{bianchi2013convergence,LorenzoScutari-J'16,tatarenko2015non,hong2016decomposing,Wai2017DeFW} and nonsmooth convex ones  \cite{nedic2015distributed}.  Theorem~\ref{thm:convergence} is also the first convergence result of distributed algorithms employing  a \emph{fixed} step-size and applied to i)  composite and/or constrained convex optimization problems; or ii)  nonconvex ones over  time-varying (di)graphs. Finally, Theorem~\ref{thm:convergence_complexity}  provides the first complexity result  of distributed algorithms for \emph{constrained} and/or \emph{composite} optimization.
	\smallskip
 	
	%On the technical side, our major contribution, which allows us to obtain the above new results, is i) a new technique in the convergence analysis .. .[mention that we do not treat optimization and consensus separately], which represents a major departure from current approaches [cite]; and ii) the introduction of a new Lyapunov function, instrumental to prove .... . We refer the interested reader to Sec.~\ref{sec:convergence proof} for details. \smallskip 

%\noindent\textbf{On the choice of the algorithm parameters.} 
SONATA represents a
gamut of algorithms, each of them corresponding to a specific choice of the surrogate function $\widetilde{f}_i$, step-size $\alpha^n$, and  matrices $\mathbf{A}^n$. Convergence is guaranteed under several choices of the free parameters of the algorithms, some of which are   briefly discussed next. 

\noindent$\bullet$ \textbf{On the choice of $\widetilde{f}_i$.}
Examples of $\widetilde{f}_i$ satisfying Assumption~\ref{assumption:f tilde} are \vspace{-0.1cm} 
\begin{itemize}
	\item[$-$]\textit{Linearization}: Linearize $f_{i}$ and add a proximal regularization (to make $\tilde{f}_i$ strongly convex), which leads to  \vspace{-0.1cm} 
$$\widetilde{f}_{i}\big(\mathbf{x}_{(i)};\mathbf{x}_{(i)}^n\big) = f_{i}\big(\mathbf{x}_{(i)}^n\big)+  \nabla f_{i}\big(\mathbf{x}_{(i)}^n\big)^{\top}\big(\mathbf{x}_{(i)}-\mathbf{x}_{(i)}^n\big) +\frac{\tau_{i}}{2}\big\|\mathbf{x}_{(i)}-\mathbf{x}_{(i)}^n\big\|^{2}_2;$$

 \item[$-$]\textit{Partial Linearization}: Consider the case where $f_{i}$ can be decomposed as $f_{i}(\mathbf{x}_{(i)})=f_{i}^{(1)}(\mathbf{x}_{(i)})+f_{i}^{(2)}(\mathbf{x}_{(i)})$, where $f_{i}^{(1)}$ is convex and $f_{i}^{(2)}$ is  nonconvex with Lipschitz continuous gradient.  Preserving the convex part of $f_i$ while linearizing $f_{i}^{(2)}$  leads to the following valid surrogate  \vspace{-0.1cm} 
\begin{align*}
	\widetilde{f}_{i}(\mathbf{x}_{(i)};\mathbf{x}_{(i)}^n) = & \,\,f_{i}^{(1)}\!(\mathbf{x}_{(i)}) + f_{i}^{(2)}(\mathbf{x}_{(i)}^n)+ \frac{\tau_{i}}{2}\,\|\mathbf{x}_{i}-\mathbf{x}_{(i)}^n\|^{2}\smallskip\\ 
&+ \nabla f_{i}^{(2)}(\mathbf{x}_{(i)}^n)^{\top}(\mathbf{x}_{i}-\mathbf{x}_{(i)}^n);\vspace{-0.1cm} 
\end{align*}
 
\item[$-$]\textit{Partial Convexification:} Consider the case where  $\mathbf{x}_{(i)}$ is partitioned as $(\mathbf{x}_{(i,1)},\mathbf{x}_{(i,2)})$, and $f_{i}$ is convex in $\mathbf{x}_{(i,1)}$ but not in $\mathbf{x}_{(i,2)}$.  Then, one can % $\widetilde{f}_{i}$ can be constructed by convexifying 
convexify only  the nonconvex part of $f_{i}$, which leads to the  surrogate: 
\begin{align*}
\widetilde{f}_{i}(\mathbf{x}_{(i)};\mathbf{x}_{(i)}^n)=& \,f_{i}(\mathbf{x}_{(i,1)},\mathbf{x}_{(i,2)}^{n})+\frac{\tau_{i}}{2}\,\|\mathbf{x}_{(i,2)}-{\mathbf{x}_{(i,2)}^{n}}\|^{2}\smallskip\\
&+\nabla^{(2)}f_{i}(\mathbf{x}_{(i)}^n)^{\top}(\mathbf{x}_{(i,2)}-\mathbf{x}_{(i,2)}^n),
\end{align*}
where $\nabla^{(2)}f_{i}$ denotes the gradient of $f_{i}$ with respect to $\mathbf{x}_{(i,2)}$.
Other choices of surrogates can be  obtained hinging on  \cite{facchinei2015parallel,scutari_PartI,scutari_PartIII}.\end{itemize}

\noindent$\bullet$ \textbf{On the choice of the step-size.}  Several options are possible  for the  step-size sequence $\{\alpha^n\}_{n}$ satisfying the  diminishing-rule in Assumption \ref{assumption:alpha}.1; see, e.g., \cite{bertsekas1999nonlinear}. Two instances  we found to be effective in our experiments are:  i) $\alpha^n = \alpha_{0}/\left(n+1\right)^{\beta},$ with $\alpha_0>0$ and $0.5<\beta\leq 1$; and ii) $\alpha^n = \alpha^{n-1}\left(1-\mu\alpha^{n-1}\right),$ with $\alpha^0\in(0,1]$, and  $\mu\in\left(0,1\right)$. 
\smallskip

\noindent$\bullet$ \textbf{On the choice of matrix $\mathbf{A}^n$.}  When dealing with digraphs, the key requirement of Assumption~\ref{assumption:G} is that each $\mathbf{A}^n$ is column stochastic. Such matrices can be built locally  by the agents:  each agent $j$ can simply choose weight $a_{ij}^n$ for $i\in \mathcal{N}_j^{\rm out}[n]$ so that $\sum_{i\in \mathcal{N}_j^{\rm out}[n]} a_{ij}^n = 1$. As a special case, $\mathbf{A}^n $ can be set to be the following push-sum matrix \cite{kempe2003gossip}: $a_{ij}^n =1/d_j^n$, if $\left(j,i\right)\in \mathcal{E}^n$; and $a_{ij}^n =0$, otherwise; 
%\begin{equation}
%\begin{aligned}\label{eq: push sum weighting}
%a_{ij}^n = \begin{cases}
%1/d_j^n & \left(j,i\right)\in \mathcal{E}^n,\\
%0 & \textrm{otherwise};
%\end{cases}
%\end{aligned}
%\end{equation}
where $d_i^n $ is the out-degree of agent $i$. In this case, the information mixing process in Step 2 %performing consensus and tracking, 
becomes a broadcasting protocol, which requires from each agent only the knowledge of its out-degree.

When the digraphs  $\mathcal{G}^n$ admit a double-stochastic matrix (e.g., they are  \emph{undirected}),  as already observed in Sec.~\ref{sec:avg_consensus} (cf.~Remark~\ref{remark_consensus}),  one can choose $\mathbf{A}^n$ as  double-stochastic; and the consensus and tracking protocols in Step 3 reduce respectively to  \vspace{-0.3cm}
%\begin{equation} \label{SONATA-NEXT}
%\begin{aligned}  				\mathbf{x}^{n+1} &=\widehat{\mathbf{A}}^n\left(\mathbf{x}^n+\alpha^n\Delta\mathbf{x}^n\right);\smallskip\\
%	 	 \mathbf{y}^{n+1}   &=\widehat{\mathbf{A}}^n\mathbf{y}^n+\left(\mathbf{g}^{n+1}-\mathbf{g}^{n}\right), 
%	 \end{aligned}
%	 \end{equation}
	 \begin{align}\label{SONATA-NEXT}
	 \begin{split}
\mathbf{x}_{(i)}^{n+1} &= \sum_{j=1}^I  a_{ij}^n   \big(\mathbf{x}_{(j)}^n + \alpha^n (\widetilde{\mathbf{x}}_{(j)}^n - \mathbf{x}_{(j)}^n) \big)\\
\mathbf{y}_{(i)}^{n+1} & = \sum_{j=1}^I  a_{ij}^n \mathbf{y}_{(j)}^n + \nabla f_i (\mathbf{x}_{(i)}^{n+1}) - \nabla f_i (\mathbf{x}_{(i)}^{n}).	 
	 \end{split}
	 \end{align}
%with  $\widehat{\mathbf{A}}^n=\mathbf{A}^n\otimes \mathbf{I}_m$ [cf.~(\ref{def_matrices})].  
Several choices have been proposed in the literature to build in a distributed way a  double stochastic matrix   $\mathbf{A}^n$, including:  the Laplacian,   Metropolis-Hastings, and maximum-degree weights; see, e.g.,  \cite{xiao2005scheme}.
%Note that the instance of SONATA as in (\ref{SONATA-NEXT}) coincides with our previous algorithm NEXT \cite{Lorenzo2015NEXT-CAMSAP15,Lorenzo2015NEXT-ICASSP16,LorenzoScutari-J'16}. %We will refer to \eqref{eq:NEXT matrix form} as \emph{(ATC/CAA-)SONATA-NEXT}. 

\noindent$\bullet$ \textbf{ATC/CAA updates.} In the case of unconstrained optimization,   the information mixing step in Algorithm~\ref{alg:SONATA} can be performed following two alternative protocols, namely: i) the  \emph{Adapt-Then-Combine-based} (ATC) scheme; and ii) the   \emph{Combine-And-Adapt-based} (CAA) approach (termed ``consensus strategy'' in \cite{sayed2014adaptation}). The former is the one used in \eqref{eq:mixing x}--each agent $i$ first updates its local copy $\mathbf{x}_{(i)}^n$ along the direction $\widetilde{\mathbf{x}}_{(i)}^n - \mathbf{x}_{(i)}^n$, and then combines its new update with that of its neighbors  via consensus. Alternatively, in the CAA update,   agent $i$ first mixes its own local copy $\mathbf{x}_{(i)}^n$ with that of its neighbors via consensus, and  then performs its local optimization-based update using $\widetilde{\mathbf{x}}_{(i)}^n - \mathbf{x}_{(i)}^n$, that is \vspace{-0.2cm}$$\mathbf{x}_{(i)}^{n+1}  = \frac{1}{\phi_i^{n+1}} \sum_{j = 1}^I a_{ij} \phi_j^n \mathbf{x}_{(j)}^n + \frac{\phi_{i}^n}{\phi_{i}^{n+1}} \cdot \alpha^n (\widetilde{\mathbf{x}}_{(i)}^n - \mathbf{x}_{(i)}^n).\vspace{-0.2cm}$$
It is not difficult to check that  SONATA based on CAA updates   converges under the same conditions as in Theorem \ref{thm:convergence}. %In the following we will consider both versions of SONATA.
\vspace{-0.4cm}

\section{SONATA and special cases}\label{sec:connection}\vspace{-0.2cm}

In this section, we contrast SONATA with related algorithms proposed in the literature \cite{Lorenzo2015NEXT-CAMSAP15,Lorenzo2015NEXT-ICASSP16,LorenzoScutari-J'16,Xu2015augmented}  and very recent proposals \cite{nedich2016achieving,XiKha-J'16,qu2016harnessing} for \emph{special}   instances of Problem~\eqref{eq: P}. We show   that algorithms in \cite{Xu2015augmented,nedich2016achieving,XiKha-J'16,qu2016harnessing}   are all special cases of  SONATA  and NEXT,   proposed in our earlier  works \cite{Lorenzo2015NEXT-CAMSAP15,Lorenzo2015NEXT-ICASSP16,LorenzoScutari-J'16,SunScutariPalomar-C'16}. 

We preliminarily   rewrite Algorithm~\ref{alg:SONATA} in a matrix-vector form.  { Similarly to $\mathbf{x}^n$, define the concatenated vectors \begin{align}\label{eq:definition_symbols_2}
	\widetilde{\mathbf{x}}^n &\triangleq [\widetilde{\mathbf{x}}_{(1)}^{n\top},\ldots,\widetilde{\mathbf{x}}_{(I)}^{n\top}]^\top,\\
	\mathbf{y}^n  & \triangleq  [\mathbf{y}_{(1)}^{n\top},\ldots,\mathbf{y}_{(I)}^{n\top}]^\top,\\
	\mathbf{g}^n & \triangleq  [\mathbf{g}_{1}^{n\top},\ldots,\mathbf{g}_{I}^{n\top}]^\top,\quad \mathbf{g}_{i}^n  \triangleq   \nabla f_{i}(\mathbf{x}_{(i)}^n),\\
	\Delta \mathbf{x}^n & \triangleq \widetilde{\mathbf{x}}^n-{\mathbf{x}}^n,
\end{align} 
where  $\widetilde{\mathbf{x}}_{(i)}^{n}$ and  $\mathbf{y}^n_{(i)}$   are defined in (\ref{eq: x_tilde}) and  \eqref{eq: update y}, respectively.   
%$\widetilde{\mathbf{x}}^n\triangleq [\widetilde{\mathbf{x}}_{(1)}^{n\top},\ldots,\widetilde{\mathbf{x}}_{(I)}^{n\top}]^\top$,  $\mathbf{y}^n\triangleq  [\mathbf{y}_{(1)}^{n\top},\ldots,\mathbf{y}_{(I)}^{n\top}]^\top$, and  $\mathbf{g}^n\triangleq  [\mathbf{g}_{1}^{n\top},\ldots,\mathbf{g}_{I}^{n\top}]^\top$, where  $\widetilde{\mathbf{x}}_{(i)}^{n}$ and  $\mathbf{y}^n_{(i)}$   are defined in (\ref{eq: x_tilde}) and  \eqref{eq: update y},   respectively; and $\mathbf{g}_{i}^n  \triangleq   \nabla f_{i}(\mathbf{x}_{(i)}^n)$. 
Using the above  notation and the matrices introduced in~\eqref{def_matrices},  SONATA [cf.~\eqref{eq: update phi}-\eqref{eq: update y}] can be written in compact form   as 
\begin{align}
\boldsymbol{\phi}^{n+1} & = \mathbf{A}^n \boldsymbol{\phi}^{n} \label{eq: update vec phi}\\
\mathbf{x}^{n+1} & = \widehat{\mathbf{W}}^n (\mathbf{x}^n + \alpha^n \Delta \mathbf{x}^n) \label{eq: update vec x}\\
	 \mathbf{y}^{n+1}   &=\widehat{\mathbf{W}}^n\mathbf{y}^n+ (\DiagPhi{n+1})^{-1}\left(\mathbf{g}^{n+1}-\mathbf{g}^{n}\right).\label{eq: update vec y}
\end{align}
}
 \vspace{-0.8cm}

  \subsection{Preliminaries: SONATA-NEXT and SONATA-L}\vspace{-0.2cm}
Since   \cite{Xu2015augmented,qu2016harnessing,nedich2016achieving,XiKha-J'16} are applicable  only to \emph{unconstrained} ($\mathcal{K}=\mathbb{R}^{m}$), \emph{smooth} ($G=0$) and \emph{convex} (each $f_i$ is convex) multiagent problems, in the following,  we  consider only such an instance of Problem~\eqref{eq: P}. %Consider  Problem~\eqref{eq: P}, with $\mathcal{K} = \mathbb{R}^{m}$ (unconstraint) $G = 0$ (only smooth objectives), and  $f_i$'s convex. 
Choose each $\widetilde{f}_{i}$ as first order approximation of $f_i$ plus a proximal term, that is,\vspace{-0.2cm}
$$
\widetilde{f}_{i}(\mathbf{x}_{(i)};\mathbf{x}_{(i)}^n) =  f_{i}(\mathbf{x}_{(i)}^n) + \nabla f_{i}(\mathbf{x}_{(i)}^n)^{\top}(\mathbf{x}_{(i)} - \mathbf{x}_{(i)}^n) + \frac{\tau_{i}}{2}\,\|\mathbf{x}_{(i)} - \mathbf{x}_{(i)}^n\|^{2},
$$
and set $\tau_i=I$. Then, $\widetilde{\mathbf{x}}_{(i)}^n$  can be computed in closed form   [cf.\,\eqref{eq: x_tilde}]:\vspace{-0.2cm} 
\begin{equation}\label{x_tilde_linearized}
\begin{aligned}
 \widetilde{\mathbf{x}}_{(i)}^n 
&  = \argmin_{\mathbf{x}_{(i)}}\, (I\cdot\mathbf{y}_{(i)}^n)^{\top}(\mathbf{x}_{(i)} - \mathbf{x}_{(i)}^n) + \frac{I}{2}\,\| \mathbf{x}_{(i)} - \mathbf{x}_{(i)}^n\|^{2}\\
& = \argmin_{\mathbf{x}_{(i)}}\, \frac{I}{2} \big\| \mathbf{x}_{(i)} - \mathbf{x}_{(i)}^n +  \mathbf{y}_{(i)}^n \big\|^{2} = \mathbf{x}_{(i)}^n - \mathbf{y}_{(i)}^n.
\end{aligned} 
\end{equation}
Therefore, $\Delta \mathbf{x}_{(i)}^n=\widetilde{\mathbf{x}}_{(i)}^n-{\mathbf{x}}_{(i)}^n=\mathbf{y}_{(i)}^n$.

Substituting \eqref{x_tilde_linearized} into \eqref{eq: update vec x} and using either ATC or CAA mixing protocols, Algorithm~\ref{alg:SONATA} reduces to  \vspace{-0.1cm}  
\begin{align}\label{eq:SONATA-L}
	\boldsymbol{\phi}^{n+1} &= \mathbf{A}^n\,\boldsymbol{\phi}^n \nonumber\\
	%\mathbf{W}^n&=\left(\boldsymbol{\Phi}^{n+1}\right)^{-1}\mathbf{A}^n \boldsymbol{\Phi}^n\nonumber\\
	\mathbf{x}^{n+1} & = \begin{cases}
		\widehat{\mathbf{W}}^n\left(\mathbf{x}^n - \alpha^n\,\mathbf{y}^n\right) & \text{ (ATC-based update)}\\
		\widehat{\mathbf{W}}^n\mathbf{x}^n - \alpha^n\,\left(\DiagPhi{n+1}\right)^{-1}\DiagPhi{n}\mathbf{y}^n &\text{ (CAA-based update)}
	\end{cases}\\
	\mathbf{y}^{n+1}   &=\widehat{\mathbf{W}}^n\mathbf{y}^n+\left(\widehat{\mathbf{D}}_{\boldsymbol{\phi}^{n+1}}\right)^{-1}\left(\mathbf{g}^{n+1}-\mathbf{g}^{n}\right);\nonumber\hspace{-0.3cm}
\end{align}
which we will  refer to as \emph{(ATC/CAA-)SONATA-L} (L stands for ``linearized'').

 When the digraph $\mathcal{G}^n$ admits a \emph{double-stochastic}  matrix $\mathbf{A}^n$, and $\mathbf{A}^n$ in \eqref{eq: update vec phi} is chosen so, the iterates \eqref{eq:SONATA-L}  can be further simplified as  reduces to\hspace{-0.3cm}
%\begin{equation}\label{eq:NEXT matrix form}
	\begin{align}\label{eq:NEXT matrix form}
				\mathbf{x}^{n+1} &=\begin{cases}
			\widehat{\mathbf{W}}^n\left(\mathbf{x}^n-\alpha^n\mathbf{y}^n\right) & \text{ (ATC-based update)}\\
			\widehat{\mathbf{W}}^n\mathbf{x}^n-\alpha^n\mathbf{y}^n & \text{ (CAA-based update)} 
		\end{cases}\\
		\mathbf{y}^{n+1} & =\widehat{\mathbf{W}}^n\mathbf{y}^n+\mathbf{g}^{n+1}-\mathbf{g}^n,\nonumber\hspace{-0.3cm}
	\end{align}
%\end{equation}
where $\mathbf{W}^n=\mathbf{A}^n$ and thus $\widehat{\mathbf{W}}^n= \mathbf{W}^n \otimes \mathbf{I}_m$. The ATC-based updates    coincide  with our previous algorithm NEXT [based on the surrogate (\ref{x_tilde_linearized})], introduced in \cite{Lorenzo2015NEXT-CAMSAP15,Lorenzo2015NEXT-ICASSP16,LorenzoScutari-J'16}. We will refer to \eqref{eq:NEXT matrix form} as  (\emph{ATC/CAA-})\emph{NEXT-L}.\vspace{-0.3cm}

\subsection{Connection with current algorithms}\vspace{-0.2cm}
We can  now show that  the algorithms recently  studied in \cite{Xu2015augmented, qu2016harnessing,nedich2016achieving,XiKha-J'16} are all special cases of SONATA and NEXT, earlier proposed in \cite{Lorenzo2015NEXT-CAMSAP15,Lorenzo2015NEXT-ICASSP16,LorenzoScutari-J'16}. % Since algorithms in \cite{Xu2015augmented,qu2016harnessing,nedich2016achieving,XiKha-J'16} are applicable  only to \emph{unconstrained} ($\mathcal{K}=\mathbb{R}^{m}$), \emph{smooth} ($G=0$) and \emph{convex} (each $f_i$ is convex) multiagent problems, in the following,  we tacitly consider such an instance of Problem~\eqref{eq: P}.
\vspace{1ex}\\
\textbf{Aug-DGM  \cite{Xu2015augmented} and Algorithm in \cite{qu2016harnessing}.}  Introduced in \cite{Xu2015augmented} for \emph{undirected, time-invariant} graphs, the  Aug-DGM algorithm  reads
\begin{equation}\label{eq:augDGM uncoordinated step-size}
\begin{aligned}
\mathbf{x}^{n+1} & = \widehat{\mathbf{W}}\left(\mathbf{x}^n - \textrm{Diag}\left(\boldsymbol{\alpha}\otimes \mathbf{1}_m\right) \mathbf{y}^n\right)\\
\mathbf{y}^{n+1} & = \widehat{\mathbf{W}}\left(\mathbf{y}^n + \mathbf{g}^{n+1}- \mathbf{g}^n\right)
\end{aligned}
\end{equation}
where  $\widehat{\mathbf{W}} \triangleq \mathbf{W}\otimes \mathbf{I}_m$;  $\mathbf{W}$ is a double stochastic matrix  satisfying Assumption~\ref{assumption:A}, and   $\boldsymbol{\alpha}$ is the vector of agents' step-sizes $\alpha_i$'s. \\\indent A similar algorithm was proposed independently   in \cite{qu2016harnessing} (in the same networking setting of \cite{Xu2015augmented}), which reads  \vspace{-0.2cm} 
\begin{equation}\label{eq:augDGM coordinated step-size}
\begin{aligned}
\mathbf{x}^{n+1} & = \widehat{\mathbf{W}}\left(\mathbf{x}^n- \alpha\mathbf{y}^n\right)\\
\mathbf{y}^{n+1} & = \widehat{\mathbf{W}}\mathbf{y}^n + \mathbf{g}^{n+1}- \mathbf{g}^n.
\end{aligned}
\end{equation}
%\del{While   algorithm \eqref{eq:augDGM uncoordinated step-size} is in principle more general than  \eqref{eq:augDGM coordinated step-size}--agents can use different step-sizes $\alpha_i$s--the assumptions in \cite{Xu2015augmented} on $\boldsymbol{\alpha}$ to guarantee convergence are difficult to be enforced in practice, and in particular in a distributed setting.}  % must satisfy to guarantee convergence are not practical, because they call for some global knowledge are not practical satisfy depend not only on the  parameters of the network, but also on the global knowledge of $\boldsymbol{\alpha}$ itself, which makes the scheme less appealing in a distributed computational setting.

%Aug-DGM  \cite{Xu2015augmented} was  shown to achieve convergence rate  $O\left(1/n\right)$  for smooth convex functions $f_i$'s, and   linear convergence $O\left(\gamma^n\right)$ for some $\gamma \in \left(0,1\right)$, if $f_i$'s are  strongly convex. 

Clearly  Aug-DGM  \cite{Xu2015augmented} in  \eqref{eq:augDGM uncoordinated step-size} with the $\alpha_i$'s equal, and Algorithm \cite{qu2016harnessing} in \eqref{eq:augDGM coordinated step-size}  coincide with (ATC-)NEXT-L [cf.~\eqref{eq:NEXT matrix form}].% if $\mathbf{W}[n] \equiv \mathbf{W}$ for all $n$, and the algorithm is shown to achieve a $O\left(1/n\right)$ rate for smooth convex functions and a linear rate $O\left(\gamma^n\right)$ for some $\gamma \in \left(0,1\right)$ if the function is in addition strongly convex. 
\vspace{1ex}\\
\noindent\textbf{(Push-)DIGing \cite{nedich2016achieving}.} Appeared in \cite{nedich2016achieving} and applicable to  \emph{$B$-strongly connected undirected graphs}, the DIGing Algorithm   reads\vspace{-0.2cm} 
\begin{equation}\label{eq:DIGing}
\begin{aligned}
\mathbf{x}^{n+1} & = \widehat{\mathbf{W}}^n\mathbf{x}^n - \alpha \mathbf{y}^n\\
\mathbf{y}^{n+1} & = \widehat{\mathbf{W}}^n\mathbf{y}^n + \mathbf{g}^{n+1}- \mathbf{g}^n,
\end{aligned}
\end{equation}
where $\mathbf{W}^n$ is a double-stochastic matrix satisfying Assumption~\ref{assumption:A}. Clearly, DIGing coincides with (CAA-)NEXT-L \cite{Lorenzo2015NEXT-CAMSAP15,Lorenzo2015NEXT-ICASSP16,LorenzoScutari-J'16}[cf.~\eqref{eq:NEXT matrix form}]. The  push-DIGing algorithm, studied in the same paper  \cite{nedich2016achieving}, extends DIGing to  $B$-strongly connected digraphs. It turns out that  push-DIGing  coincides with (ATC-)SONATA-L [cf. Eq. \eqref{eq:SONATA-L}] when $a_{ij}^n = 1/d_j^n$. %The contribution of \cite{nedich2016achieving} is then to prove that  both (CAA-)SONATA-NEXT-L and (ATC-)SONATA-L  enjoy  R-linear  convergence rate, when solving smooth strongly convex unconstraints problems. %As a byproduct, SONATA is also related to  primal-dual algorithms such as EXTRA \cite{shi2015extra},  which is shown in \cite[Sec. 2.2]{nedich2016achieving} and briefly mentioned next for completeness.
\vspace{1ex}\\
%\noindent \textbf{EXTRA.} Proposed in \cite{shi2015extra}, algorithm EXTRA for static undirected graph takes form
%\begin{equation}\label{eq:EXTRA}
%\begin{aligned}
%\mathbf{x}[n+2]  = \left(\mathbf{I} + \mathbf{W}'\right)\mathbf{x}[n+1] - \widetilde{\mathbf{W}}\mathbf{x}[n] - \alpha \left(\mathbf{g}[n+1] - \mathbf{g}[n]\right).\\
%\end{aligned}
%\end{equation}
%Consider fixed $\mathbf{W}[n]$ and eliminating variable $\mathbf{y}$ in Eq. \eqref{eq:DIGing} leads to the following update:
%\begin{equation}
%	\mathbf{x}[n+2] = 2\mathbf{W}\mathbf{x}[n+1]  - \mathbf{W}^2\mathbf{x}[n] - \alpha\left(\mathbf{g}[n+1] - \mathbf{g}[n]\right),
%\end{equation}
%which fits in Eq. \eqref{eq:EXTRA} with the identification $\mathbf{W}' = 2\mathbf{W} - \mathbf{I}$ and $\widetilde{\mathbf{W}} = \mathbf{W}^2$. 
\noindent \textbf{ADD-OPT \cite{XiKha-J'16}.}  Finally, we mention the  ADD-OPT algorithm,  proposed in \cite{XiKha-J'16} for \emph{strongly connected static digraphs}, which takes the following form:\vspace{-0.3cm}
\begin{equation}\label{eq:ADD-OPT}
\begin{aligned}
\mathbf{z}^{n+1} & = \widehat{\mathbf{A}}\mathbf{z}^n - \alpha \widetilde{\mathbf{y}}^n\\
\boldsymbol{\phi}^{n+1}& = \mathbf{A}\,\boldsymbol{\phi}^n\\
\mathbf{x}^{n+1} & = \left(\widehat{\mathbf{D}}_{\boldsymbol{\phi}^{n+1}}\right)^{-1}\mathbf{z}^{n+1}\\
\widetilde{\mathbf{y}}^{n+1} & = \widehat{\mathbf{A}}\,\widetilde{\mathbf{y}}^n + \mathbf{g}^{n+1} - \mathbf{g}^n,
\end{aligned}
\end{equation}
where $\mathbf{A}$ is a column stochastic matrix satisfying Assumption~\ref{assumption:A}, and  $\widehat{\mathbf{A}}=\mathbf{A}\otimes \mathbf{I}_m$. Defining $\mathbf{y}^n = (\widehat{\mathbf{D}}_{\boldsymbol{\phi}^{n+1}})^{-1}\widetilde{\mathbf{y}}^n$, it is not difficult to check  that \eqref{eq:ADD-OPT} can be rewritten as\vspace{-0.3cm}
\begin{equation}
\begin{aligned}
\boldsymbol{\phi}^{n+1} &= \mathbf{A}\boldsymbol{\phi}^n,\quad 
\mathbf{W}=\left(\widehat{\mathbf{D}}_{\boldsymbol{\phi}^{n+1}}\right)^{-1}\widehat{\mathbf{A}}\, \widehat{\mathbf{D}}_{\boldsymbol{\phi}^{n}}\\
\mathbf{x}^{n+1} & = \widehat{\mathbf{W}} \mathbf{x}^n - \alpha \left(\widehat{\mathbf{D}}_{\boldsymbol{\phi}^{n+1}}\right)^{-1}\widehat{\mathbf{D}}_{\boldsymbol{\phi}^{n}}\, \mathbf{y}^n\\
\mathbf{y}^{n+1} & = \widehat{\mathbf{W}}\mathbf{y}^n + \left(\widehat{\mathbf{D}}_{\boldsymbol{\phi}^{n+1}}\right)^{-1}\,\left(\mathbf{g}^{n+1} - \mathbf{g}^n\right).
\end{aligned}\label{eq:ADD-OPT2}
\end{equation}
Comparing Eq. \eqref{eq:SONATA-L} and \eqref{eq:ADD-OPT2}, one can see that %except for using a fixed mixing matrix $\mathbf{A}$, 
ADD-OPT coincides with (CAA-)SONATA-L.  %ADD-OPT is shown to have linear rate $O\left(\gamma^n\right)$ for strongly convex objective functions.

We summarize the connections between the different versions of SONATA(-NEXT) and its special cases  in  Table~\ref{tab:connection}.\vspace{-0.4cm}
\begin{table}[h]
 \setlength{\tabcolsep}{4pt}
	\centering
	\caption{Connection of SONATA with current algorithms}
	\label{tab:connection}
	\resizebox{\textwidth}{!}{
	\begin{tabular}{c  c  c  c}
		\hline\noalign{\smallskip}
		\multicolumn{1}{l}{\textbf{Algorithms}}            &   \begin{tabular}{@{}l}\textbf{Connection with} \\\textbf{SONATA}   \end{tabular}                                                                                                                          & \begin{tabular}{@{}l}\textbf{Instance of} \\\textbf{Problem (P)}   \end{tabular}                                                                & \begin{tabular}{@{}l} \textbf{Graph topology/}\\\textbf{Weight matrix}\end{tabular}                                                                      
		 \smallskip\\ \hline \noalign{\smallskip}
		\begin{tabular}{c}  NEXT \\ \cite{LorenzoScutari-J'16}    \end{tabular}                        &  \begin{tabular}{@{}l}special case of \\ SONATA \eqref{SONATA-NEXT} \end{tabular}                                                                                                                                                                                                                              & \begin{tabular}[c]{@{}l@{}} $F$ nonconvex\\ $G \neq 0$\\ $\mathcal{K}\subseteq \mathbb{R}^m$\end{tabular} & \begin{tabular}[c]{@{}l@{}}time-varying digraph/\\ doubly-stochastic weights \end{tabular}  \smallskip\\ \hline \noalign{\smallskip}
		\begin{tabular}{c} Aug-DGM  \\       \cite{qu2016harnessing,Xu2015augmented}  \end{tabular}                & \begin{tabular}[c]{@{}l@{}}ATC-NEXT-L \\($\boldsymbol{\alpha}=\alpha\mathbf{1}_I$) \eqref{eq:NEXT matrix form}\end{tabular}               & \begin{tabular}[c]{@{}l@{}} $F$ convex\\ $G = 0$\\ $\mathcal{K} = \mathbb{R}^m$\end{tabular}     & \begin{tabular}[c]{@{}l@{}}static undirected graph/\\ doubly-stochastic weights \end{tabular}                                                            \smallskip\\ \hline \noalign{\smallskip}
		\begin{tabular}{c} DIGing\\ \cite{nedich2016achieving}  \end{tabular}                         & \begin{tabular}{@{}l} CAA-\\NEXT-L   \eqref{eq:NEXT matrix form}  \end{tabular}                                                                                          & \begin{tabular}[c]{@{}l@{}} $F$ convex\\ $G = 0$\\ $\mathcal{K}=\mathbb{R}^m$\end{tabular}     & \begin{tabular}[c]{@{}l@{}}time-varying digraph/\\ doubly-stochastic weights \end{tabular}  \smallskip\\ \hline \noalign{\smallskip}
		\begin{tabular}{c}
		push-DIGing \\\cite{nedich2016achieving}
		\end{tabular} & \begin{tabular}{@{}l}ATC-\\SONATA-L   \eqref{eq:SONATA-L}      \end{tabular}                                                                                 & \begin{tabular}[c]{@{}l@{}} $F$ convex\\ $G = 0$\\ $\mathcal{K}=\mathbb{R}^m$\end{tabular}      & \begin{tabular}[c]{@{}l@{}}time-varying digraph/\\ column-stochastic weights \end{tabular}                          \smallskip\\ \hline \noalign{\smallskip}
		\begin{tabular}{c}
		ADD-OPT \\ \cite{XiKha-J'16}  
		\end{tabular}                       & \begin{tabular}[c]{@{}l@{}}ATC-\\SONATA-L \eqref{eq:SONATA-L}\end{tabular} & \begin{tabular}[c]{@{}l@{}} $F$ convex\\ $G = 0$\\ $\mathcal{K}=\mathbb{R}^m$\end{tabular}     & \begin{tabular}[c]{@{}l@{}}static  digraph/\\ column-stochastic weights \end{tabular}                                                                          \\ \hline
	\end{tabular}}\vspace{-0.2cm}

\end{table}

\section{Convergence Proof of SONATA}\label{sec:convergence proof}\vspace{-0.2cm}

In this section, we prove   convergence of SONATA; because of space limitation we prove only Theorem~\ref{thm:convergence}.  The proof consists in studying the dynamics of a suitably chosen Lyapunov function along the weighted average of the agents' local copies, and of the consensus disagreement and tracking errors. We begin introducing some convenient notation along with some preliminary results. %Unless stated otherwise, 
For the sake of simplicity, all the results of the forthcoming subsections are stated under the blanket  Assumptions A-\ref{assumption:alpha}.
\vspace{-0.4cm}

\subsection{Notations and preliminaries}\vspace{-0.2cm}
The weighted average and associated consensus disagreement are denoted   by \vspace{-0.2cm}
\begin{equation}
\wavg{x}{n}  \triangleq \frac{1}{I}\left(\boldsymbol{\phi}^{n\top}\otimes \mathbf{I}_m\right) \mathbf{x}^n  \quad \text{and}\quad 
\var{x}{n}  \triangleq \mathbf{x}^n - \Vwavg{x}{n},  \end{equation}
respectively. Similar quantities are defined for  the tracking variables $\mathbf{y}^n_{(i)}$:%Recall that 
%in Section~\ref{sec:avg_consensus} we have shown that for the plain consensus algorithm \eqref{eq:consensus_mat_form}, the weighted average of the local variable $\mathbf{x}_{(i)}^n$ is invariant over time (cf. \eqref{eq:avg_invariance}), while the row stochasticity of $\widehat{\mathbf{W}}^n$ plays the role of killing consensus error. As SONATA leverages  \eqref{eq:consensus_mat_form} to force the consensus of the local copies, it is then natural  to study  the convergence of SONATA by analyzing both of the evolution of the  average process and the consensus errors. In particular, we would like to show that the weighted average  of $\mathbf{x}_{(i)}^n$ is driven to a stationary point due to the local optimization, while all  $\mathbf{x}_{(i)}^n$'s and $\mathbf{y}_{(i)}^n$'s converges to their weighted averages, respectively.  To this end, we introduce the following notations 
\begin{equation}
 \wavg{y}{n}  \triangleq  \frac{1}{I}\left(\boldsymbol{\phi}^{n\top}\otimes \mathbf{I}_m\right) \mathbf{y}^n\quad\text{and}\quad
\var{y}{n} \triangleq \mathbf{y}^n - \Vwavg{y}{n}.
\end{equation}
 %For simplicity,  we will use  $\bar{\mathbf{x}}^n\triangleq \bar{\mathbf{x}}_{\mathbf{1}_I}$. %, the plain average of the $\mathbf{x}_{(i)}^n$'s, as $\bar{\mathbf{x}}^n$. 
 
 Recalling \eqref{eq:definition_symbols_2},   %$\widetilde{\mathbf{x}}^n= [\widetilde{\mathbf{x}}_{(1)}^{n\top},\ldots,\widetilde{\mathbf{x}}_{(I)}^{n\top}]^\top$ (cf.~Sec.~\ref{SONATA_formal}), 
 define the deviation of the local solution  $\widetilde{\mathbf{x}}_{(i)}^n$ of each agent from the weighted average as
\begin{equation}
\Deltaxi{i}^n \triangleq \widetilde{\mathbf{x}}_{(i)}^n - \wavg{x}{n},
\end{equation}
and the associated stacked vector 
%\begin{equation}
%\widetilde{\mathbf{x}}^n \triangleq \left[ \widetilde{\mathbf{x}}_{(1)}^{n\top},\ldots,  \widetilde{\mathbf{x}}_{(I)}^{n\top}\right]^\top,
%\end{equation}
% and $\Deltax^n $ defined as
\begin{equation}
\begin{aligned}
\Deltax^n &\triangleq \widetilde{\mathbf{x}}^n - \Vwavg{x}{n}.
 %= \left[ \Deltaxi{1}^{n\top},\ldots, \Deltaxi{I}^{n\top}\right]^\top.
\end{aligned}
\end{equation}
Note that $\Delta \mathbf{x}^n$ [cf.~\eqref{eq:definition_symbols_2}] can be rewritten as
\begin{equation}\label{eq:Delta_x_eq}
\Delta \mathbf{x}^n= \Deltax^n - \var{x}{n}.
\end{equation}

Using the above notation,  the dynamics of $\wavg{x}{n}$ and $\wavg{y}{n}$   generated by Algorithm~\ref{alg:SONATA} are given by [cf. \eqref{eq: update vec x} and \eqref{eq: update vec y}]:
% can be readily obtained by multiplying \eqref{eq: update vec x} and \eqref{eq: update vec y} from the left by $\left(\boldsymbol{\phi}^{n+1}\right)^{\top}\otimes \mathbf{I}_{m}/I$, which lead to the following average processes:
 \begin{subequations}\label{eq:avg process}
	\begin{align}
	\wavg{x}{n+1}&=   \wavg{x}{n} + \dfrac{\alpha^{n}}{I}\left((\boldsymbol{\phi}^{n})^{\top}\otimes \mathbf{I}_{m}\right)\Deltax^n\label{eq: x weighted ave}\\	
	\wavg{y}{n+1} & = \wavg{y}{n} + \bar{\mathbf{g}}^{n+1} - \bar{\mathbf{g}}^{n}.\label{eq: y weighted ave}
	\end{align}
	\end{subequations}
	%\begin{subequations}\label{eq:avg process}
	%\begin{align}
	%\wavg{x}{n+1}&=\dfrac{1}{I}\left((\boldsymbol{\phi}^{n+1})^{\top}\otimes \mathbf{I}_{m}\right) \, \mathbf{x}^n = \wavg{x}{n} + \frac{\alpha^{n}}{I}\sum_{i=1}^{I}\phi_{i}^{n}\Deltaxi{i}^{n}\label{eq: x weighted ave}\\	
	%\wavg{y}{n+1} & = \wavg{y}{n} + \bar{\mathbf{g}}^{n+1} - \bar{\mathbf{g}}^{n}.\label{eq: y weighted ave}
	%\end{align}
%	\end{subequations}
	Note that, since  $\mathbf{y}^{0} = \mathbf{g}^{0}$ and $\phi_i^0 = 1$, we have $\wavg{y}{n} = \bar{\mathbf{g}}^{n}$, for all $n\in \mathbb{N}_+$. 	
	 %Eq. \eqref{eq: y weighted ave} shows the in-network average of the gradients   is preserved in $\wavg{y}{n}$. If all $\mathbf{y}_{(i)}^n$'s are consensual, we must have $\wavg{y}{n}= \mathbf{y}_{(i)}^n =  \bar{\mathbf{g}}^{n}$ for all $i$. Eq. \eqref{eq: x weighted ave} shows the weighted average of the local copies $\mathbf{x}_{(i)}^n$ moves  along a convex combination of  directions $\Deltaxi{i}^{n}$, each of which can be regarded as a noisy version of the  descent direction $\widehat{\mathbf{x}}_{(i)}\left(\wavg{x}{n}\right) - \wavg{x}{n}$ generated by agent $i$ of $U$ at $ \wavg{x}{n}$. 

	Finally, we introduce the error-free local solution  map of each agent $i$, denoted by $\widehat{\mathbf{x}}_{(i)}: \mathcal K\rightarrow \mathcal K$: Given  $\mathbf{z}\in \mathcal K$ and $i=1,\ldots, I$, let\vspace{-0.2cm}%This will be instrumental to prove stationarity of the sequence generated by  
	%recalling    $\boldsymbol{\pi}^n= \sum_{j\neq i}\nabla f(\mathbf{x}_{(i)}^{n})$,    
	\begin{equation}\label{error-free_best_response}
 \widehat{\mathbf{x}}_{(i)}(\mathbf{z})\triangleq  \underset{_{\mathbf{x}_{(i)}\in\mathcal{K}}}{\text{{argmin}}}\begin{array}[t]{l}
\left\{ \widetilde{f_{i}}\left(\mathbf{x}_{(i)};\mathbf{z}\right)-\nabla G^{-}\left(\mathbf{z}\right)^{\top}\left(\mathbf{x}_{(i)}-\mathbf{z}\right)\right.\\
\left.\quad + \left(\sum_{j\neq i}\nabla f_j(\mathbf{z})\right)^\top \left(\mathbf{x}_{(i)}-\mathbf{z}\right) + G^+(\mathbf{z} ) \right\}. 
\end{array}\vspace{-0.2cm}
 \end{equation}
It is not difficult to check that  $\widehat{\mathbf{x}}_{(i)}(\bullet)$ enjoys the following properties (the proof of the next lemma  follows similar steps as  in \cite[Prop. 8]{facchinei2015parallel} and thus is omitted). 
\begin{lem}\label{lemma_BR_properties} 
	 %In the setting of Theorem~\ref{thm:convergence}, 
	 Each   $\widehat{\mathbf{x}}_{(i)}(\bullet)$ satisfies:
	\begin{itemize}
	%\item[i)] \emph{[\texttt{Optimality}]:} \textcolor{red}{write here the inequality coming from the min principle if you use later}
	\item [i)] \emph{[\texttt{Lipschitz continuity}]:} $\widehat{\mathbf{x}}_{(i)}(\bullet)$ is $\hat{L}$-Lipschitz continuous on $\mathcal K$, that is, there exits a finite $\hat{L}>0$ such that \vspace{-0.1cm}\begin{equation}\label{BR_Lip}
 	\norm{\widehat{\mathbf{x}}_{(i)}(\mathbf{z})-\widehat{\mathbf{x}}_{(i)}(\mathbf{w})} \leq \hat{L}\, \norm{\mathbf{z}-\mathbf{w}},\quad \forall \mathbf{z},\,\mathbf{w}\in \mathcal K;\smallskip
 \end{equation}
	\item[ii)] \emph{[\texttt{Fixed-points}]:} The set of fixed points of   $\widehat{\mathbf{x}}_{(i)}(\bullet)$ coincides with the set of  d-stationary solutions of Problem~\eqref{eq: P}.
	\end{itemize} 
\end{lem}  
% \begin{proof} \textcolor{red}{see my comments in the appendix} See Appendix \ref{proof_lemma_best-response_properties}.\qed \end{proof} 
 
 The next result shows that, as expected, the disagreement between agent $i$'s  solution $\widetilde{\mathbf{x}}_{(i)}^n$ and its error-free counterpart  $\widehat{\mathbf{x}}_{(i)}(\mathbf{x}_{(i)}^n)$ asymptotically vanishes if both the consensus error $\var{x}{n}$ and the tracking error $\var{y}{n}$ do so.\smallskip 
 
 \begin{lem}%[Best-response consistency]
 \label{lem:best-response consistency}
	    $\widetilde{\mathbf{x}}_{(i)}^{n}$  [cf.~(\ref{eq: x_tilde})] and  $\widehat{\mathbf{x}}_{i}(\mathbf{x}^{n}_{(i)})$ [cf.~(\ref{error-free_best_response})]  satisfy:\vspace{-0.2cm}\begin{equation}\label{BR-consistency_eq}
	 \norm{\widehat{\mathbf{x}}_{i}(\mathbf{x}^{n}_{(i)})- \widetilde{\mathbf{x}}_{(i)}^{n}} \leq \dfrac{I}{\tau_{i}} \,\norm{\var{y}{n}} + { \frac{2\,I\,L}{\tau_i}\, \|\var{x}{n}\|}.\vspace{-0.1cm}
	\end{equation}
	Therefore,  
$ \|\var{x}{n}\|,  \|\var{y}{n}\|\underset{n\to\infty}{\longrightarrow} 0$ $\Rightarrow$  $\|\widehat{\mathbf{x}}_{i}(\mathbf{x}^{n}_{(i)})-\widetilde{\mathbf{x}}_{(i)}^{n}\| \underset{n\to\infty}{\longrightarrow} 0$.
\end{lem} 

 The last result of this section is a %We conclude this section stating a 
 standard martingale-like result; the proof follows similar to that of \cite[Lemma 1]{Ber_Tsitsi} and thus is omitted. 
 % that will be used in our proofs. 
 \begin{lem}\label{B-martingale}
 	 {Let $\{X^n\}_{n\in\mathbb{N}_+}$, $\{Y^n\}_{n\in\mathbb{N}_+}$ and $\{Z^n\}_{n\in\mathbb{N}_+}$ be three sequences such that $X^n$ and $Z^n$ are nonnegative, for all $n\in \mathbb{N}_+$. Suppose that  
 	\begin{equation}\label{eq:B-step-descent}
 	\sum_{k=0}^{\bar{B}-1} Y^{n+\bar{B}+ k} \leq \sum_{k=0}^{\bar{B}-1} Y^{n+ k} - \sum_{k=0}^{\bar{B}-1} X^{ n+k} + \sum_{k=0}^{\bar{B}-1} Z^{n+k}, \quad n = 0,1,\ldots,
 	\end{equation}
 	and that $\sum_{n=0}^{\infty} Z^n < +\infty$. Then, either $\sum_{k=0}^{\bar{B}-1} Y^{n+ k} \to -\infty$, or else $\sum_{k=0}^{\bar{B}-1} Y^{n+ k}$ converges to a finite value and $\sum_{n=0}^{\infty} X^n < +\infty$.}
 \end{lem}\vspace{-0.4cm}
 %\begin{proof}
%Define new variables
%\begin{align*}
%y_{\bar{B}}^r \triangleq \sum_{k=0}^{\bar{B}-1} y^{r\bar{B} +k},\quad x_{\bar{B}}^r \triangleq \sum_{k=0}^{\bar{B}-1} x^{r\bar{B} +k},\quad z_{\bar{B}}^r \triangleq \sum_{k=0}^{\bar{B}-1} z^{r\bar{B} +k}, 
%\end{align*}
%Eq. \eqref{eq:B-step-descent} implies 
%\begin{equation}
%y_{\bar{B}}^{r+1}  \leq y_{\bar{B}}^r  - z_{\bar{B}}^r + x_{\bar{B}}^r ,r = 0,1,\ldots.
%\end{equation}
%From the super-martingale convergence theorem we have if $\sum_{r=0}^{\infty}z_{\bar{B}}^r < +\infty$, then either $y_{\bar{B}}^r \to -\infty$ or $y_{\bar{B}}^r$ converges to a finite value and $\sum_{r=0}^{\infty} x_{\bar{B}}^r < +\infty$.\\
%From 
%\begin{equation}
%\sum_{t=0}^{r} x_{\bar{B}}^r = \sum_{t=0}^{r} \sum_{k=0}^{\bar{B}-1} x^{t\bar{B} +k} = \sum_{k=0}^{(r+1)\bar{B}-1}x^k,
%\end{equation}
%we have $\sum_{r=0}^{\infty} x_{\bar{B}}^r < +\infty$    iff $\sum_{n=0}^\infty x^n <\infty$. Similarly, $\sum_{r=0}^{\infty} z_{\bar{B}}^r < +\infty$    iff $\sum_{n=0}^\infty z^n <\infty$. This completes the proof. \qed
% \end{proof}\end{comment}
  
\subsection{Average descent}	 \vspace{-0.2cm}
 We begin our analysis  studying the  dynamics of   $U$ along the trajectory of $\wavg{x}{n}$. We define the total energy of the optimization input $\alpha^n\Deltax^{n}$ and consensus errors $\var{x}{n}$ and $\var{y}{n}$ in $\bar{B}$ consecutive iterations [$\bar{B}$ is defined in Lemma~\ref{cor:averaging matrix column stochastic}]:\vspace{-0.2cm} 
\begin{equation}\label{eq: energy_consensus}
\begin{aligned}
\eDeltax^n  \triangleq \cumsum \left(\alpha^{n+t}\right)^2 \norm{\Deltax^{n+t}}^2,\quad
\exorth^n   \triangleq \cumsum \norm{\var{x}{n+t}}^2,\quad 
\eyorth^n  \triangleq \cumsum \norm{\var{y}{n+t}}^2.
\end{aligned}
\end{equation}
Recalling the definitions of $c_\tau$, $\LBphi$, and $\UBphi$   [see  \eqref{consensus-error-bound-pert-cons} and 
\eqref{eq:deinitions_c_L}], we have the following.
\begin{lem}\label{lem:V-U part}
Let $\left\{(\mathbf{x}^{n},\mathbf{y}^{n})\right\}_{n\in\mathbb{N}_+}$ be the sequence generated by Algorithm~\ref{alg:SONATA}. %, under Assumption \ref{assumption:P}-\ref{assumption:f tilde}. 
Then,  there holds 
\begin{equation}\label{eq:V-U part column stochastic}
\begin{aligned}
&\phantom{{}\leq{}}\sum_{k=0}^{\bar{B}-1}U\left(\wavg{x}{n+\bar{B}+k}\right)\\
& \leq  \sum_{k=0}^{\bar{B}-1}U\left(\wavg{x}{n+k}\right)- \frac{c_{\tau}\, \LBphi}{I}\cdot\sum_{k=0}^{\bar{B}-1}\cumsum\alpha^{n+k+t}\norm{\Deltax^{n+k+t}}^{2} \\
&\phantom{{}\leq{}}+{ \frac{\UBphi}{2}\left(\frac{L+{L_G}}{I}\,\UBphi + c_{L}\epsilon_{x} + \epsilon_{y}\right)\sum_{k=0}^{\bar{B}-1}\eDeltax^{n+k}} \,+\,\underbrace{\frac{\UBphi}{2}\sum_{k=0}^{\bar{B}-1}\left(c_{L}\epsilon_{x}^{-1}\exorth^{n+k} + \epsilon_{y}^{-1}\eyorth^{n+k}\right)}_{\text{\emph{term iv}}},
\end{aligned}
\end{equation}
where $\epsilon_x >0$ and $\epsilon_y>0$ are arbitrary, finite constants.  %$c_{\tau}\triangleq \min\limits_{1\leq i \leq I}\tau_{i}$, $c_{L}\triangleq L + \left(L+\tilde{L}\right)/I$ with $L\triangleq \max\left\{\max\limits_{1\leq i\leq I}L_{i}, L_G\right\}$ and $\tilde{L} \triangleq \max\limits_{1\leq i\leq I}\tilde{L}_{i}$, and
\end{lem}
\begin{proof}
Denote for simplicity $\bar{F} \triangleq F - G^-$.
Since $\widetilde{f}_i$ is strongly convex and $G^+$ is convex, by the first order optimality of $\widetilde{\mathbf{x}}_{(i)}^n$, we have \vspace{-0.2cm}
	\begin{multline}\label{eq:x tilde descent}	
	\left(\Deltaxi{i}^{n}\right)^\top\left(I\cdot\mathbf{y}_{(i)}^n +\nabla \widetilde{f}_i(\wavg{x}{n};\mathbf{x}_{(i)}^n) - \nabla G^- (\mathbf{x}_{(i)}^n) -\nabla f_i(\mathbf{x}_{(i)}^n)\right)\\
		+  G^+(\widetilde{\mathbf{x}}_{(i)}^n) - G^+(\wavg{x}{n}) \leq -\tau_i\|\Deltaxi{i}^{n}\|^2.
	\end{multline}
%Adding and subtracting 	$\nabla G^- \left(\wavg{x}{n}\right)$  we get
%	\begin{multline}
%	\left(\widetilde{\mathbf{x}}_{(i)}^n-\wavg{x}n\right)^\top\left(I\cdot\mathbf{y}_{(i)}^n +\nabla \widetilde{f}_i\left(\wavg{x}{n};\mathbf{x}_{(i)}^n\right) - \nabla G^-\left(\wavg{x}{n}\right) -\nabla f_i\left(\mathbf{x}_{(i)}^n\right)\right) +  G^+\left(\widetilde{\mathbf{x}}_{(i)}^n\right) \\- G^+\left(\wavg{x}{n}\right)
%		 \leq -\tau_i\norm{\Deltaxi{i}^{n}}^2 - 	 \left(\widetilde{\mathbf{x}}_{(i)}^n-\wavg{x}n\right)^\top\left(\nabla G^- \left(\wavg{x}{n}\right) - \nabla G^- \left(\mathbf{x}_{(i)}^n\right) \right)\\
%		 \leq -\tau_i\norm{\Deltaxi{i}^{n}}^2 + 	L_{G} \norm{\mathbf{x}_{(i)}^n-\wavg{x}n}\norm{\Deltaxi{i}^n}
%	\end{multline}

Since $\nabla f_i$ and $\nabla G^-$ are $L_i$ and  $L_G$-Lipschitz, respectively,  $\nabla F$ is $(L+L_G)$-Lipschitz, where $L \triangleq \sum_{i=1}^I L_i$ [cf. def. \eqref{eq:deinitions_c_L}]. Applying the descent lemma to $\bar{F}$ and using \eqref{eq: x weighted ave} yields 
\begin{align} \label{eq:F-SONATA}
 &\bar{F}\left(\wavg{x}{n+1}\right)  \nonumber\\
& \leq \bar{F}\left(\wavg{x}{n}\right) + \frac{\alpha^{n}}{I}\,\nabla \bar{F}\left(\wavg{x}{n}\right)^{\top}\left((\boldsymbol{\phi}^{n})^{\top}\otimes \mathbf{I}_{m}\right)\Deltax^n \nonumber\\
&\quad + \frac{L+L_G}{2}\cdot \frac{\left(\alpha^{n}\right)^{2}}{I} \norm{\left((\boldsymbol{\phi}^{n})^{\top}\otimes \mathbf{I}_{m}\right)\Deltax^n}^{2}\nonumber\\
%& \leq \bar{F}\left(\wavg{x}{n}\right) +\frac{\alpha^{n}}{I}\sum_{i=1}^{I}\phi_{i}^{n}\left(I\cdot\mathbf{y}_{(i)}^n +\nabla \widetilde{f}_i\left(\wavg{x}{n};\mathbf{x}_{(i)}^n\right)- \nabla G^- \left(\mathbf{x}_{(i)}^n\right)  -\nabla f_i\left(\mathbf{x}_{(i)}^n\right)\right)^{\top}\Deltaxi{i}^{n} \nonumber\\
%& \phantom{{}\leq{}}+ \frac{\alpha^{n}}{I}\sum_{i=1}^{I}\phi_{i}^{n}\left(\nabla \bar{F}\left(\wavg{x}{n}\right) - I\cdot\mathbf{y}_{(i)}^n -\nabla \widetilde{f}_i\left(\wavg{x}{n};\mathbf{x}_{(i)}^n\right) + \nabla G^- \left(\mathbf{x}_{(i)}^n\right) +\nabla f_i\left(\mathbf{x}_{(i)}^n\right)\right)^{\top}\Deltaxi{i}^{n}\nonumber\\
%& \phantom{{}\leq{}}+ \frac{L+L_G}{2} \cdot \frac{\left(\alpha^{n}\right)^{2}}{I}\UBphi^2\norm{\Deltax^{n}}^{2}\nonumber\\
& \stackrel{(a)}{\leq} \bar{F}\left(\wavg{x}{n}\right) + \frac{L+L_G}{2} \cdot \frac{\left(\alpha^{n}\right)^{2}}{I}\,\UBphi^2\,\norm{\Deltax^{n}}^{2}\nonumber\\
&\phantom{{}\leq{}} - \frac{\alpha^{n}}{I}\sum_{i=1}^{I}\phi_{i}^{n}\left(\tau_{i}\norm{\Deltaxi{i}^{n}}^{2} + G^+(\widetilde{\mathbf{x}}_{(i)}^n) - G^+(\wavg{x}{n})\right)\nonumber \\
& \phantom{{}\leq{}}+\frac{\alpha^{n}}{I}\sum_{i=1}^{I}\phi_{i}^{n}\left(\nabla \bar{F}\left(\wavg{x}{n}\right) + \nabla G^- (\mathbf{x}_{(i)}^n)- I\cdot \wavg{y}{n} + I\cdot \wavg{y}{n} - I\cdot \mathbf{y}_{(i)}^{n}\right)^{\top}\Deltaxi{i}^{n} \nonumber\\
& \phantom{{}\leq{}}+ \frac{\alpha^{n}}{I}\sum_{i=1}^{I}\phi_{i}^{n}\left(\nabla f_{i}(\mathbf{x}_{(i)}^{n}) -\nabla f_{i}(\wavg{x}{n}) + \nabla \widetilde{f}_{i}(\wavg{x}{n};\wavg{x}{n})- \nabla \widetilde{f}_{i}(\wavg{x}{n};\mathbf{x}_{(i)}^{n})\right)^{\top}\Deltaxi{i}^{n} \nonumber\\
& \stackrel{(b)}{\leq} \bar{F}\left(\wavg{x}{n}\right) + \frac{L+L_G}{2} \cdot \frac{\left(\alpha^{n}\right)^{2}}{I}\,\UBphi^2\,\|\Deltax^{n}\|^{2}\nonumber\\
& \phantom{{}\leq{}}- \frac{\alpha^{n}}{I}\sum_{i=1}^{I}\phi_{i}^{n}\left(\tau_{i}\|\Deltaxi{i}^{n}\|^{2} + G^+(\widetilde{\mathbf{x}}_{(i)}^n) - G^+(\wavg{x}{n})\right) \nonumber\\
& \phantom{{}\leq{}}+ \frac{\alpha^{n}}{I}\sum_{i=1}^{I}\phi_{i}^{n}\norm{\nabla \bar{F}(\wavg{x}{n}) -\Big(\sum_{j=1}^{I}\nabla f_{j}(\mathbf{x}_{(j)}^{n})- \nabla G^-(\mathbf{x}_{(i)}^n)\Big)} \norm{\Deltaxi{i}^{n}} \nonumber\\
& \phantom{{}\leq{}}+ \alpha^{n}\sum_{i=1}^{I}\phi_{i}^{n}\|\wavg{y}{n} - \mathbf{y}_{(i)}^{n}\|\,\|\Deltaxi{i}^{n}\|\nonumber\\
&  \phantom{{}\leq{}} + \frac{\alpha^{n}}{I}\sum_{i=1}^{I}\phi_{i}^{n}\norm{\nabla f_{i}(\mathbf{x}_{(i)}^{n}) - \nabla f_{i}\left(\wavg{x}{n}\right)}\,\|\Deltaxi{i}^{n}\| \nonumber\\
& \phantom{{}\leq{}}+ \frac{\alpha^{n}}{I}\sum_{i=1}^{I}\phi_{i}^{n}\norm{\nabla \widetilde{f}_{i}(\wavg{x}{n};\wavg{x}{n})- \nabla \widetilde{f}_{i}(\wavg{x}{n};\mathbf{x}_{(i)}^{n})}\,\|\Deltaxi{i}^{n}\|\nonumber\\
&\stackrel{(c)}{\leq}  \bar{F}\left(\wavg{x}{n}\right)+ \frac{L+L_G}{2} \cdot \frac{\left(\alpha^{n}\right)^{2}}{I}\,\UBphi^2\,\|\Deltax^{n}\|^{2}\nonumber\\
&  \phantom{{}\leq{}} - \frac{\alpha^{n}}{I}\sum_{i=1}^{I}\phi_{i}^{n}\left(\tau_{i}\|\Deltaxi{i}^{n}\|^{2} + G^+(\widetilde{\mathbf{x}}_{(i)}^n) - G^+(\wavg{x}{n})\right) \nonumber\\
& \phantom{{}\leq{}}+\frac{\alpha^{n}}{I}\sum_{i=1}^{I}\phi_{i}^{n}\left(\sum_{j=1}^{I}L_{j}\|\wavg{x}{n}- \mathbf{x}_{(j)}^{n}\| + L_G\|\wavg{x}{n}- \mathbf{x}_{(i)}^{n}\| \right)\,\|\Deltaxi{i}^{n}\| \nonumber\\
& \phantom{{}\leq{}}  +\alpha^{n}\sum_{i=1}^{I}\phi_{i}^{n}\|\wavg{y}{n} - \mathbf{y}_{(i)}^{n}\|\|\Deltaxi{i}^{n}\|\nonumber\\
& \phantom{{}\leq{}}+\frac{\alpha^{n}}{I}\sum_{i=1}^{I}\phi_{i}^{n}\left(L_{i}\|\mathbf{x}_{(i)}^{n} - \wavg{x}{n}\|\,\|\Deltaxi{i}^{n}\| +\tilde{L}_{i}\|\mathbf{x}_{(i)}^{n} - \wavg{x}{n}\|\,\|\Deltaxi{i}^{n}\|\right)\nonumber\\
& \stackrel{(d)}{\leq} \bar{F}\left(\wavg{x}{n}\right) + \frac{L+L_G}{2} \cdot \frac{\left(\alpha^{n}\right)^{2}}{I}\,\UBphi^2\,\norm{\Deltax^{n}}^{2} \nonumber\\
&\phantom{{}\leq{}} - \frac{\alpha^{n}}{I}\sum_{i=1}^{I}\phi_{i}^{n}\left(\tau_{i}\|\Deltaxi{i}^{n}\|^{2} + G^+(\widetilde{\mathbf{x}}_{(i)}^n) \!-\! G^+(\wavg{x}{n})\right)\nonumber\\
& \phantom{{}\leq{}}+\alpha^{n}c_{L}\UBphi\norm{\var{x}{n}}\norm{\Deltax^{n}} + \alpha^{n}\,\UBphi\,\norm{\var{y}{n}}\norm{\Deltax^{n}},\smallskip 
\end{align}

\noindent where in  (a) we used  \eqref{eq:x tilde descent}, Assumption \ref{assumption:f tilde}.1, and the bound (\ref{eq:def delta_up}) (along with some basic manipulations); in (b) we used $\wavg{y}{n} = \bar{\mathbf{g}}^n$ [cf. \eqref{eq: y weighted ave}]; (c) follows from the $L_i$-Lipschitz continuity of $\nabla f_i$, $L_G$-Lipschitz continuity of $\nabla G^-$, and the uniformly $\tilde{L}_i$-Lipschitz continuity of $\nabla \widetilde{f}_i (\mathbf{x};\bullet)$; and in (d) we used the inequality $\norm{\mathbf{x}}_{1}\leq \sqrt{n}\norm{\mathbf{x}}$, and the definition of  $c_L$ [cf.~\eqref{eq:deinitions_c_L}].
 
Invoking the convexity of $G^+$ and using  \eqref{eq: x weighted ave}, we can write\vspace{-0.2cm}
\begin{equation*}
G^+\left(\wavg{x}{n+1}\right) \leq \left(1-\alpha^{n}\right)G^+\left(\wavg{x}{n}\right) + \frac{\alpha^{n}}{I}\sum_{i=1}^{I}\phi_{i}^{n}G^+ (\widetilde{\mathbf{x}}_{(i)}^{n}),\vspace{-0.2cm}
\end{equation*}
which combined  with \eqref{eq:F-SONATA} yields\vspace{-0.2cm}
\begin{align}%\label{eq:U column stochastic}
&U\left(\wavg{x}{n+1}\right) \nonumber\\
&\leq U\left(\wavg{x}{n}\right) - \frac{\alpha^{n}}{I}\,\LBphi\,c_{\tau}\,\norm{\Deltax^{n}}^{2} + \frac{L+L_G}{2}\cdot\frac{\left(\alpha^{n}\right)^{2}}{I}\,\UBphi^2\,\norm{\Deltax^{n}}^{2}\nonumber\\
&\phantom{{}\leq{}}+\alpha^{n}\,c_{L}\UBphi\,\norm{\var{x}{n}}_2\norm{\Deltax^{n}} + \alpha^{n}\,\UBphi\,\norm{\var{y}{n}}\norm{\Deltax^{n}}\nonumber\\
& \leq U\left(\wavg{x}{n}\right)- \frac{\alpha^{n}}{I}\,\LBphi\,c_{\tau}\,\norm{\Deltax^{n}}^{2} + \frac{L+L_G}{2} \cdot \frac{\left(\alpha^{n}\right)^{2}}{I}\,\UBphi^2\,\norm{\Deltax^{n}}^{2}\nonumber\\
&\phantom{{}\leq{}}+\frac{\UBphi}{2}\left( c_{L}\epsilon_{x} + \epsilon_{y}\right)\left(\alpha^{n}\right)^{2}\norm{\Deltax^{n}}^{2} + \frac{\UBphi}{2}\,c_{L}\,\epsilon_{x}^{-1}\norm{\var{x}{n}}^{2} + \frac{\UBphi}{2}\epsilon_{y}^{-1}\norm{\var{y}{n}}^{2},\nonumber 
\end{align}
where the last inequality follows from the Young's inequality, with $\epsilon_{x}>0$ and $\epsilon_{y}>0$.
Applying the above inequality recursively for $\bar{B}$ steps, with $\bar{B}$ defined in Lemma~\ref{cor:averaging matrix column stochastic},   yields\vspace{-0.3cm}
\begin{multline}\label{eq:U-B steps column stochastic}
U\left(\wavg{x}{n+\bar{B}}\right) 
\leq U\left(\wavg{x}{n}\right) -  \frac{c_\tau\,\LBphi}{I} \sum_{t=0}^{\bar{B}-1}\alpha^{n+t}\norm{\Deltax^{n+t}}^2 \\
+ \frac{\UBphi}{2}\left(\frac{L + L_G}{I}\cdot\UBphi+c_{L}\,\epsilon_{x} + \epsilon_{y}\right)\eDeltax^n+  \frac{\UBphi}{2}\left(c_L\epsilon_{x}^{-1}\exorth^n + \epsilon_{y}^{-1}\eyorth^n \right).
\end{multline}
Summing up \eqref{eq:U-B steps column stochastic} over $\bar{B}$ consecutive iterations leads to the desired result.\qed
\end{proof}
 
  Since, for sufficiently small $\alpha^n$, the negative term on the RHS of \eqref{eq:V-U part column stochastic}
 dominates the positive third term, to prove convergence of $\{U(\wavg{x}{n+\bar{B}+k}))\}_{n\in \mathbb{N}_+}$,  descent-based techniques used in the  literature of distributed gradient-based algorithms would call for the summability of the consensus error  $\{\exorth^n\}_{n\in \mathbb{N}_+}$ and tracking error $\{\eyorth^n\}_{n\in \mathbb{N}_+}$ sequences. However, under constant step-size or   unbounded (sub-)gradient of $U$, it seems not possible to infer such a result by just studying the dynamics of   $\{\exorth^n\}_{n\in \mathbb{N}_+}$ and   $\{\eyorth^n\}_{n\in \mathbb{N}_+}$ \emph{independently} from the optimization error $\Deltax^{n}$. Therefore, exploring the interplay  between these quantities, %consensus/tracking and  optimization errors, 
 we put forth a new analysis, based on the following steps: \vspace{-0.1cm}
 \begin{itemize}
\item[$-$]\textbf{Step~1}: We first bound     $\exorth^n$ and    $\eyorth^n$ [specifically, term iv in (\ref{eq:V-U part column stochastic})] as a function of $\eDeltax^n$ (and thus $\Deltax^{n}$)--see   Proposition~\ref{lem:V-consensus err column stochastic} [cf.~Sec.~\ref{sec:consensus err bound}]. Using  Proposition~\ref{lem:V-consensus err column stochastic}, we then prove that  $\{\exorth^n\}_{n\in \mathbb{N}_+}$ and   $\{\eyorth^n\}_{n\in \mathbb{N}_+}$ are summable, if  $\{\eDeltax^n\}_{n\in \mathbb{N}_+}$  is so$-$see Proposition~\ref{prop_bound_errors} [cf.~Sec.~\ref{sec:bound_errors}].\smallskip

\item[$-$]\textbf{Step~2}: Using Propositions~\ref{lem:V-consensus err column stochastic} and \ref{prop_bound_errors}, we build a new Lyapunov function [cf.~Sec.~\ref{sec:proof th1}], whose convergence implies the summability of $\{\eDeltax^n\}_{n\in \mathbb{N}_+}$  and thus convergence of all error sequences [cf.~Sec.~\ref{sec:final_part_of_the_proof}], as stated in Theorem~\ref{thm:convergence}.\vspace{-0.3cm}  \end{itemize}

%Lemma~\ref{lem:V-U part} shows the sum of the objective function evaluated at $\wavg{x}{}$ decreases every  $\bar{B}$ iterations, up to the last positive term III associated to the consensus errors $\exorth$ and $\eyorth$. If   $\sum_{n=0}^{\infty}\exorth^n<+\infty$ and $\sum_{n=0}^{\infty}\eyorth^n<+\infty$ hold,  then we can invoke the supermartingale  convergence theorem to show  $\lim_{n\to\infty}\|\Deltax^n\|=0$ with $\alpha^n$  sufficiently small so that the negative term I dominates the positive term II. As we have mentioned, however, without the boundedness of $\partial U$ and $\alpha^n$ diminishing, the summablity of the consensus error cannot be proved in a priori independently from optimization.
%This observation suggests adopting a new Lyapunov  function instead of $U$ to prove the stability of  the dynamic system  \eqref{eq:SONATA matrix form}. 
%This observation suggests the necessity of bounding the consensus error term III with input  $\alpha^n \Deltax^n$ generated from optimization, as shown in the next Section~\ref{sec:consensus err bound}.

%To bound term  III in (\ref{eq:V-U part column stochastic}), we first study
%.........
\subsection{Interplay among $\exorth^n$,  $\eyorth^n$ and  $\eDeltax^n$}\vspace{-0.1cm}
\subsubsection{Bounding $\exorth^n$ and $\eyorth^n$ }\label{sec:consensus err bound}\vspace{-0.1cm}
%In Section \ref{sec:avg_consensus} we have shown the plain consensus algorithm \eqref{eq:push_sum_consensus} converges at a geometric rate. This section is devoted to study 
%To bound term  III in (\ref{eq:V-U part column stochastic}), 
%To bound the consensus and tracking disagreements, 
We first study the dynamics of  $\|\var{x}n\|$ and $\|\var{y}n\|$. % in the presence of an input generated by optimization (cf {\color{red}. .. and ...[matrix form update of x and y]}), as stated in the next lemma.

\begin{lem}\label{lem:consensus err contraction column stochastic}
%Under Assumption \ref{assumption:P}-\ref{assumption:A}, 
The disagreements $\|\var{x}n\|$ and $\|\var{y}n\|$  satisfy  \vspace{-0.2cm}
\begin{equation}\label{eq:consensus err contraction x column stochastic}
\|\var{x}{n+\bar{B}}\| \leq \rhoopt\norm{\var{x}{n}} + c\sum_{t=0}^{\bar{B}-1}\alpha^{n+t}\norm{\Delta \mathbf{x}^{n+t}},\vspace{-0.2cm}
\end{equation}
\begin{equation}\label{eq:consensus err contraction y column stochastic}
\|\var{y}{n+\bar{B}}\|  \leq \rhoopt\norm{\var{y}{n}} + c \,{L_{\rm{mx}}}\,\LBphi^{-1}\,\sum_{t=0}^{\bar{B}-1}\left(2 \,\|\var{x}{n+t}\| + \alpha^{n+t} \|\Delta \mathbf{x}^{n+t}\|\right),
\end{equation}
 where $c= I\sqrt{2I}$. Furthermore, if all $\mathbf{A}^{n}$ are double stochastic, then \eqref{eq:consensus err contraction x column stochastic} and \eqref{eq:consensus err contraction y column stochastic} hold with   $\bar{B}=B$, $\rhoopt = \sqrt{1-\kappa/(2I^{2})}$  and $c = 1$.\vspace{-0.1cm}
\end{lem}
\begin{proof}
See Appendix~\ref{app:lem 5 proof}.\hfill$\square$
\end{proof}
 %Lemma~\ref{lem:consensus err contraction column stochastic} shows that  after $\bar{B}$  iterations the coefficient of the consensus error shrinks by a factor of $\rho<1$. 
 
 %To bound the error term III in \eqref{eq:V-U part column stochastic}, 
Using Lemma~\ref{lem:consensus err contraction column stochastic}, we   now  study  the dynamics of the    weighted sum of the disagreements $\|\var{x}n\|$ and $\|\var{y}n\|$   over $\bar{B}$ consecutive iterations.
%cumulated consensus errors over $\bar{B}$ consecutive iterations   . In particular, in   Proposition~\ref{lem:V-consensus err column stochastic}, we show the weighted sum of the consensus errors over $\bar{B}$ iterations
%decreases every $\bar{B}$ iterations, up to a positive term associated to the optimization input $\eDeltax$.

\begin{prop}\label{lem:V-consensus err column stochastic}
	%Under Assumption \ref{assumption:P}-\ref{assumption:A}, 
	The   sequences  $\{\|\var{x}{n}\|^2\}_{n\in\mathbb{N}_+}$ and $\{\|\var{y}{n}\|^2\}_{n\in\mathbb{N}_+}$ satisfy 
	\begin{multline}
	\wexorth{n+\bar{B}}
	\leq \wexorth{n} \\
	- \underbrace{\left(1- \left(\epsilon^{-1}+\bar{B}\right)\frac{2\bar{B}c^{2}}{1-\tilde{\rho}} \left(\alpha_{\rm{mx}}^n\right)^2\right)}_{\mu^n}\sum_{k=0}^{\bar{B}-1}\exorth^{n+k} 
	+ \underbrace{\left(\epsilon^{-1}+\bar{B}\right)\frac{2\bar{B}c^{2}}{1-\tilde{\rho}}}_{\cDeltax}\, \sum_{k=0}^{\bar{B}-1}\eDeltax^{n+k},	\label{eq:V-x column stochastic} 
	\end{multline} 
	\begin{multline}\label{eq:V-y column stochastic}
\weyorth{n+\bar{B}} 
	\leq \weyorth{n} \\- \sum_{k=0}^{\bar{B}-1}\eyorth^{n+k} + \underbrace{\left(\epsilon^{-1}+\bar{B}\right)\frac{2\bar{B}c^{2}}{1-\tilde{\rho}} \,{L_{\rm{mx}}^2}\,\LBphi^{-2}}_{\corth}\left(2 +\alpha_{\rm{mx}}^{n}\right)^2 \sum_{k=0}^{\bar{B}-1}\exorth^{n+k}\\
+ \underbrace{\left(\epsilon^{-1}+\bar{B}\right)\frac{2\bar{B}c^{2}}{1-\tilde{\rho}} \,{L_{\rm{mx}}^2}\, \LBphi^{-2}}_{\corth}\sum_{k=0}^{\bar{B}-1}\eDeltax^{n+k},	
	\end{multline}
 where
%\begin{align*}
%\omega_{k} &\triangleq \frac{k+1+(\bar{B}-k-1)\tilde{\rho}}{1-\tilde{\rho}},\  k = 0,\ldots,\bar{B}-1,\\
%\cDeltax &\triangleq \left(\epsilon^{-1}+\bar{B}\right)\frac{2\bar{B}c^{2}}{1-\tilde{\rho}}\\
$\alpha_{\rm{mx}}^n   \triangleq \max\limits_{k=0,\ldots,2\bar{B}-2}{\alpha^{n+k}}$;  %,\\
%\mu^n &\triangleq 1-\cDeltax\left(\alpha_{\rm{mx}}^n\right)^2,
%\end{align*}
$\tilde{\rho}\triangleq \rhoopt^2\left(1+\bar{B}\epsilon\right)$; and  $\epsilon >0$ is any constant such that $\tilde{\rho} <1$.\vspace{-0.2cm}
\end{prop}
\begin{proof}
 	We prove only \eqref{eq:V-x column stochastic};  \eqref{eq:V-y column stochastic} can be proved using similar steps. 
 	 Squaring both sides of the inequality \eqref{eq:consensus err contraction x column stochastic} leads to
 	 \begin{align}\label{eq: consensus err x-B column stochastic}
 	 	\begin{split}
 	 		 			 &\norm{\var{x}{n+\bar{B}}}^2 \\
 	 		& \leq  \rhoopt^2\norm{\var{x}{n}}^2 + \left(c\cdot\sum_{t=0}^{\bar{B}-1} \alpha^{n+t}\norm{\Delta\mathbf{x}^{n+t}}\right)^2 + 2   \,\sum_{t=0}^{\bar{B}-1}c \,\rhoopt\,\alpha^{n+t}\,\norm{\var{x}{n}} \norm{\Delta\mathbf{x}^{n+t}}\\
 	 		%	& \overset{(a)}{\leq} \rhoopt^2\left(1+\bar{B}\epsilon\right)\norm{\var{x}n}^2_2 \!+\! \sum_{t=0}^{\bar{B}-1} \frac{1}{\epsilon}c^{2} (\alpha^{n+t})^2\norm{\Delta \mathbf{x}^{n+t}}^2_2 \!+\! \left(\sum_{t=0}^{\bar{B}-1}c \,\alpha^{n+t}\,\norm{\Delta\mathbf{x}^{n+t}}_2\right)^2\nonumber\\
 	 		& \overset{(a)}{\leq} %\underset{\tilde{\rho}}
 	 		{{\rhoopt^2\left(1+\bar{B}\epsilon\right)}}\norm{\var{x}n}^2 + \sum_{t=0}^{\bar{B}-1} \left(\frac{1}{\epsilon}+\bar{B}\right)c^{2}\,(\alpha^{n+t})^2\,\norm{\Delta \mathbf{x}^{n+t}}^2\\
 	 		& \overset{(b)}{\leq} \tilde{\rho}\norm{\var{x}n}^2 + \sum_{t=0}^{\bar{B}-1} \left(\frac{1}{\epsilon}+\bar{B}\right)\,2\,c^{2}\,(\alpha^{n+t})^2\left(\norm{\Deltax^{n+t}}^2 + \norm{\var{x}{n+t}}^2\right),
 	 	\end{split}
 	 \end{align}
	where (a) follows from the Young's inequality, with $\epsilon>0$, and %; (b) follows from 
	the Jensen's inequality; and in (b) we used (\ref{eq:Delta_x_eq}). 
 	Note that, since $\rhoopt <1$,   $\tilde{\rho}=\rhoopt^2\left(1+\bar{B}\epsilon\right)<1$, for all  $\epsilon \in \left(0,\left(1-\rhoopt^2\right)/(\rhoopt^2\bar{B})\right)$.

 {Denote $\tilde{\alpha}^n_{\rm{mx}} \triangleq \max\limits_{k=0,\ldots,\bar{B}-1}{\alpha^{n+k}}$}. Multiplying   \eqref{eq: consensus err x-B column stochastic}   by ${1}/({1-\tilde{\rho}})$ [resp.~${\tilde{\rho}}/({1-\tilde{\rho}})$],  adding $\|\var{x}n\|^2$ (resp.~$\|\var{x}{n+\bar{B}}\|^2$) to both sides, and using the definitions of $\eDeltax^n$ and $\exorth^n$ [cf.~(\ref{eq: energy_consensus})], yield \vspace{-0.2cm}
	\begin{equation}\label{eq:err x backward column stochastic}
	\begin{aligned}
	&\frac{1}{1-\tilde{\rho}}\,\|\var{x}{n+\bar{B}}\|^2 + \|\var{x}n\|^2\\
	&\leq  \frac{\tilde{\rho}}{1-\tilde{\rho}}\norm{\var{x}{n}}^2 + \norm{\var{x}n}^2 + \frac{2\,c^{2}}{1-\tilde{\rho}}\left(\frac{1}{\epsilon} + \bar{B}\right)\left(\eDeltax^n + \left(\tilde{\alpha}_{\rm{mx}}^n\right)^2\exorth^n\right)\\
	&=  \frac{1}{1-\tilde{\rho}}\norm{\var{x}n}^2 + \frac{2\,c^{2}}{1-\tilde{\rho}}\left(\frac{1}{\epsilon} + \bar{B}\right) \left(\eDeltax^n + \left(\tilde{\alpha}_{\rm{mx}}^n\right)^2\exorth^n\right)
	\end{aligned}
	\end{equation}
%	% Checked 
%	{\color{blue} **** analogy for $\mathbf{y}$******
%	\begin{equation}
%	\begin{aligned}
%	&\frac{1}{1-\tilde{\rho}}\norm{\var{y}{n+\bar{B}}}^2_2 + \norm{\var{y}n}^2_2\\
%	&\leq  \frac{\tilde{\rho}}{1-\tilde{\rho}}\norm{\var{y}{n}}^2_2 + \norm{\var{y}n}^2_2 + \frac{2\,c^{2}L_{\rm{mx}}^2\LBphi^{-2}}{1-\tilde{\rho}}\left(\frac{1}{\epsilon} + \bar{B}\right)\left(I\cdot\eDeltax^n + \left(\sqrt{I} + 1 +\tilde{\alpha}_{\rm{mx}}^{n}\sqrt{I}\right)^2\exorth^n\right)\\
%	&=  \frac{1}{1-\tilde{\rho}}\norm{\var{y}n}^2_2 + \frac{2\,c^{2}L_{\rm{mx}}^2\LBphi^{-2}}{1-\tilde{\rho}}\left(\frac{1}{\epsilon} + \bar{B}\right)\left(I\eDeltax^n + \left(\sqrt{I} + 1 +\tilde{\alpha}_{\rm{mx}}^{n}\sqrt{I}\right)^2\exorth^n\right)
%	\end{aligned}
%	\end{equation}	
%	}
	and	\begin{equation}\label{eq:err x forward column stochastic}
	\begin{aligned}
	&\frac{\tilde{\rho}}{1-\tilde{\rho}}\norm{\var{x}{n+\bar{B}}}^2 + \norm{\var{x}{n+\bar{B}}}^2 = \frac{1}{1-\tilde{\rho}}\,\norm{\var{x}{n+\bar{B}}}^2\\
	&\leq  \frac{\tilde{\rho}}{1-\tilde{\rho}}\norm{\var{x}{n}}^2 + \frac{2\,c^2}{1-\tilde{\rho}}\left(\frac{1}{\epsilon} + \bar{B}\right)\left(\eDeltax^n + \left(\tilde{\alpha}_{\rm{mx}}^n\right)^2\exorth^n\right),
	\end{aligned}
	\end{equation}
respectively. 	
%	{\color{blue} **** analogy for $\mathbf{y}$******
%	\begin{equation}
%	\begin{aligned}
%	&\frac{\tilde{\rho}}{1-\tilde{\rho}}\norm{\var{y}{n+\bar{B}}}^2_2 + \norm{\var{y}{n+\bar{B}}}^2_2 = \frac{1}{1-\tilde{\rho}}\,\norm{\var{y}{n+\bar{B}}}^2_2\\
%	&\leq  \frac{\tilde{\rho}}{1-\tilde{\rho}}\norm{\var{y}{n}}^2_2 + \frac{2\,c^2 {L_{\rm{mx}}}^2\,\LBphi^{-2}}{1-\tilde{\rho}}\left(\frac{1}{\epsilon} + \bar{B}\right)\left(I\eDeltax^n + \left(\sqrt{I} + 1 +\tilde{\alpha}_{\rm{mx}}^{n}\sqrt{I}\right)^2\exorth^n\right).
%	\end{aligned}
%	\end{equation}
%	}	

	We write now  $\sum_{k=0}^{\bar{B}-1}\exorth^{n+k}$  as
	\begin{equation}\label{eq:cumsum err x}
	\hspace{-0.2cm}\begin{aligned}
	 \sum_{k=0}^{\bar{B}-1}\exorth^{n+k}  =& \left(\norm{\var{x}{n+2\bar{B}-2}}^2 + 2\norm{\var{x}{n+2\bar{B}-3}}^2 + \cdots + (\bar{B}-1)\norm{\var{x}{n+\bar{B}}}^2\right)\\
	& + \left(\bar{B}\norm{\var{x}{n+\bar{B}-1}}^2 + (\bar{B}-1)\norm{\var{x}{n+\bar{B}-2}}^2 + \cdots + \norm{\var{x}{n}}^2\right).
	\end{aligned}
	\end{equation}
	Using \eqref{eq:err x backward column stochastic} and \eqref{eq:err x forward column stochastic} on the two terms in  \eqref{eq:cumsum err x}, we obtain the following bounds:   
	\begin{align} 
	&\frac{\tilde{\rho}}{1-\tilde{\rho}}\left(\norm{\var{x}{n+2\bar{B}-2}}^2 + 2\norm{\var{x}{n+2\bar{B}-3}}^2 + \cdots + (\bar{B}-1)\norm{\var{x}{n+\bar{B}}}^2\right)\nonumber\\
	& \hphantom{{}\leq{}}+ \left(\norm{\var{x}{n+2\bar{B}-2}}^2 + 2\norm{\var{x}{n+2\bar{B}-3}}^2 + \cdots + (\bar{B}-1)\norm{\var{x}{n+\bar{B}}}^2\right)\nonumber\\
	 	& \leq  \frac{\tilde{\rho}}{1-\tilde{\rho}}\left(\norm{\var{x}{n+\bar{B}-2}}^2 + 2\norm{\var{x}{n+\bar{B}-3}}^2 + \cdots + (\bar{B}-1)\norm{\var{x}{n}}^2\right)\label{eq:bound_E_first_row}\\
	& \hphantom{{}\leq{}}+\frac{2\,c^{2}}{1-\tilde{\rho}}\left(\frac{1}{\epsilon}+\bar{B}\right)\left[\left(\eDeltax^{n+\bar{B}-2} +\left(\tilde{\alpha}_{\rm{mx}}^{n+\bar{B}-2}\right)^2\exorth^{n+\bar{B}-2}\right)+\cdots \right.\nonumber\\
	&\left.\hspace{4cm} + \left(\bar{B} - 1\right)\left(\eDeltax^{n} + \left(\tilde{\alpha}_{\rm{mx}}^n\right)^2\exorth^{n}\right)\right],\nonumber%\\
	%& \hphantom{{}\leq{}}+ \frac{1}{1-\tilde{\rho}}\left(\bar{B}\norm{\var{x}{n+\bar{B}-1}}^2 + (\bar{B}-1)\norm{\var{x}{n+\bar{B}-2}}^2 + \cdots + \norm{\var{x}{n}}^2\right)\\
	%&\hphantom{{}\leq{}}+\frac{2c^{2}}{1-\tilde{\rho}}\left(\frac{1}{\epsilon}+\bar{B}\right)\left(\bar{B}\left(\eDeltax^{n+\bar{B}-1}+\left(\alpha_{\rm{mx}}^n\right)^2\exorth^{n+\bar{B}-1}\right)+\cdots\right.\\
	% &\left.\hspace{3.6cm}+\left(\eDeltax^{n} + \left(\alpha_{\rm{mx}}^n\right)^2\exorth^{n}\right)\right).
	\end{align}
	and \vspace{-0.2cm}
\begin{align}
	%&\hphantom{{}\leq{}}\frac{\tilde{\rho}}{1-\tilde{\rho}}\left(\norm{\var{x}{n+2\bar{B}-2}}^2 + 2\norm{\var{x}{n+2\bar{B}-3}}^2 + \cdots + (\bar{B}-1)\norm{\var{x}{n+\bar{B}}}^2\right)\\
	%& \hphantom{{}\leq{}}+ \left(\norm{\var{x}{n+2\bar{B}-2}}^2 + 2\norm{\var{x}{n+2\bar{B}-3}}^2 + \cdots + (\bar{B}-1)\norm{\var{x}{n+\bar{B}}}^2\right)\\
	&   \frac{1}{1-\tilde{\rho}}\left(\bar{B}\norm{\var{x}{n+2\bar{B}-1}}^2 + (\bar{B}-1)\norm{\var{x}{n+2\bar{B}-2}}^2 + \cdots + \norm{\var{x}{n+\bar{B}}}^2\right)\nonumber\\
	& \hphantom{{}\leq{}}+ \left(\bar{B}\norm{\var{x}{n+\bar{B}-1}}^2 + (\bar{B}-1)\norm{\var{x}{n+\bar{B}-2}}^2 + \cdots + \norm{\var{x}{n}}^2\right)\nonumber\\
	& \leq {  \frac{1}{1-\tilde{\rho}}\left(\bar{B}\norm{\var{x}{n+\bar{B}-1}}^2 + (\bar{B}-1)\norm{\var{x}{n+\bar{B}-2}}^2 + \cdots + \norm{\var{x}{n}}^2\right)\label{eq:bound_E_second_row}}\\
	& \hphantom{{}\leq{}} { +\frac{2\,c^{2}}{1-\tilde{\rho}}\left(\frac{1}{\epsilon}+\bar{B}\right)\left[\bar{B}\left(\eDeltax^{n+\bar{B}-1} +\left(\tilde{\alpha}_{\rm{mx}}^{n+\bar{B}-1}\right)^2\exorth^{n+\bar{B}-1}\right)+\cdots \right.}\nonumber\\
	&{ \left.\hspace{4cm} +\left(\eDeltax^{n} + \left(\tilde{\alpha}_{\rm{mx}}^n\right)^2\exorth^{n}\right)\right].}\nonumber%\\
	%& \hphantom{{}\leq{}}+ \frac{1}{1-\tilde{\rho}}\left(\bar{B}\norm{\var{x}{n+\bar{B}-1}}^2 + (\bar{B}-1)\norm{\var{x}{n+\bar{B}-2}}^2 + \cdots + \norm{\var{x}{n}}^2\right)\\
	%&\hphantom{{}\leq{}}+\frac{2c^{2}}{1-\tilde{\rho}}\left(\frac{1}{\epsilon}+\bar{B}\right)\left(\bar{B}\left(\eDeltax^{n+\bar{B}-1}+\left(\alpha_{\rm{mx}}^n\right)^2\exorth^{n+\bar{B}-1}\right)+\cdots\right.\\
	% &\left.\hspace{3.6cm}+\left(\eDeltax^{n} + \left(\tilde{\alpha}_{\rm{mx}}^n\right)^2\exorth^{n}\right)\right).
	\end{align}

	Summing (\ref{eq:bound_E_first_row}) and (\ref{eq:bound_E_second_row}) and rearranging terms while using (\ref{eq:cumsum err x}), it is not difficult to check that 
	\begin{equation}
	\begin{aligned}
	&\hphantom{{}\leq{}}\sum_{k=0}^{\bar{B}-1}\frac{k+1+(\bar{B}-k-1)\tilde{\rho}}{1-\tilde{\rho}}\norm{\var{x}{n+\bar{B}+k}}^2 + \sum_{k=0}^{\bar{B}-1}\exorth^{n+k}\\
	&\leq \sum_{k=0}^{\bar{B}-1}\frac{k+1+(\bar{B}-k-1)\tilde{\rho}}{1-\tilde{\rho}}\norm{\var{x}{n+k}}^2\\
	& \hphantom{{}\leq{}}+ \left(\frac{1}{\epsilon}+\bar{B}\right)\frac{2\bar{B}c^{2}}{1-\tilde{\rho}}\sum_{k=0}^{\bar{B}-1}\eDeltax^{n+k}+ \left(\frac{1}{\epsilon}+\bar{B}\right)\frac{2\bar{B}c^{2}}{1-\tilde{\rho}}\left(\alpha_{\rm{mx}}^n\right)^2\sum_{k=0}^{\bar{B}-1}\exorth^{n+k},
	\end{aligned}
	\end{equation}
which leads to the desired result \eqref{eq:V-x column stochastic}. \qed
\end{proof}

We use now Proposition~\ref{lem:V-consensus err column stochastic} in conjunction with  {Lemma~\ref{B-martingale}} to prove the  summability of $\{\exorth^n\}_{n\in \mathbb{N}_+}$ and  $\{\eyorth^n\}_{n\in \mathbb{N}_+}$, under that of $\{\eDeltax^{n}\}_{n\in \mathbb{N}_+}$.  Let %$\alpha_{\rm{mx}}$ be 
 {\begin{equation}\label{eq:alpha upperbound}
\alpha_{\rm{mx}} \triangleq \sigma \cdot \sqrt{\frac{1-\tilde{\rho}}{2\bar{B}\left(\bar{B}+\epsilon^{-1}\right)c^{2}}}\end{equation}}
with $\sigma \in (0,1)$. This implies
  [recall the definition of $\mu^n$ in \eqref{eq:V-x column stochastic}]
\begin{equation}\label{mu_lower}
\mu^n\geq \mu_{\min} \triangleq \left(1- \left(\epsilon^{-1}+\bar{B}\right)\,\frac{2\bar{B}c^{2}}{1-\tilde{\rho}} \,\alpha_{\rm{mx}}^2\right) = 1-\sigma^2>0,\quad \forall \alpha^n_{\rm{mx}}\leq \alpha_{\rm{mx}}.
\end{equation}
\begin{prop}\label{prop:summable consensus err column stochastic}
Suppose that   i) $\sum_{n=0}^{\infty}(\alpha^{n})^2\|\Deltax^{n}\|^2 <\infty$; and ii) $\alpha^{n} \leq \alpha_{\rm{mx}}$, for all  but finite $n\in\mathbb{N}_+$.
Then, the consensus and tracking disagreements %$\|\var{x}n\|_2$ and $\|\var{y}n\|_2$ 
%are square summable, that is, 
satisfy $\sum_{n=0}^{\infty}\|\var{x}n\|^2<\infty$ and $\sum_{n=0}^{\infty}\|\var{y}n\|^2<\infty$, respectively.  \end{prop}
\begin{proof}
%By the definition of $\exorth^{n}$ and  $\eyorth^n$   [cf.~\eqref{eq: energy_consensus}], 
It follows from \eqref{eq: energy_consensus} that it is sufficient to prove $\sum_{n=0}^{\infty}\exorth^{n}<\infty$ (for  $\sum_{n=0}^{\infty}\|\var{x}n\|^2<\infty$) and  $\sum_{n=0}^{\infty}\eyorth^n<\infty$ (for $\sum_{n=0}^{\infty}\|\var{y}n\|^2$ $<\infty$). We prove next only the former result. 

By Assumption~\ref{assumption:alpha} and \eqref{mu_lower},   there exists a sufficiently large $n$, say $\bar{n}$, such that  $\mu^{n}\geq \mu_{\min}> 0$, for all $n\geq \bar{n}$. We assume, without loss of generality, that $\bar{n} = 0$. %, since   $\sum_{k=0}^{\bar{n}-1}\|\var{x}k\|^2_2$ is finite.  
Applying Lemma~\ref{B-martingale} to   \eqref{eq:V-x column stochastic} [cf.~Proposition~\ref{lem:V-consensus err column stochastic}], we have 
%From Eq. \eqref{eq:V-x column stochastic} we can see $ \sum_{j=0}^{\infty}\sum_{k=0}^{\bar{B}-1}\exorth^{j\bar{B}+k} < +\infty$ if
%\begin{equation*}
$\sum_{n=0}^{\infty}\eDeltax^{n}<+\infty$  $\Longrightarrow$    $\sum_{n=0}^{\infty}\exorth^{n} < +\infty$.
%\end{equation*}
It is then sufficient to prove that %The rest of the proof consists then in proving that  %$\sum_{n=0}^{\infty}\eDeltax^{n}<+\infty$ if $\sum_{n=0}^{\infty}(\alpha^{n})^2\|\Deltax^{n}\|^2_2 <\infty$.
%\begin{equation*}
$\sum_{n=0}^{\infty}(\alpha^{n})^2\|\Deltax^{n}\|^2 <\infty\quad \Longrightarrow \quad \sum_{n=0}^{\infty}\eDeltax^{n}<+\infty$.
%\end{equation*}
This comes readily from the following chain of inequalities:\vspace{-0.2cm}
$$\sum_{k=0}^{n}\eDeltax^{k} 
=\sum_{k=0}^{n}\sum_{t=0}^{\bar{B}-1}\left(\alpha^{k+t}\right)^{2}\norm{\Deltax^{k+t}}^2 \leq \bar{B}\sum_{k=0}^{n+\bar{B}-1}\left(\alpha^{k}\right)^{2}\norm{\Deltax^{k}}^{2}.
$$ \qed\vspace{-0.3cm}
 \end{proof} 
\subsubsection{Bounding term iv in \eqref{eq:V-U part column stochastic}}\label{sec:bound_errors}

%In this section, we build the Lyapunov function based on Lemma~\ref{lem:V-U part} and Proposition~\ref{prop:summable consensus err column stochastic}. With the result of  
%Building  on Proposition~\ref{prop:summable consensus err column stochastic}, 
We are now ready to bound  term iv in \eqref{eq:V-U part column stochastic}, as stated next.
\begin{prop}\label{prop_bound_errors}
Suppose that  $\alpha^{n} \leq \alpha_{\rm{mx}}$,
%, for all but  finite $n\in\mathbb{N}_+$.
  then 
\begin{align}\label{eq:V-consensus-err}
&\epsilon_{y}^{-1}	\weyorth{n+\bar{B}} \nonumber\\
&+ \frac{1}{\mu_{\min}}\left(c_L\epsilon_{x}^{-1}+\epsilon_y^{-1}\corth\left(2+\alpha_{\rm{mx} }  \right)^2\right)\wexorth{n+\bar{B}} \nonumber\\
{}\leq{} &\epsilon_{y}^{-1}	\weyorth{n} -\underbrace{ \sum_{k=0}^{\bar{B}-1}\left(c_{L}\epsilon_{x}^{-1}\exorth^{n+k} + \epsilon_{y}^{-1}\eyorth^{n+k}\right)}_{\rm{term\, iv}}\nonumber\\
&+ \frac{1}{\mu_{\min}}\left(c_L\epsilon_{x}^{-1}+\epsilon_y^{-1}\corth\left(2+\alpha_{\rm{mx} }\right)^2\right)\wexorth{n} \nonumber\\
&  + \left(\left(\epsilon_y^{-1}\corth\left(2+\alpha_{\rm{mx} }  \right)^2  +c_L\epsilon_x^{-1}\right)\dfrac{\cDeltax}{{\mu_{\min}}} + \epsilon_y^{-1}\corth\right) \sum_{k=0}^{\bar{B}-1}\eDeltax^{n+k} 
\end{align}
\end{prop}
\begin{proof}
Multiplying \eqref{eq:V-y column stochastic} by $\epsilon_y^{-1}$ on both sides we have
	\begin{align}\label{eq:V-y-1}
& \epsilon_y^{-1}\weyorth{n+\bar{B}} \nonumber\\
	 &  \leq {}   \epsilon_y^{-1}\weyorth{n} \nonumber\\
	&  \quad  - \epsilon_y^{-1} \sum_{k=0}^{\bar{B}-1}\eyorth^{n+k} + \epsilon_y^{-1}\corth\left(2+\alpha_{\rm{mx} }^n  \right)^2 \sum_{k=0}^{\bar{B}-1}\exorth^{n+k}
+ \epsilon_y^{-1}\corth\sum_{k=0}^{\bar{B}-1}\eDeltax^{n+k}\nonumber\\
 &	={} \epsilon_y^{-1}\weyorth{n} \\
	&\quad - \epsilon_y^{-1}\sum_{k=0}^{\bar{B}-1}\eyorth^{n+k} - c_L\epsilon_x^{-1}\sum_{k=0}^{\bar{B}-1}\exorth^{n+k} \nonumber\\
	&\quad + \left(\epsilon_y^{-1}\corth\left(2+\alpha_{\rm{mx} }^n \right)^2 +c_L\epsilon_x^{-1} \right) \sum_{k=0}^{\bar{B}-1}\exorth^{n+k}
+ \epsilon_y^{-1}\corth\sum_{k=0}^{\bar{B}-1}\eDeltax^{n+k}.\nonumber
	\end{align}

Since  $\alpha^n \leq \alpha_{\rm{mx}}$, we have $\alpha_{\rm{mx}}^n \leq \alpha_{\rm{mx}}$ and $\mu^n \geq \mu_{\min}$. Eq. \eqref{eq:V-x column stochastic}
then implies
\begin{align}%\label{eq:V-x-1}
	& \wexorth{n+\bar{B}}\nonumber\\
	%\leq {}& \wexorth{n}  -\mu^n\sum_{k=0}^{\bar{B}-1}\exorth^{n+k} 
	%+ \cDeltax \sum_{k=0}^{\bar{B}-1}\eDeltax^{n+k}	\\
	&\leq {} \wexorth{n}  -\mu_{\min}\sum_{k=0}^{\bar{B}-1}\exorth^{n+k} 
	+ \cDeltax \sum_{k=0}^{\bar{B}-1}\eDeltax^{n+k}.\nonumber 
\end{align}
Multiplying both sides of the above inequality by $(\epsilon_y^{-1}\corth\left(2+\alpha_{\rm{mx} }  \right)^2 +c_L\epsilon_x^{-1})/\mu_{\min}$   and using the fact that $\alpha_{\rm{mx}}^n \leq \alpha_{\rm{mx}}$, we have
\begin{align}\label{eq:V-x-2}
&\frac{1}{\mu_{\min}}\left(c_L\epsilon_x^{-1}+\epsilon_y^{-1}\corth\left(2+\alpha_{\rm{mx} }  \right)^2\right) \wexorth{n+\bar{B}}\nonumber\\
\leq {}&\frac{1}{\mu_{\min}}\left(c_L\epsilon_x^{-1}+\epsilon_y^{-1}\corth\left(2+\alpha_{\rm{mx} }  \right)^2\right) \wexorth{n}\nonumber\\
& - \left(c_L\epsilon_x^{-1}+\epsilon_y^{-1}\corth\left(2+\alpha_{\rm{mx} }  \right)^2\right)\sum_{k=0}^{\bar{B}-1}\exorth^{n+k}\nonumber\\ & + \left(c_L\epsilon_x^{-1}+\epsilon_y^{-1}\corth\left(2+\alpha_{\rm{mx} }  \right)^2\right)\dfrac{\cDeltax}{ {\mu_{\min}}} \sum_{k=0}^{\bar{B}-1}\eDeltax^{n+k}\vspace{-0.3cm}
 %\leq {}& \frac{1}{\mu_{\min}}\left(c_L\epsilon_x^{-1}+\epsilon_y^{-1}\corth\alpha_{\rm{mx}}^2\right) \wexorth{n}
%& - \left(c_L\epsilon_x^{-1}+\epsilon_y^{-1}\corth\alpha_{\rm{mx}}^2\right)\sum_{k=0}^{\bar{B}-1}\exorth^{n+k} + \left(c_L\epsilon_x^{-1}+\epsilon_y^{-1}\corth\alpha_{\rm{mx}}^2\right)\dfrac{\cDeltax}{ {\mu_{\min}}}  \sum_{k=0}^{\bar{B}-1}\eDeltax^{n+k}.\nonumber
\end{align}

Adding \eqref{eq:V-x-2}  to \eqref{eq:V-y-1} leads to the desired result.\qed
\end{proof}
\subsection{{Lyapunov}-like function and its descent properties}\label{sec:proof th1}

We are now in the position to construct a  function whose descent properties (every $\bar{B}$ iterations)  will used to prove Theorem~\ref{thm:convergence}. Because of that, we will refer to such a function as Lyapunov-like function. 

Adding   \eqref{eq:V-U part column stochastic} and \eqref{eq:V-consensus-err} (multiplied by ${\UBphi}/{2}$),   yields \vspace{-0.2cm} %\begin{subequations} 
\begin{align}
	&\sum_{k=0}^{\bar{B}-1}U\left(\wavg{x}{n+\bar{B}+k}\right) + \frac{\UBphi}{2}\epsilon_{y}^{-1}
	\weyorth{n+\bar{B}}\nonumber\\
	&\hphantom{{}\leq{}}+ \frac{\UBphi}{2\,\mu_{\min}}\left(c_L\epsilon_{x}^{-1}+\epsilon_y^{-1}\corth\left(2+\alpha_{\rm{mx} }  \right)^2\right)\wexorth{n+\bar{B}}\nonumber\\
	&\leq \sum_{k=0}^{\bar{B}-1}U\left(\wavg{x}{n+k}\right) + \frac{\UBphi}{2}\epsilon_{y}^{-1}\weyorth{n}\nonumber \\
	& \hphantom{{}\leq{}} + \frac{\UBphi}{2\,\mu_{\min}}\left(c_L\epsilon_{x}^{-1}+\epsilon_y^{-1}\corth\left(2+\alpha_{\rm{mx} }  \right)^2\right) \wexorth{n} \nonumber\\
	&\hphantom{{}\leq{}} - \frac{c_\tau}{I}\,\LBphi\sum_{k=0}^{\bar{B}-1}\sum_{t=0}^{\bar{B}-1}\alpha^{n+k+t}\norm{\Deltax^{n+k+t}}^2\nonumber\\
	&\hphantom{{}\leq{}} +\frac{\UBphi}{2} \left(\frac{ {L+L_G}}{I}\cdot\UBphi+c_{L}\epsilon_{x}+\epsilon_{y} + \epsilon_y^{-1}\corth \right)\sum_{k=0}^{\bar{B}-1}\eDeltax^{n+k}\nonumber\\
	&\hphantom{{}\leq{}}+ \frac{\UBphi}{2\mu_{\min}}\left(c_L\epsilon_{x}^{-1}+\epsilon_y^{-1}\corth\left(2+\alpha_{\rm{mx} } \right)^2\right)\cDeltax\sum_{k=0}^{\bar{B}-1}\eDeltax^{n+k}. \label{eq:V-2 column stochastic}
\end{align}
%\end{subequations}
Define 
\begin{equation}\label{Lyapunov_function}
\begin{aligned}
V^{n} \triangleq &\sum_{k=0}^{\bar{B}-1}U\left(\wavg{x}{n+k}\right) + \frac{\UBphi}{2}\epsilon_{y}^{-1}\weyorth{n} \\
&+  \frac{\UBphi}{2\mu_{\min}}\left(c_L\epsilon_{x}^{-1}+\epsilon_y^{-1}\corth\left(2+\alpha_{\rm{mx} }  \right)^2\right)\wexorth{n},
\end{aligned}\vspace{-0.2cm}
\end{equation}
and 
\begin{multline}\label{beta}
\beta^{n} \triangleq \frac{c_\tau}{I}\LBphi - \frac{\UBphi}{2}\alpha^n\left(\frac{ {L+L_G}}{I}\cdot\UBphi+c_{L}\epsilon_{x}+\epsilon_{y} + \epsilon_y^{-1}\corth \right.\\
+\left.\frac{\cDeltax}{\mu_{\min}}\left(c_L\epsilon_{x}^{-1}+\epsilon_y^{-1}\corth\left(2+\alpha_{\rm{mx} }  \right)^2\right)\right).
\end{multline}
Substituting (\ref{Lyapunov_function}) and (\ref{beta}) in \eqref{eq:V-2 column stochastic}, we obtain the desired descent property of $V^n$: for sufficiently large $n$, it holds 
\begin{equation}\label{eq:V column stochastic}
V^{n+\bar{B}} \leq V^{n} -\sum_{k=0}^{\bar{B}-1}\sum_{t=0}^{\bar{B}-1} \beta^{n+k+t}\alpha^{n+k+t}\norm{\Deltax^{n+k+t}}^2.
\end{equation}

\subsection{Proof of Theorem~\ref{thm:convergence}}\label{sec:final_part_of_the_proof}\vspace{-0.3cm}
The proof consists in two steps, namely: 
\begin{itemize}
\item[$-$]\textbf{Step~1}: Leveraging the descent property of the  {Lyapunov}-like  function, we first  show that $\lim\limits_{n\to\infty}\|\Deltax^{n}\| = 0$, either using a diminishing or constant step-size $\alpha^n$ (satisfying Assumption \ref{assumption:alpha}); and
\item[$-$]\textbf{Step~2}: Using the results in Step~1, we conclude the proof  showing that i)  {$\lim_{n \to \infty} D(\mathbf{x}^n) = 0$ and ii) $\lim_{n \to \infty} J(\bar{\mathbf{x}}^n) = 0$ }
\end{itemize}\vspace{-0.2cm}

\subsubsection{Step~1: $\lim\limits_{n\to\infty}\|\Deltax^{n}\| = 0$}\vspace{-0.2cm}

 Let us  distinguish the two choices of step-size, namely: $\alpha^{n}$ is constant  (satisfying Assumption \ref{assumption:alpha}.1);  or $\alpha^{n}$ is diminishing (satisfying Assumption~\ref{assumption:alpha}.2).\vspace{1ex}\\
\noindent\textbf{Case 1: constant step-size.}
Set $\alpha^{n}\equiv \alpha$ for all $n\in\mathbb{N}_+$. To obtain the desired descent on $V^n$ [cf.~\eqref{eq:V column stochastic}],  $\alpha$ has to be chosen so that $\beta^{n}={\beta}>0$ [cf.~\eqref{beta}]. We show next that if  $\alpha$ satisfies \eqref{constant_step} [cf.~Assumption~\ref{assumption:alpha}.2], then ${\beta}>0$.

Recall that \eqref{eq:V column stochastic} holds under the assumption that $\alpha \leq \alpha_{\rm{mx}}$, with $\alpha_{\rm{mx}}$ defined in (\ref{eq:alpha upperbound}). %\begin{equation}\label{eq:alpha bound 1 column stochastic}
%\alpha \leq \alpha_{\rm{mx}}  \triangleq \sigma\sqrt{\frac{1-\tilde{\rho}}{2c^{2}\bar{B}\left(\bar{B}+\epsilon^{-1}\right)}},
%\end{equation}
%under which we have $\mu_{\min} = 1-\sigma^2$ [cf.~(\ref{mu_lower})].  
%Let us meet this condition setting 
%\begin{equation}\label{eq:alpha bound 1 column stochastic}
%\alpha \leq \dfrac{\alpha_{\rm{mx}}}{\textcolor{red}{\sqrt{2}}}  = \sqrt{\frac{1-\tilde{\rho}}{4c^{2}\bar{B}\left(\bar{B}+\epsilon^{-1}\right)}},
%\end{equation}
%so that $\mu_{\min}\geq 1/2$. 
 Substituting the expressions of $\alpha_{\rm{mx}} $ and $\mu_{\min} = 1-\sigma^2$ [cf.~(\ref{mu_lower})]    in \eqref{beta} and  using the definitions of $c_{\Delta}$ and  $c_{\perp}$ [cf.~Proposition~\ref{lem:V-consensus err column stochastic}], one can check that $\beta^{n}={\beta}>0$   [cf.~\eqref{beta}] if, in addition to $\alpha \leq \alpha_{\rm{mx}} $,  $\alpha$ satisfies  also
%Eq. \eqref{eq:V column stochastic} and  $\liminf\limits_{n\to\infty}V^{n}>-\infty$ (since $U$ is bounded below) then indicates that  $\sum_{n=0}^{\infty}\norm{\Delta \wavg{x}{n}}^{2}<\infty$ if $\beta^{n}>0$. 
%Using the bound given in \eqref{eq:alpha bound 1 column stochastic} the condition  then simplifies to 
\begin{equation}\label{bound_2_alpha}
\alpha \leq \frac{2\,c_\tau\,\LBphi}{I\UBphi} \left(\frac{L+L_G}{I}\cdot\UBphi+c_{L}\epsilon_{x} +\epsilon_{y}+\frac{\cDeltax}{1-\sigma^2}\left(c_L\epsilon_{x}^{-1}
	+ {9}\,\corth  \epsilon_{y}^{-1}\right) \right)^{-1},
\end{equation}
where $\epsilon_{x},\epsilon_{y}>0$ are free parameters. The above upperbound is maximized by
\begin{equation*}
\begin{aligned}
\epsilon_{x} &= \sqrt{\frac{\cDeltax}{1-\sigma^2}} = c\,\sqrt{\frac{2\bar{B}\left(\epsilon^{-1}+\bar{B}\right)}{(1-\tilde{\rho})(1-\sigma^2)}}  \\
\epsilon_y &=\sqrt{\frac{{9}\,\corth\cDeltax}{1-\sigma^2}} = {6 \,L_{\rm{mx}} \,\LBphi^{-1}\,(\epsilon^{-1} +\bar{B})\,\frac{\bar{B}c^2}{1-\tilde{\rho}}\sqrt{\frac{1}{1-\sigma^2}}}
\end{aligned}
\end{equation*}
Combining $\alpha \leq \alpha_{\rm{mx}} $ and (\ref{bound_2_alpha}), we get the following bound for $\alpha$:
%\begin{multline}\label{bound_3_alpha}
%	\alpha \leq \min\left\{\sigma\sqrt{\frac{1-\tilde{\rho}}{2c^{2}\bar{B}\left(\bar{B}+\epsilon^{-1}\right)}},\right.\\
%	\left.\frac{2c_\tau\LBphi}{I\UBphi}\left(\frac{L+L_G}{I}\cdot\UBphi+2c\left(c_{L}+L_{\rm{mx}}\LBphi^{-1}\right)\sqrt{\frac{2\bar{B}\left(\epsilon^{-1}+\bar{B}\right)}{(1-\tilde{\rho})(1-\sigma^2)}}\right)^{-1}\right\},
%\end{multline}
%{\color{blue} ****** correction ******
\begin{multline}\label{bound_3_alpha}
	\alpha \leq \min\left\{\sigma\sqrt{\frac{1-\tilde{\rho}}{2c^{2}\bar{B}\left(\bar{B}+\epsilon^{-1}\right)}},
	\frac{2c_\tau\LBphi}{I\UBphi}\left(\frac{L+L_G}{I}\cdot\UBphi+2c \cdot c_{L}\sqrt{\frac{2\bar{B}\left(\epsilon^{-1}+\bar{B}\right)}{(1-\tilde{\rho})(1-\sigma^2)}} \right.\right.\\
	\left.\left. +12 \,L_{\rm{mx}} \,\LBphi^{-1}\,(\epsilon^{-1} +\bar{B})\frac{\bar{B}c^2}{1-\tilde{\rho}}\sqrt{\frac{1}{1-\sigma^2}} \right)^{-1}\right\},\vspace{-0.2cm}
\end{multline}
%}
where recall that $\epsilon<({1-\rhoopt^2})/({\rhoopt^2\bar{B}})$ [cf.~Proposition \ref{lem:V-consensus err column stochastic}]. 
Since $({1-\tilde{\rho}})/({\epsilon^{-1}+\bar{B}})$  is maximized by $\epsilon = ({1-\rhoopt})/({\rhoopt \cdot\bar{B}})$ with the corresponding value being ${\left(1-\rhoopt\right)^2}/{\bar{B}}$, we obtain from (\ref{bound_3_alpha}) the   final  bound \eqref{constant_step}. %for   $\alpha$:
%\begin{equation}\label{eq:alpha upperbound column stochastic}
%	\alpha \leq\min\left\{\sigma\cdot\frac{1-\rho}{\sqrt{2}c\bar{B}},\frac{2c_\tau\LBphi}{I\UBphi}\left(\frac{L+L_G}{I}\cdot\UBphi+2c\left(c_{L}+L_{\rm{mx}}\LBphi^{-1}\right) \frac{\bar{B}}{1-\rho}\sqrt{\frac{2}{1-\sigma^2}}\right)^{-1}\right\}.
%\end{equation}
%{\color{blue} ******** correction ******
%\begin{multline}\label{eq:alpha upperbound column stochastic}
	%\alpha \leq\min\left\{\sigma\cdot\frac{1-\rho}{\sqrt{2}c\bar{B}},\right.\\\left.\frac{2c_\tau\LBphi}{I\UBphi}\left(\frac{L+L_G}{I}\cdot\UBphi+2c\left(c_{L}+6L_{\rm{mx}}\LBphi^{-1}\frac{\bar{B}}{1-\rho}\right) \frac{\bar{B}}{1-\rho}\sqrt{\frac{2}{1-\sigma^2}}\right)^{-1}\right\}.
%\end{multline}
%}

Under \eqref{constant_step}, using
\eqref{eq:V column stochastic} and Lemma~\ref{B-martingale}  (recall that  $\liminf\limits_{n\to\infty}V^{n}>-\infty$, since $U$ is bounded from below on $\mathcal K$)  we get  %$$\sum_{n=0}^{\infty}\|\Deltax^{n}\|^{2}_2<\infty\quad  \Longrightarrow\quad  \lim\limits_{n\to\infty}\|\Deltax^{n}\|_2= 0,$$
$\lim\limits_{n\to\infty}\|\Deltax^{n}\|= 0$ and, by Proposition~\ref{prop:summable consensus err column stochastic},
%By Corollary~\ref{prop:summable consensus err column stochastic} we conclude that 
$\lim\limits_{n\to\infty}\|\var{x}{n}\| = 0$ and $\lim\limits_{n\to\infty}\|\var{y}{n}\| = 0$.
\vspace{2ex}\\
\noindent\textbf{Case 2: diminishing step-size.}
Since $\alpha^n$ is diminishing, there exists a sufficiently large $n_{2}$ so that $\beta^{n} \geq \beta >0$ for all $n\geq n_{2}$, 
implying
\begin{equation}\label{alpha_series_convergent}
\sum_{n=0}^{\infty}\sum_{t=0}^{\bar{B}-1}\alpha^{n+t}\norm{\Deltax^{n+t}}^2<\infty,
\end{equation}
which together with $\sum_{n=0}^{\infty}\alpha^n=\infty$ and  Proposition~\ref{prop:summable consensus err column stochastic} yield 
%\begin{eqnarray}
\begin{gather}
	\liminf\limits_{n\to\infty}\norm{\Deltax^{n}}= 0;\label{lim_inf}\\
 \lim\limits_{n\to\infty}\norm{\var{x}n}=0\quad \text{and}\quad  \lim\limits_{n\to\infty}\norm{\var{y}n}=0\label{lim_inf_2}.
\end{gather}
% \end{eqnarray}

We prove next that $\limsup\limits_{n\to\infty}\|\Deltax^{n}\|= 0$, which together with \eqref{lim_inf} implies $\lim_{n\to\infty}\|\Deltax^{n}\|= 0$. 
Suppose  that $\limsup\limits_{n\to\infty}\|\Deltax^{n}\|>0$. This, together with  $\liminf\limits_{n\to\infty}\|\Deltax^{n}\|=0$, implies that there  exists an infinite set of indices $\mathcal{N}$ such that for all $n\in\mathcal{N}$,  one can find an integer $i_{n}>n$ such that:
\begin{gather}
\norm{\Deltax^{n}}<\eta,\quad  \|\Deltax^{i_{n}}\|>2\eta\\
\eta \leq \|\Deltax^{j}\| \leq 2\eta,\quad   n<j<i_{n}.\label{lim_sup_contr_2}
\end{gather}  Denote $\widehat{\mathbf{x}}_{(i)}^n\triangleq \widehat{\mathbf{x}}_{i}(\mathbf{x}^{n}_{(i)})$ and $\widehat{\mathbf{x}}^n\triangleq [\widehat{\mathbf{x}}_{(1)}^{n\top},\ldots, \widehat{\mathbf{x}}_{(I)}^{n\top}]^\top$. We have:
\begin{align}\label{eq:limsup diff delta x column stochastic}
\eta & \leq \norm{\Deltax^{i_{n}}}-\norm{\Deltax^{n}} \leq  \norm{\Deltax^{i_{n}} -\Deltax^{n}}\nonumber\\
& \leq \norm{\widetilde{\mathbf{x}}^{i_{n}} -  \widetilde{\mathbf{x}}^{n}} + \norm{ \Vwavg{x}{i_{n}}- \Vwavg{x}{n}}\nonumber\\
& \leq \norm{\widehat{\mathbf{x}}^{i_{n}} -\widehat{\mathbf{x}}^{n}} + \underset{e_{1}^{n}}{\underbrace{\norm{\widetilde{\mathbf{x}}^{i_{n}}  - \widehat{\mathbf{x}}^{i_{n}}} + \norm{\widetilde{\mathbf{x}}^{n}  - \widehat{\mathbf{x}}^{n}}}} + \norm{\Vwavg{x}{i_{n}}-\Vwavg{x}{n}}\nonumber\\
& \overset{(a)}{\leq} \hat{L}\norm{\mathbf{x}^{i_{n}}-\mathbf{x}^{n}} + \norm{\Vwavg{x}{i_{n}}-\Vwavg{x}{n}}+ e_{1}^{n}\nonumber\\
& \leq \hat{L}\left(\norm{\mathbf{x}^{i_{n}} - \mathbf{J}_{\boldsymbol{\phi}^{i_{n}} }\mathbf{x}^{i_{n}}} + \norm{\mathbf{x}^{n} - \mathbf{J}_{\boldsymbol{\phi}^{n}} \mathbf{x}^{n}} 
 + \sqrt{I}\norm{\wavg{x}{i_{n}}-\wavg{x}{n}}\right) \nonumber \\
 &\hspace{1em}+ \sqrt{I}\norm{\wavg{x}{i_{n}}-\wavg{x}{n}} + e_{1}^{n}\nonumber\\
& \leq \left(\hat{L} + 1\right)\sqrt{I}\norm{\wavg{x}{i_{n}}-\wavg{x}{n}} + \underset{e_{2}^{n}}{\underbrace{\hat{L}\,\left(\norm{\var{x}{i_{n}}} + \norm{\var{x}{n}}\right)}} + e_{1}^{n}\nonumber\\
& \overset{(b)}{\leq} \left(\hat{L}+ 1\right)\sqrt{I} 
\sum_{t=n}^{i_{n}-1}\alpha^{t}\norm{\dfrac{1}{I}\left((\boldsymbol{\phi}^{t})^{\top}\otimes \mathbf{I}_{m}\right)\Deltax^t}+e_{2}^{n} + e_{1}^{n}\nonumber\\
%& \leq \left(\hat{L} + 1\right)\sqrt{I}\sum_{t=n}^{i_{n}-1}\alpha^{t}\norm{\Deltax^{t}}_2  + e_{2}^{n}+ e_{1}^{n}\nonumber\\
% If you uncomment the previous step, the inequality below must become equality
& \leq    \left(\hat{L} + 1\right)\sqrt{I} \sum_{t=n + 1}^{i_{n}-1}\alpha^{t}\norm{\Deltax^{t}} + \underbrace{(\hat{L} + 1)\sqrt{I} \alpha^{n}\norm{\Deltax^{n}}}_{e_3^n} + e_{2}^{n}+ e_{1}^{n}\nonumber\\
& \overset{(c)}{\leq} (\hat{L} + 1)\sqrt{I}\,\eta^{-1}\,\sum_{t=n + 1}^{i_{n}-1}\alpha^{t}\norm{\Deltax^{t}}^2 + e_3^n + e_{2}^{n}+ e_{1}^{n},
\end{align}
where in (a) we used  \eqref{BR_Lip} [cf.~Lemma~\ref{lemma_BR_properties}]; (b) follows from (\ref{eq: x weighted ave}); and in (c) we used the lower bound in \eqref{lim_sup_contr_2}.

 Since  i) $\lim\limits_{n\to\infty}\|\var{x}{n}\| = 0$ and $\lim\limits_{n\to\infty}\|\var{y}{n}\| = 0$ [cf.~(\ref{lim_inf_2})]; ii)  $\lim\limits_{n\to\infty}\|\widetilde{\mathbf{x}}^{n}-\widehat{\mathbf{x}}^{n}\| = 0$ [cf.~Lemma~\ref{lem:best-response consistency}]; iii)  and $\sum_{n=0}^{\infty}\sum_{t=0}^{\bar{B}-1}\alpha^{n+t}\|\Deltax^{n+t}\|^2<\infty$ [cf.~(\ref{alpha_series_convergent})],
there exists a sufficiently large $n_3$ such that the right-hand-side of \eqref{eq:limsup diff delta x column stochastic} is less than $\eta$, for all $n>n_3$, which leads to a contradiction. Therefore,   $\limsup\limits_{n\to\infty}\|\Deltax^{n}\|=0$.\vspace{-0.2cm}

\subsubsection{Step~2: $\lim\limits_{n\to\infty} M(\mathbf{x}^n)=0$}
 Recall that in the previous subsection we  proved that i)$ \lim\limits_{n\to\infty}\|\Deltax^{n}\|=0$; ii) $\lim\limits_{n\to\infty}\|\var{x}{n}\| = 0$; and iii) $\lim\limits_{n\to\infty}\|\var{y}{n}\| = 0$,
 using either a constant step-size  $\alpha^{n}\equiv \alpha$, with $\alpha$ satisfying \eqref{constant_step}, or a diminishing one. 
 The statement $\lim\limits_{n\to\infty} D(\mathbf{x}^n) = 0$ follows readily from point ii) and \vspace{-0.2cm}
 \begin{align}\label{eq:avg_consensus_x}
 \begin{split}
   \lim_{n\to\infty} \|\mathbf{x}_{(i)}^n - \bar{\mathbf{x}}^n\| \leq {} &  \lim_{n\to\infty} \|\mathbf{x}_{(i)}^n - \wavg{x}{n}\| + \lim_{n\to\infty}\norm{\wavg{x}{n} - \bar{\mathbf{x}}^n} \\
   \leq {} & \lim_{n\to\infty} \|\mathbf{x}_{(i)}^n - \wavg{x}{n}\| + \lim_{n\to\infty}\frac{1}{I} \sum_{j = 1}^I \|  \mathbf{x}_{(j)}^n - \wavg{x}{n}\| 
   = 0. 
 \end{split}
   \end{align} 
   
   Next we show $\lim\limits_{n\to\infty} J(\bar{\mathbf{x}}^n) = 0$. Recall the definition $J(\bar{\mathbf{x}}^n) \triangleq \| \bar{\mathbf{x}} (\bar{\mathbf{x}}^n) - \bar{\mathbf{x}}^n \|$, where for notation simplicity, we set 
   \begin{align}
    \bar{\mathbf{x}} (\bar{\mathbf{x}}^n) \triangleq \argmin_{\mathbf{z} \in \mathcal{K}}~ \left\{\Big(\nabla F(\bar{\mathbf{x}}^n)-\nabla G^{-}\big(\bar{\mathbf{x}}^n\big)\Big)^{\top} (\mathbf{z}-\bar{\mathbf{x}}^n)+ \dfrac{1}{2}\|\mathbf{z}-\bar{\mathbf{x}}^n\|^2+G(\mathbf{z})^+\right\} .
   \end{align}
   Since 
   \begin{align}\label{eq:bound_J}
   J (\bar{\mathbf{x}}^n)  \leq \| \widehat{\mathbf{x}}_i (\bar{\mathbf{x}}^n) - \bar{\mathbf{x}}^n \| + \| \bar{\mathbf{x}} (\bar{\mathbf{x}}^n) - \widehat{\mathbf{x}}_i (\bar{\mathbf{x}}^n)\|,
   \end{align}
   it is sufficient to show that the two terms on the right hand side are asymptotically vanishing, which is proved below.  
   
   \begin{itemize}
   	\item[$\bullet$] 
     $ \lim\limits_{n\to\infty} \|\widehat{\mathbf{x}}_i (\bar{\mathbf{x}}^n) - \bar{\mathbf{x}}^n \| = 0$.  We bound $\|\widehat{\mathbf{x}}_i (\bar{\mathbf{x}}^n) - \bar{\mathbf{x}}^n \|$   as
   \begin{align}\label{eq:piece_1}
   \begin{split}
   \|\widehat{\mathbf{x}}_i (\bar{\mathbf{x}}^n) - \bar{\mathbf{x}}^n \| \leq {} & \|\widehat{\mathbf{x}}_i (\wavg{x}{n}) - \wavg{x}{n} \| + \| \wavg{x}{n} - \bar{\mathbf{x}}^n\| + \|\widehat{\mathbf{x}}_i (\bar{\mathbf{x}}^n) -\widehat{\mathbf{x}}_i (\wavg{x}{n}) \|   \\
  \overset{(a)}{ \leq} {} &  \|\widehat{\mathbf{x}}_i (\wavg{x}{n}) - \wavg{x}{n} \| + (1+\hat{L})\| \wavg{x}{n} - \bar{\mathbf{x}}^n \|,
   \end{split}
   \end{align}
   where (a) follows from Lemma~\ref{lemma_BR_properties}. From~\eqref{eq:avg_consensus_x} we know $\lim\limits_{n\to\infty} \| \wavg{x}{n} - \bar{\mathbf{x}}^n \| = 0$.
   
    To show $\|\widehat{\mathbf{x}}_i (\wavg{x}{n}) - \wavg{x}{n} \|$ is asymptotically vanishing, we bound it as
    \begin{align}\label{eq:piece_2}
    \begin{split}
    \|\widehat{\mathbf{x}}_i (\wavg{x}{n}) - \wavg{x}{n} \| \leq {} & \|  \widehat{\mathbf{x}}_i (\wavg{x}{n}) - \widehat{\mathbf{x}}_{i}(\mathbf{x}^{n}_{(i)})\| 
     +  \|   \widehat{\mathbf{x}}_{i}(\mathbf{x}^{n}_{(i)})- \widetilde{\mathbf{x}}_{(i)}^n\| \\
     &  + \|\widetilde{\mathbf{x}}_{(i)}^n  - \wavg{x}{n}\|    .
    \end{split}
    \end{align}
    The result $\lim\limits_{n\to\infty}     \|\widehat{\mathbf{x}}_i (\wavg{x}{n}) - \wavg{x}{n} \| = 0$ follows from  Lemma~\ref{lemma_BR_properties}, Lemma~\ref{lem:best-response consistency} and points i)-iii).  
    
    From ~\eqref{eq:piece_1} and~\eqref{eq:piece_2} we conclude
    \begin{align}\label{eq:term_I}
    \lim\limits_{n\to\infty} \|\widehat{\mathbf{x}}_i (\bar{\mathbf{x}}^n) - \bar{\mathbf{x}}^n \| = 0.
    \end{align}
    
      \item[$\bullet$]  We prove $ \lim\limits_{n\to\infty} \| \bar{\mathbf{x}} (\bar{\mathbf{x}}^n) - \widehat{\mathbf{x}}_i (\bar{\mathbf{x}}^n) \|= 0$. Using the first order optimality conditions of $\bar{\mathbf{x}} (\bar{\mathbf{x}}^n)$ and $\widehat{\mathbf{x}}_i (\bar{\mathbf{x}}^n)$, we can bound their difference as
       \begin{align}
       \begin{split}
       \| \bar{\mathbf{x}} (\bar{\mathbf{x}}^n) - \widehat{\mathbf{x}}_i (\bar{\mathbf{x}}^n) \| \leq {} & \| \nabla \widetilde{f}_i \big(\widehat{\mathbf{x}}_i (\bar{\mathbf{x}}^n); \bar{\mathbf{x}}^n\big) - \nabla f_i (\bar{\mathbf{x}}^n) - \widehat{\mathbf{x}}_i (\bar{\mathbf{x}}^n) + \bar{\mathbf{x}}^n\| \\
       \leq {} & \|  \nabla \widetilde{f}_i \big(\widehat{\mathbf{x}}_i (\bar{\mathbf{x}}^n); \bar{\mathbf{x}}^n\big) - \nabla f_i (\widehat{\mathbf{x}}_i (\bar{\mathbf{x}}^n))\| \\
       & + \| \nabla f_i(\widehat{\mathbf{x}}_i (\bar{\mathbf{x}}^n)) - \nabla f_i (\bar{\mathbf{x}}^n)\| +\| \widehat{\mathbf{x}}_i (\bar{\mathbf{x}}^n) - \bar{\mathbf{x}}^n \|\\
       \leq {} & (\tilde{L}_i + L_i + 1)\| \widehat{\mathbf{x}}_i (\bar{\mathbf{x}}^n) - \bar{\mathbf{x}}^n \|.
       \end{split}
       \end{align}
       Using~\eqref{eq:term_I} we have 
       \begin{align}\label{eq:term_II}
        \lim\limits_{n\to\infty} \| \bar{\mathbf{x}} (\bar{\mathbf{x}}^n) - \widehat{\mathbf{x}}_i (\bar{\mathbf{x}}^n) \|= 0.
       \end{align}
        \end{itemize}
      The proof is completed just  combining~\eqref{eq:bound_J}, \eqref{eq:term_I} and~\eqref{eq:term_II}.

\input{sec_simulation}\vspace{-0.3cm}
%\section{Conclusions}
%\label{sec:conclusions}

%Some conclusions here.
%\section*{Acknowledgments}
%We would like to acknowledge the assistance of volunteers in putting
%together this example manuscript and supplement.
\section{Appendix}\vspace{-0.1cm}
\appendix
   \vspace{-0.2cm}

 {\section{Proof of Lemma~\ref{cor:averaging matrix column stochastic}}\label{app:proof_lemma_1}\vspace{-0.2cm}
We begin introducing  the following intermediate result. 
 \begin{lem}\label{lem:geo decay row stochastic}
 	In the setting of Lemma~\ref{cor:averaging matrix column stochastic},    the following hold: 
 	\begin{enumerate}
 	\item[(i)]  
The elements of   $\mathbf{A}^{n:0}$, $n\in \mathbb{N}_+$,  can be bounded as\vspace{-0.1cm}
 	\begin{align} \label{eq:def delta}
 	&\inf_{t\in\mathbb{N}_+} \left(\min_{1\leq i\leq I} \left(\mathbf{A}^{t:0}\mathbf{1}\right)_{i}\right)\geq \LBphi,\\
 	&\sup_{t\in\mathbb{N}_+} \left(\max_{1\leq i\leq I} \left(\mathbf{A}^{t:0}\mathbf{1}\right)_{i}\right)\leq  \UBphi\label{eq:def delta_up},
 	\end{align} 
 	where $\phi_{lb}$ and $\phi_{ub}$ are defined in \eqref{consensus-error-bound-pert-cons};
 	\item[(ii)] 
 %	Furthermore, let $\{\mathbf{W}^{n}\}_{n\in\mathbb{N}}$ be a sequence of row stochastic matrices related to  $\{\mathbf{A}^{n}\}_{n\in\mathbb{N}}$ according to Eq.~\eqref{eq:def W}, then
 For any given $n, k\in \mathbb{N}_+$, $n\geq k$,	there exists a stochastic vector $\boldsymbol{\xi}^{k}\triangleq [\xi_1^{k},\ldots \xi_I^{k}]^\top$ $($i.e., $\boldsymbol{\xi}^{k}>  \mathbf{0}$ and $\mathbf{1}^\top \, \boldsymbol{\xi}^{k}=1)$ such that % the matrix difference $\mathbf{W}^{n:k} - \mathbf{1}(\boldsymbol{\xi}^{k})^\top$ decays geometrically:
 	\begin{equation}\label{eq:exp_decay_W}
 	\left\vert\mathbf{W}^{n:k}_{ij}  - \xi_{j}^{k}\right\vert \leq c_{0}\,(\rho)^{\big\lfloor  \frac{n-k+1}{(I-1)B}\big \rfloor},\qquad \forall i,j\in [I],
 	\end{equation}
 	where $c_{0}$ and $\rho$ are defined in (\ref{definitions_kappa_c0_rho_0}). 
 %	where 
 %	\begin{equation}
 %	\begin{aligned}
 %	c_{0} \triangleq &\frac{2\left(1+\tilde{\kappa}^{-\left(I-1\right)B}%\right)}{1-\tilde{\kappa}^{\left(I-1\right)B}}\\
 %	\rho_{0} \triangleq  &\left(1-\tilde{\kappa}^{\left(I-1\right)B}\right)^{\frac{1}{\left(I-1\right)B}}
 %	\end{aligned}
%\end{equation} 	 with $\tilde{\kappa} \triangleq \kappa^{IB+1}/I$.
	\end{enumerate}
 \end{lem}
The proof Lemma~\ref{lem:geo decay row stochastic}  follows similar steps as those in~\cite[Lemma 2, Lemma 4]{nedic2009distributedsubgradient} and thus is omitted, although the results in~\cite{nedic2009distributedsubgradient} are established under a stronger condition on $\mathcal{G}^n$  than  Assumption~\ref{assumption:G}.
 
We prove now   Lemma~\ref{cor:averaging matrix column stochastic}. Let $\mathbf{z}\in \mathbb{R}^{I\cdot m}$ be an arbitrary vector. For each $\ell=1,\ldots ,m$,  define  $\mathbf{z}_{\ell}\triangleq (\mathbf{I}_I \otimes \mathbf{e}_{\ell}^\top)\,\mathbf{z}$, where $\mathbf{e}_{\ell}$ is the $\ell$-th canonical vector; we denote by  ${z}_{\ell,j}$ the $j$-th component of  $\mathbf{z}_{\ell}$, with $j\in [I]$. %For any given   $\mathbf{z}\in \mathbb{R}^{I\cdot m}$, 
 We have\vspace{-0.1cm} %It is sufficient to prove that, for all $\mathbf{z}\in \mathbb{R}^{I\cdot m}$, the following bound holds
\begin{align}\label{eq:bound_on_W_hat}
	 	\norm{\left(\widehat{\mathbf{W}}^{n:k} - \mathbf{J}_{\boldsymbol{\phi}^{k}}\right)\mathbf{z}}_2%= \sqrt{\sum_{\ell=1}^{m}\norm{\left(\mathbf{W}^{n:t} - \frac{1}{I}\mathbf{1}\,\boldsymbol{\phi}^{t\top}\right)\mathbf{z}^{\ell}}^{2}_2}\nonumber \\
	 	\leq \sqrt{I\cdot \sum_{\ell=1}^{m} \norm{  \left(\mathbf{W}^{n:k} - \frac{1}{I}\mathbf{1}\,(\boldsymbol{\phi}^{k})^\top\right)\mathbf{z}_{\ell}}^{2}_\infty }.
\end{align}
%To prove the above inequality, we bound next $\|(\mathbf{W}^{n:t} - \frac{1}{I}\mathbf{1}\,\boldsymbol{\phi}^{t\top})\mathbf{z}^{(\ell)}\|_2$.  
We bound next the above term. Given $\boldsymbol{\xi}^k$ as in Lemma~\ref{lem:geo decay row stochastic} [cf.~\eqref{eq:exp_decay_W}],  define $\mathbf{E}^{n:k}\triangleq \mathbf{W}^{n:k} - \mathbf{1}(\boldsymbol{\xi}^{k})^\top$, whose $ij$-th element is denoted by  ${E}^{n:k}_{ij}$. We have 
 %Let $\boldsymbol{\varepsilon}_{i} \in \mathbb{R}^{m}$ and $\boldsymbol{\varepsilon} \triangleq \left(\boldsymbol{\varepsilon}_{i}\right)_{i=1}^{I}$. Furthermore, construct $\boldsymbol{\varepsilon}^{(\ell)}\triangleq \left(\varepsilon_{i}^{(\ell)}\right)_{i=1}^{I}$, where $\varepsilon_{i}^{(\ell)}$ is the $\ell$-th element of $\boldsymbol{\varepsilon}_{i}$. 
	%Decompose $\mathbf{W}^{n:t}$ as $\mathbf{W}^{n:t} = \mathbf{1}\boldsymbol{\xi}^{t\top} + \mathbf{D}^{n:t}$, then it holds that 
	\begin{align}\label{eq: consensus matrix column stochastic}
	&\norm{\left(\mathbf{W}^{n:k} - \frac{1}{I}\mathbf{1}(\boldsymbol{\phi}^{k})^\top\right)\mathbf{z}_{\ell}}_{\infty} \overset{\eqref{eq:def W}} {=} \norm{\left(\mathbf{W}^{n:k} - \frac{1}{I}\mathbf{1}\left(\boldsymbol{\phi}^{n+1}\right)^{\top}\mathbf{W}^{n:k}\right)\mathbf{z}_{\ell}}_{\infty}\nonumber\\
	%&   = \norm{\left(\mathbf{I} - \frac{1}{I}\mathbf{1}\left(\boldsymbol{\phi}^{n+1}\right)^{\top}\right)\left(\mathbf{1}\boldsymbol{\xi}^{t\top} + \mathbf{E}^{n:t}\right)\mathbf{z}^{\ell}}_{\infty}\nonumber\\
	&  = \norm{\left(\mathbf{I} - \frac{1}{I}\mathbf{1}\left(\boldsymbol{\phi}^{n+1}\right)^{\top}\right)\mathbf{E}^{n:k}\,\mathbf{z}_{\ell}}_{\infty}\nonumber\\
	%&   \leq  \max_{1\leq i \leq I}\left\vert\left(1-\frac{\phi_{i}^{n+1}}{I}\right)\sum_{j=1}^{I}E_{ij}^{n:t}\,z_{j}^{\ell} - \sum_{k\neq i}^{I}\frac{\phi_{k}^{n+1}}{I}\sum_{j=1}^{I}E_{kj}^{n:t}\,z_{j}^{\ell}\right\vert \nonumber\\
	&   \leq \max_{1\leq i \leq I}\left(\left(1-\frac{\phi_{i}^{n+1}}{I}\right)\sum_{j=1}^{I}\left\vert E_{ij}^{n:k}\right\vert \left\vert z_{\ell,j}\right\vert + \sum_{j'\neq i}^{I}\frac{\phi_{j'}^{n+1}}{I}\sum_{j=1}^{I}\left\vert E_{j'j}^{n:k}\right\vert \left\vert z_{\ell,j}\right\vert \right)\nonumber\\
	&   \leq 2\,c_{0}\,(\rho)^{\big\lfloor  \frac{n-k+1}{(I-1)B}\big \rfloor}\norm{\mathbf{z}_{\ell}}_{1}\leq 2\,c_{0}\,(\rho)^{\big\lfloor  \frac{n-k+1}{(I-1)B}\big \rfloor}\,\sqrt{I}\,\norm{\mathbf{z}_{\ell}}_2.	
	\end{align} 	
Combining  (\ref{eq:bound_on_W_hat}) and (\ref{eq: consensus matrix column stochastic}) we obtain  
\begin{align}
 	\norm{\widehat{\mathbf{W}}^{n:k} - \mathbf{J}_{\boldsymbol{\phi}^{k}}}_{2} \leq  2c_0 I (\rho)^{\big\lfloor  \frac{n-k+1}{(I-1)B}\big \rfloor}.
\end{align}
Moreover, the matrix difference above can be alternatively uniformly bounded as follows:
\begin{align*}
\norm{\widehat{\mathbf{W}}^{n:k} - \mathbf{J}_{\boldsymbol{\phi}^{k}}} = \| (\mathbf{I} - \mathbf{J}_{\boldsymbol{\phi}^{n+1}}) \widehat{\mathbf{W}}^{n:k}\| \leq \|\mathbf{I} - \mathbf{J}_{\boldsymbol{\phi}^{n+1}}\| \|\widehat{\mathbf{W}}^{n:k}\| \overset{(a)}{\leq} \sqrt{2I} \cdot \sqrt{I},
\end{align*}
where (a) follows from~\eqref{eq:I_minus_J} and $ \|\widehat{\mathbf{W}}^{n:k}\| \leq \sqrt{I}$. This completes the proof. %Lemma~\ref{cor:averaging matrix column stochastic} follows readily. 
\hfill $\square$}
\vspace{-0.3cm}
 {\section{Proof of Lemma~\ref{lem:consensus err contraction column stochastic}}\label{app:lem 5 proof}\vspace{-0.2cm}
Recall the SONATA update written in vector-matrix form in ~\eqref{eq: update vec phi}-\eqref{eq: update vec y}. %, which is restated below for convenience:
%\begin{align}
%\boldsymbol{\phi}^{n+1} & = \mathbf{A}^n \boldsymbol{\phi}^{n} \label{eq: update vec phi appendix}\\
%%\mathbf{x}^{n+1} & = \widehat{\mathbf{W}}^n (\mathbf{x}^n + \alpha^n \Delta \mathbf{x}^n) \label{eq: update vec x appendix}\\
	% \mathbf{y}^{n+1}   &=\widehat{\mathbf{W}}^n\mathbf{y}^n+ (\DiagPhi{n+1})^{-1}\left(\mathbf{g}^{n+1}-\mathbf{g}^{n}\right).\label{eq: update vec y appendix}
%\end{align}
Note that the $x$-update therein is a special case of the perturbed condensed push-sum algorithm~\eqref{eq:consensus_mat_form}, with perturbation $\error^{n+1} = \alpha^n \widehat{\mathbf{W}}^n \mathbf{x}^n$. We can then apply   Proposition~\ref{prop:error_decay} and readily obtain (\ref{eq:consensus err contraction x column stochastic}).  %we have
%\begin{align}
%\begin{split}
%\| \var{x}{n+\bar{B}}\| \leq {}& \rhoopt \| \var{x}{n}\| + \sqrt{2} I \sum_{t=0}^{\bar{B}-1}\alpha^{n+t}\|\widehat{\mathbf{W}}^{n+t} \Delta\mathbf{x}^{n+t} \|\\
%\leq {} &\rhoopt \| \var{x}{n}\| + \sqrt{2} I\cdot \sqrt{I} \sum_{t=0}^{\bar{B}-1}\alpha^{n+t}\| \Delta\mathbf{x}^{n+t} \|.
%\end{split}
%\end{align}

To prove~\eqref{eq:consensus err contraction y column stochastic}, we follow a similar approach:  noticing that the $y$-update in (\ref{eq: update vec y}) is a special case of~\eqref{eq:consensus_mat_form}, with perturbation $\error^{n+1} =  (\DiagPhi{n+1})^{-1}\left(\mathbf{g}^{n+1}-\mathbf{g}^{n}\right)$, we can write
\begin{align*}
\begin{split}
\| \var{y}{n+\bar{B}}\| \leq {} & \rhoopt \| \var{y}{n}\| + \sqrt{2} I \sum_{t = 0}^{\bar{B}-1} \| (\DiagPhi{n+t+1})^{-1}\left(\mathbf{g}^{n+t+1}-\mathbf{g}^{n+t}\right)\|\\
%\leq {} & \rhoopt \| \var{y}{n}\| + \sqrt{2} I L_{\rm mx}\LBphi^{-1}\sum_{t = 0}^{\bar{B}-1} \| \mathbf{x}^{n+t+1}-\mathbf{x}^{n+t} \| \\
%\leq {} & \rhoopt \| \var{y}{n}\| + \sqrt{2} I L_{\rm mx}\LBphi^{-1}\sum_{t = 0}^{\bar{B}-1} \| \mathbf{x}^{n+t+1}-\mathbf{x}^{n+t} \| \\
 \leq {} &\rhoopt\norm{\var{y}{n}} +  \sqrt{2} I\,L_{\rm{mx}}\LBphi^{-1}\sum_{t=0}^{\bar{B}-1}\norm{\widehat{\mathbf{W}}^{n+t}(\mathbf{x}^{n+t} + \alpha^{n+t}\Delta \mathbf{x}^{n+t}) - \mathbf{x}^{n+t}}\\
 \leq {} & \rhoopt\norm{\var{y}{n }} +  \sqrt{2} I\,L_{\rm{mx}}\LBphi^{-1}\sum_{t=0}^{\bar{B}-1}\left(\norm{\widehat{\mathbf{W}}^{n+t}\var{x}{n+t}} + \norm{\var{x}{n+t}} + \alpha^{n+t}\norm{\widehat{\mathbf{W}}^{n+t}\Delta \mathbf{x}^{n+t}}\right)\\
 \leq {} & \rhoopt\norm{\var{y}{n}} +  \sqrt{2} I\,L_{\rm{mx}}\LBphi^{-1}\sum_{t=0}^{\bar{B}-1}\left((\sqrt{I} + 1) \norm{\var{x}{n+t}} + \alpha^{n+t}\sqrt{I}\norm{\Delta \mathbf{x}^{n+t}} \right)\\
 \leq {} & \rhoopt\norm{\var{y}{n}} +   I\sqrt{2I}\,L_{\rm{mx}}\LBphi^{-1}\sum_{t=0}^{\bar{B}-1}\left(2\norm{\var{x}{n+t}} + \alpha^{n+t}\norm{\Delta \mathbf{x}^{n+t}} \right).
\end{split}
\end{align*}
This completes the proof.}

\bibliographystyle{spmpsci} 

\input{revised.bbl}
\end{document}

%% file: ex_shared.tex
\usepackage{lipsum}
\usepackage{graphicx}
\usepackage{epstopdf}
\usepackage{amsfonts,amssymb,amsmath}
\usepackage{algorithmic,algorithm}
\usepackage{setspace}
\usepackage{comment}
\usepackage{cite}
\usepackage{pdflscape}
\usepackage[dvipsnames]{xcolor}

\newtheorem{lem}[theorem]{Lemma}

\newtheorem{prop}[theorem]{Proposition}
\newtheorem{cor}[theorem]{Corollary}
\newtheorem{remk}[theorem]{Remark}

\newtheorem{assumption}{Assumption}

\usepackage{enumitem}

\allowdisplaybreaks
\newcommand{\norm}[1]{\left\lVert#1\right\rVert}
\ifpdf
  \DeclareGraphicsExtensions{.eps,.pdf,.png,.jpg}
\else
  \DeclareGraphicsExtensions{.eps}
\fi

\newcommand{\TheTitle}{An Example Article}

\title{{\TheTitle}\thanks{Submitted to the editors DATE.
\funding{This work was funded by the Fog Research Institute under contract no.~FRI-454.}}}

\author{
  Dianne Doe\thanks{Imagination Corp., Chicago, IL
    (\email{ddoe@imag.com}, \url{http://www.imag.com/\string~ddoe/}).}
  \and
  Paul T. Frank\thanks{Department of Applied Mathematics, Fictional
    University, Boise, ID (\email{ptfrank@fictional.edu},
    \email{jesmith@fictional.edu}).}
  \and
  Jane E. Smith\footnotemark[3]
}

\usepackage{amsopn}

\usepackage{amsmath}
\DeclareMathOperator*{\argmin}{argmin}

\renewcommand*{\theassumption}{\Alph{assumption}}

\newcommand{\wavg}[2]{\bar{\mathbf{#1}}_{\boldsymbol{\phi}^{#2}}}
\newcommand{\Vwavg}[2]{\mathbf{J}_{\boldsymbol{\phi}^{#2}}\mathbf{#1}^{#2}}

\newcommand{\var}[2]{ \mathbf{e}_{#1}^{#2}}
\newcommand{\Deltax}{\Delta\widetilde{\mathbf{x}}_{\boldsymbol{\phi}}}
\newcommand{\Deltaxi}[1]{\Delta\widetilde{\mathbf{x}}_{(#1),\boldsymbol{\phi}}}

\newcommand{\eDeltax}{E_{\Delta\widetilde{\mathbf{x}}}}
\newcommand{\exorth}{E_{\mathbf{x}_{\bot}}}
\newcommand{\eyorth}{E_{\mathbf{y}_\bot}}
\newcommand{\wexorth}[1]{\sum_{k=0}^{\bar{B}-1} \frac{k+1+(\bar{B}-k-1)\tilde{\rho}}{1-\tilde{\rho}}\,\norm{\var{x}{#1+k}}^2}
\newcommand{\weyorth}[1]{\sum_{k=0}^{\bar{B}-1} \frac{k+1+(\bar{B}-k-1)\tilde{\rho}}{1-\tilde{\rho}} \,\norm{\var{y}{#1+k}}^2}

\newcommand{\cumsum}{\sum_{t=0}^{\bar{B}-1}}
\newcommand{\DiagPhi}[1]{\widehat{\mathbf{D}}_{\boldsymbol{\phi}^{#1}}}

\newcommand{\cDeltax}{c_{\Delta}}
\newcommand{\corth}{c_{\bot}}

\newcommand{\LBphi}{\phi_{lb}}
\newcommand{\UBphi}{\phi_{ub}}

\newcommand{\error}{\boldsymbol{\delta}}
\newcommand{\rhoopt}{\rho_{\bar{B}}}

%% file: intro_table.tex
\renewcommand{\arraystretch}{1.3}\begin{table}[t]
\vspace{-0.2cm}
\centering
\resizebox{0.95\textwidth}{!}{
\begin{tabular}{l c | ccccccc}			\hline	
																											& 						&  \begin{tabular}[c]{@{}c@{}}\textbf{Proj-}\\ \textbf{DGM}\\\cite{bianchi2013convergence} \end{tabular}	& \begin{tabular}[c]{@{}c@{}}\textbf{NEXT}\\\cite{LorenzoScutari-J'16} \end{tabular}		& \begin{tabular}[c]{@{}c@{}}\textbf{Push-sum}\\ \textbf{DGM}\\\cite{tatarenko2015non} \end{tabular}  	& \begin{tabular}[c]{@{}l@{}}\textbf{Prox-}\\\textbf{PDA}\\\cite{hong2016decomposing} \end{tabular} 	&\begin{tabular}[c]{@{}l@{}}\textbf{DeFW} \cite{Wai2017DeFW}\end{tabular} 						& \begin{tabular}[c]{@{}l@{}}\textbf{SONATA}\\This work \end{tabular}  \\\hline
\textbf{nonsmooth $G^+$} 					& 						& 																				&\checkmark	&																						&																				&												& \checkmark \\ \hline
\textbf{constraints} 																		 					&  					&\checkmark													  			&\checkmark  &																						&																				&	$\mathcal{K}$ compact			&\checkmark \\\hline
 \textbf{unbounded} \textbf{gradient} 				& 						&																			  	& 			 	&																						&						 \checkmark														&												&\checkmark \\\hline
 \multirow{2}{*}{\begin{tabular}[c]{@{}l@{}}\textbf{network}\vspace{-0.1cm}\\\textbf{topology} \end{tabular} \textbf{:}} &  \textbf{time-varying} 				&																			 	& \checkmark	&	\checkmark																	&																				&												&\checkmark \\	\cline{2-8}
																										  	& \textbf{digraph}   		&																				&	restricted	&	\checkmark																	&																				&												&\checkmark \\\hline			
\multirow{2}{*}{\textbf{step-size: }}    																		& \textbf{constant}  		&																				&					&																						&	\checkmark															&												&\checkmark \\ \cline{2-8}
																											& \textbf{diminishing} 	&\checkmark																& \checkmark &	\checkmark																	&		\checkmark																		&	\checkmark							&\checkmark \\\hline
\textbf{complexity} 																							&						&																				&					&\checkmark																		&\checkmark																& \checkmark								&\checkmark \\\hline
\end{tabular}}\smallskip
\centering 
\caption{Distributed nonconvex optimization: Current works and contribution of this paper.}
\label{nonconvex_table}\vspace{-1.2cm}
\end{table}

%% file: sec_simulation.tex
\section{Numerical results}
\label{sec:experiments}\vspace{-0.2cm}
\subsection{Sparse regression}\label{sec:sparse-regression}\vspace{-0.3cm}
In this section, we test the performance of SONATA on  the   sparse linear regression problem  \eqref{p: sparse regression} [cf.~Sec.~\ref{sec:examples}]. 
We generated the data set as follows. The ground truth signal $\mathbf{x}^\star \in \mathbb{R}^{500}$ is built by first drawing  randomly a vector from the 
normal distribution $\mathcal{N}(\mathbf{0},\mathbf{I})$, then thresholding the smallest $80\%$ of its elements to zero.  The underlying linear model is $\mathbf{b}_i = \mathbf{A}_i \mathbf{x}^\star + \mathbf{n}_i$, where 
the observation  matrix $\mathbf{A}_i \in \mathbb{R}^{20\times 500}$ is generated by first drawing i.i.d. elements from the 
distribution $\mathcal{N}(0,1)$, and  then normalizing the rows to unit norm; and   $\mathbf{n}_i$ is the additive noise, with    i.i.d. entries from  $\mathcal{N}(0,0.1)$.  We simulated  $100$ Monte Carlo trials, generating in each trial new   $\mathbf{A}_i$'s and $\mathbf{n}_i$'s.  
We considered  a time-varying digraph, composed of $I = 30$ agents.  In every time slot, a new digraph is generated according to the following 
procedure:  each agent $i$ has two out-neighbors,  one of them belonging to a  chain connecting all the agents and the other one  picked uniformly at random.   
To promote sparsity we use  %two penalty functions: the $\ell_1$-norm $G=G_{\ell_1}$ and 
the (nonconvex) $\log$ function $G(\mathbf{x})= \lambda\cdot \sum_i \log (1+\theta\,|x_i|)/\log(1+\theta)$, where the parameter $\theta$ controls the tightness of the approximation
of the  $\ell_0$ function. We set $\lambda= 0.1$ and $\theta=2$.  
It is convenient to rewrite  $G(\mathbf{x})$ in the DC  form $G(\mathbf{x})= G^+(\mathbf{x})-G^-(\mathbf{x})$, with $G^+(\mathbf{x})=\|\mathbf{x}\|_1\cdot (\theta/\log(1+\theta))$. %and $G^-(\mathbf{x})=G(\mathbf{x})-G^+(\mathbf{x})$.  
It is not difficult to check that 
such $G^+$ and $G^-$ satisfy Assumption \ref{assumption:P}.3; see, e.g., \cite{pang2016PartI}.

% we set $\tau_{PL}$ and $\tau_L$ as $1.5$. ..........The regularization parameter $\lambda$ is set to $\lambda= 0.1$; for $G_{\log}$,     $\theta$ in (\ref{eq: dc decompose})  is set t`o   $\theta=2$. 

We run SONATA considering two alternative choices of $\widetilde{f}_i$, namely:

 %termed ``SONATA-PL'' (PL stands for \emph{partial linearization})  and ``SONATA-L'' (L stands for \emph{linearization}), respectively, wherein we used two different surrogates $\widetilde{f}_i$, as detailed next.

\noindent$\bullet$ \texttt{SONATA-PL} (PL stands for \emph{partial linearization}): Since $f_i = \|\mathbf{b}_i - \mathbf{A}_i \mathbf{x}\|^2$ is convex, one can keep $f_i$ unalterated and set in (\ref{surrogate_F}) $\widetilde{f}_i(\mathbf{x}_{(i)}) = f_i(\mathbf{x}_{(i)}) +  ({\tau_{PL}}/{2})\cdot \|\mathbf{x}_{(i)} - \mathbf{x}_{(i)}^n\|^2$. We set $\tau_{PL}=1.5$.  
%\begin{align*}
%\widetilde{F}_i\left(\mathbf{x}_{(i)}^n; \mathbf{x}_{(i)}^n\right) = &\norm{\mathbf{b}_i - \mathbf{A}_i \mathbf{x}}_2^2 +\left(\widetilde{\boldsymbol{\pi}}_i^n- \lambda\nabla G^- \left(\mathbf{x}_{(i)}^n\right)\right)^\top \left(\mathbf{x}_{(i)} - \mathbf{x}_{(i)}^n\right)\\
%& + \frac{\tau_{PL}}{2}\norm{\mathbf{x}_{(i)} - \mathbf{x}_{(i)}^n}_2^2.
%\end{align*}
The unique solution   $\widetilde{\mathbf{x}}_{(i)}^n$ of the resulting subproblem  \eqref{eq: x_tilde}  is computed  using the  FLEXA algorithm, with 
the following tuning (see \cite{facchinei2015parallel} for details): 
the initial point is selected randomly; the proximal parameter in the subproblems solved by FLEXA  is  set to be $2$; and the step-size of FLEXA  is chosen 
according to the diminishing rule $\gamma^r = \gamma^{r-1}\left(1-\mu\gamma^{r-1}\right),$ with  $\gamma^0 = 0.5$ and $\mu = 0.01$, with  $r$ denoting the (inner) iteration index. We terminate FLEXA when $J_{(i)}^r\leq 10^{-8}$, with $J_{(i)}^r\triangleq \|\mathbf{x}_{(i)}^{n,r}- \mathcal{S}_{\eta\lambda} (\mathbf{x}_{(i)}^{n,r} -2\mathbf{A}_i^\top\,(\mathbf{A}_i \mathbf{x}_{(i)}^{n,r} - \mathbf{b}_i)- \tau_{PL} \cdot(\mathbf{x}_{(i)}^{n,r} -\mathbf{x}_{(i)}^{n})+ \widetilde{\boldsymbol{\pi}}_i^n+\lambda \nabla G^- (\mathbf{x}_{(i)}^n))\|_\infty$, where  
%$\left(\mathbf{x}_{(i)}^{n,r}\right) $, defined as
%\begin{align*}
%J_{(i)}\left(\mathbf{x}_{(i)}^{n,r},\mathbf{x}_{(i)}^n\right) \triangleq &\norm{\mathbf{x}_{(i)}^{n,r}- \mathcal{S}_{\eta\lambda}\left(\mathbf{x}_{(i)}^{n,r} -2\mathbf{A}_i^\top\left(\mathbf{A}_i \mathbf{x}_{(i)}^{n,r} - \mathbf{b}_i\right)\right.\right.\\
%&\left.\left.\hspace{2cm}- \tau_{PL}\left(\mathbf{x}_{(i)}^{n,r} -\mathbf{x}_{(i)}^{n}  \right)+ \widetilde{\boldsymbol{\pi}}_i^n+\lambda \nabla G^- \left(\mathbf{x}_{(i)}^n\right)\right)}_\infty.\\
%\end{align*}
%drops below $10^{-8}$, wherein 
$\mathbf{x}_{(i)}^{n,r}$ denotes  the value of $\mathbf{x}_{(i)}$ at the $n$-th outer and the $r$-th inner iteration,
 and   $\mathcal{S}_\beta \left(\mathbf{x}\right) \triangleq 
 \rm{sign}\left(\mathbf{x}\right)\cdot \max\{\vert \mathbf{x}\vert - \lambda\mathbf{1},\mathbf{0} \}$ is the   soft-thresholding operator  (intended to be applied to $\mathbf{x}$ component-wise). 
\\
\noindent$\bullet$ \texttt{SONATA-L} (L stands for \emph{linearization}): %In SONATA-PL, $\widetilde{\mathbf{x}}_{(i)}^n$ has to be computed locally using an iterative algorithm.
 To obtain a closed form expression for  $\widetilde{\mathbf{x}}_{(i)}^n$ in (\ref{surrogate_F}), one can choose $\widetilde{f}_i$ as  linearization of $f_i$ (plus the proximal term), that is, $\widetilde{f}_i(\mathbf{x}_{(i)}) = 2\mathbf{A}_i^\top(\mathbf{A}_i \mathbf{x}_{(i)}^n - \mathbf{b}_i +    ({\tau_{L}}/{2})\cdot \|\mathbf{x}_{(i)} - \mathbf{x}_{(i)}^n\|^2$. We set $\tau_{L}=1.5$.
 %the surrogate
% \begin{align*}
%\widetilde{F}_i\left(\mathbf{x}_{(i)}^n; \mathbf{x}_{(i)}^n\right) =&\left(2\mathbf{A}_i^\top\left(\mathbf{A}_i \mathbf{x}_{(i)}^n - \mathbf{b}_i\right) + \widetilde{\boldsymbol{\pi}}_i^n-\lambda \nabla G^- \left(\mathbf{x}_{(i)}^n\right)\right)^\top \left(\mathbf{x}_{(i)} - \mathbf{x}_{(i)}^n\right)\\
%& + \frac{\tau_L}{2}\norm{\mathbf{x}_{(i)} - \mathbf{x}_{(i)}^n}_2^2.
%\end{align*}
The solution $\widetilde{\mathbf{x}}_{(i)}^n$ of the resulting subproblem (\ref{surrogate_F}) has the following closed form expression $\widetilde{\mathbf{x}}_{(i)}^n = \mathcal{S}_{\eta\lambda/\tau_L}(\mathbf{x}_{(i)}^n - \frac{1}{\tau_L}(2\mathbf{A}_i^\top(\mathbf{A}_i \mathbf{x}_{(i)}^n - \mathbf{b}_i) + \widetilde{\boldsymbol{\pi}}_i^n-\lambda \nabla G^- (\mathbf{x}_{(i)}^n)))$.
%\begin{equation}
%\widetilde{\mathbf{x}}_{(i)}^n = \mathcal{S}_{\eta\lambda/\tau_L}\left(\mathbf{x}_{(i)}^n - \frac{1}{\tau_L}\left(2\mathbf{A}_i^\top\left(\mathbf{A}_i \mathbf{x}_{(i)}^n - \mathbf{b}_i\right) + \widetilde{\boldsymbol{\pi}}_i^n-\lambda \nabla G^- \left(\mathbf{x}_{(i)}^n\right)\right)\right).
%\end{equation}

 As benchmark, we also simulated  the subgradient-push algorithm~\cite{nedic2015distributed} with diminishing step-size. Note that there is no  proof of convergence for such a scheme, when applied to the nonconvex, nonsmooth problem \eqref{p: sparse regression}. For all the algorithms, we  use the same step-size rule:  $\alpha^n = \alpha^{n-1}\left(1-\mu\alpha^{n-1}\right)$, with  $\alpha^0 = 0.5$ and $\mu = 0.01$.  Also, for all algorithms, we set $\mathbf{x}_{(i)}^0=\mathbf{0}$, for all $i$.

We monitor the progresses of the algorithms towards stationarity and consensus using respectively  the following two functions: i) $J^n\triangleq \|\bar{\mathbf{x}}^n - \mathcal{S}_{\eta\lambda}(\bar{\mathbf{x}}^n -2\sum_i\mathbf{A}_i^\top\left(\mathbf{A}_i \bar{\mathbf{x}}^n - \mathbf{b}_i) +\lambda \nabla G^- \left(\bar{\mathbf{x}}^n\right)\right)\|_\infty$; and ii) 
 $D^n\triangleq \|\mathbf{x}^n - \mathbf{J}\mathbf{x}^n\|_\infty$.   
%\begin{align}
%J\left(\mathbf{x}^n\right) =&\norm{\bar{\mathbf{x}}^n - \mathcal{S}_{\eta\lambda}\left(\bar{\mathbf{x}}^n -2\sum_{i=1}^I\mathbf{A}_i^\top\left(\mathbf{A}_i \bar{\mathbf{x}}^n - \mathbf{b}_i\right) +\lambda \nabla G^- \left(\bar{\mathbf{x}}^n\right)\right)}_\infty \nonumber\\
 %D\left(\mathbf{x}^n\right)=&\norm{\mathbf{x}^n - \overline{\mathbf{x}}^n}_\infty.\label{eq:merit-D}
%\end{align}
It is not difficult to check that $J^n$ is a valid distance of the average iterates $\mathbf{J}\mathbf{x}^n$ from stationarity: it is continuous and zero if and only if its argument is a stationary solution of  \eqref{p: sparse regression}. %Therefore, $J\left(\mathbf{x}^n\right)$ measures the distanze from stationarity of the average sequence $\{\mathbf{x}^n\}_{n\in \mathbb{N}_+}$. 
%The function $D^n$ measures the disagreement among the agents' local variables. 
 We also use the normalized  mean squared error (NMSE), defined as $\text{NMSE}^n \triangleq  \|\mathbf{x}^n - (\mathbf{1}\otimes\mathbf{I})\,{\mathbf{x}}^\star\|^2/(I\cdot \|\mathbf{x}^\star\|^2)$.
 %\begin{equation}\label{eq:merit-NMSE}
 %\text{NMSE}(\mathbf{x}^n) = \norm{\mathbf{x}^n - \mathbf{J}\mathbf{x}}_2^2/(I\cdot \norm{\mathbf{x}}_2^2),
% \end{equation}

%The $5\%$ trimmed mean of $J(\mathbf{x}^n)$, $D(\mathbf{x}^n)$,  
 In Fig. \ref{fig:sparse_regression_log}, we plot  $\log_{10}J^n$ and  $\log_{10}D^n$  [subplot (a)] and the NMSE [subplot (b)]
   versus the number of agents' message exchanges, averaged over $100 $ Monte-Carlo trials (we applied the $\log_{10}$ transform to $J^n$ and $D^n$
so that their distribution is  closer to  the normal one). 
%Note that   one message exchange of the  subgradient-push  corresponds to two messages of SONATA. 
 The figures show that both versions of  SONATA are  much 
 faster than the  distributed gradient algorithm. This seems mainly due to the gradient tracking mechanism put forth by the proposed scheme. Under  the same  tuning, SONATA-PL converges  faster than SONATA-L. According to  our intensive simulations (not reported here),
 SONATA-PL becomes up to one order of magnitude
 faster than SONATA-L when    $\tau_{PL}$ is reduced whereas   reducing $\tau_L$ slows down SONATA-L. \vspace{-0.5cm}%, even causes divergence in practice,
%since $\mathbf{x}^n$ gets too large in the first few iterations.
%More specifically,  
%our experiments show that  SONATA-PL converges  up to one order of magnitude
% faster than SONATA-L when a proper tuning is used (different from the one of SONATA-L). %ally, a comparison of Fig. \ref{fig:LASSO}(a) and Fig. \ref{fig:sparse_regression_log}(a)  shows SONATA converges faster for convex problem (Fig. \ref{fig:LASSO}(a))
% than nonconvex problem (Fig. \ref{fig:sparse_regression_log}(a)).
%In terms of statistical  performance, Fig. \ref{fig:LASSO}(b) and Fig. \ref{fig:sparse_regression_log}(b) reveal that the NMSE 
%stabilizes after $400$ message exchanges, and using nonconvex regularizer $G_{\log}$ achieves smaller estimation error compared to
%using convex regularizer $G_{\ell_1}$.
 %\begin{figure}
%\begin{tabular}{cc}
%\includegraphics{Figure/LASSO_log_opt} &\includegraphics{Figure/LASSO_MSE_opt}\\
%\small{(a)} & \small{(b)}
%\end{tabular}
%\caption{Performance comparison of SONATA-PL, SONATA-L, and subgradient-push for sparse regression Problem \eqref{p: sparse regression} with regularizer $G_{\ell_1}$.  Fig. 1(a): average of $\log_{10}J(\mathbf{x}^n)$ and
 %$\log_{10}D(\mathbf{x}^n)$; Fig. 1(b) average of NMSE.}\label{fig:LASSO}
%\end{figure}
\begin{figure}
\begin{tabular}{cc}
 \includegraphics{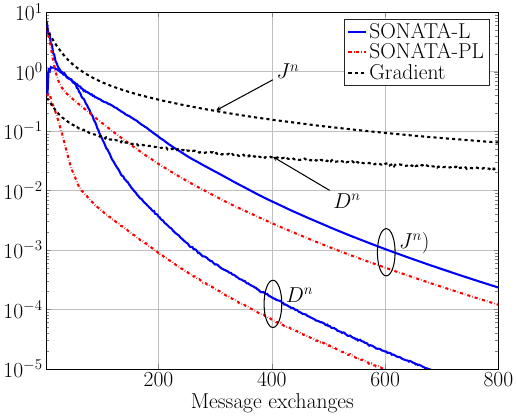}&\includegraphics{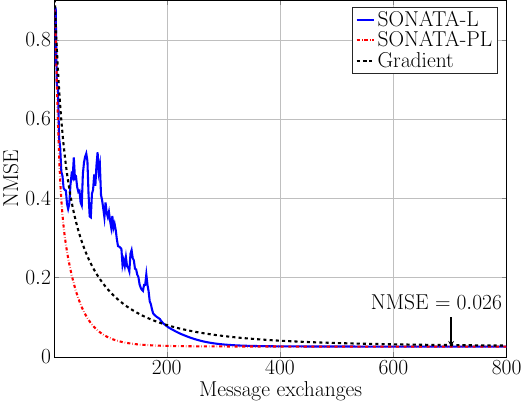}\vspace{-0.2cm}\\
\footnotesize{(a)} & \footnotesize{(b)}
\end{tabular}\vspace{-0.2cm}
\caption{Sparse regression problem \eqref{p: sparse regression} with log regularizer:  SONATA-PL, SONATA-L, and subgradient-push;  average of $\log_{10}J^n$ and
 $\log_{10}D^n$ vs. agent's message exchange [subplot (a)]; average of NMSE vs. agent's message exchange [subplot (b)].}\label{fig:sparse_regression_log}\vspace{-0.2cm}
\end{figure}

\subsection{Distributed PCA}\label{sec:DPCA}\vspace{-0.2cm}

Our second application is the distributed PCA problem\vspace{-0.2cm}
 \begin{equation}
 \begin{aligned}
 &\underset{\norm{\mathbf{x}}_2 \leq 1}{\textrm{min}}& & F\left(\mathbf{x}\right)\triangleq - \sum_{i=1}^{I} \norm{\mathbf{D}_i\mathbf{x}}^2,
 \end{aligned}\vspace{-0.2cm}
 \end{equation}
 with $I=30$.

  Each agent $i$ locally owns a data matrix $\mathbf{D}_i \in \mathbb{R}^{d_i\times m}$ and  communicate via a time-varying  digraph generated in the same way as the previous sparse regression example (cf. Sec.~\ref{sec:sparse-regression}). 

Since $f_i(\mathbf{x})\triangleq - \norm{\mathbf{D}_i\mathbf{x}}^2$ is concave, to apply SONATA we construct $\widetilde{f}_i$ by linearizing $f_i$, which leads to 
 %\begin{equation}
 $\widetilde{F}_i(\mathbf{x}_{(i)};\mathbf{x}_{(i)}^n)= (I \cdot \mathbf{y}_{(i)}^n)^\top (\mathbf{x}_{(i)} - \mathbf{x}_{(i)}^n)
 + ({\tau}/{2})\cdot\|\mathbf{x}_{(i)}-\mathbf{x}_{(i)}^n\|^2.$
% \end{equation}
 The solution $\widetilde{\mathbf{x}}_{(i)}^n$ of the resulting subproblem  has the closed form solution 
% \begin{equation}
 $\widetilde{\mathbf{x}}_{(i)}^n = \mathcal{P}_{\|\mathbf{x}_{(i)}\| \leq 1}(\mathbf{x}_{(i)}^n - I\cdot\mathbf{y}_{(i)}^n/\tau),$
% \end{equation}
 where $\mathcal{P}$ denotes the Euclidean projection onto the set $\{\mathbf{x}_{(i)}\,:\,\|\mathbf{x}_{(i)}\| \leq 1\}$. 
 As benchmark, we implemented also the gradient projection algorithm \cite{bianchi2013convergence}, adapted to time-varying network. Note that there is no formal proof of this algorithm in the simulated setting. 
  The performance of the algorithms is tested on both synthetic and real data sets, as detailed next.\vspace{-0.2cm}
  
 \subsubsection{Synthetic data}\label{sec:PCA_synthetic_data}\vspace{-0.2cm}
  Each agent $i$ locally owns a data matrix $\mathbf{D}_i \in \mathbb{R}^{30\times 500}$, whose rows are i.i.d., drawn by the $\mathcal{N}\left(\mathbf{0},\boldsymbol{\Sigma}\right)$.
  The covariance matrix $\boldsymbol{\Sigma}$, whose  eigendecomposition is $\boldsymbol{\Sigma} = \mathbf{U}\boldsymbol{\Lambda}\mathbf{U}^T$, is generated as follows:
 we synthesize   $\mathbf{U}$ by first generating a square matrix whose entries follow the i.i.d. standard normal distribution, then perform the QR decomposition
 to obtain its orthonormal basis; and the eigenvalues $\rm{diag}(\boldsymbol{\Lambda})$ are i.i.d. uniformly distributed in $[0,1]$. %We set the data dimension $m = 500$ and let each agent own $d_i = 30$ samples, i.e., $\mathbf{D}_i \in \mathbb{R}^{30\times 500}$.
 
 The algorithms are 
  tuned as follows: $\mathbf{x}_{(i)}^0$ is generated with  i.i.d elements drawn by the standard Normal distribution. The step-size $\alpha^n$ is chosen according to the diminishing rule used in the previous example, where we set  $\alpha^0 = 1$
  and $\mu= 10^{-3}$ for SONATA and $\alpha^0 = 1$ and $\mu= 10^{-2}$ for the gradient algorithm.  The proximal parameter $\tau$ for SONATA is set to be $1$.
  The distance of $\bar{\mathbf{x}}^n$ from stationarity is  measured by 
 $J^n\triangleq  \|\bar{\mathbf{x}}^n -  \mathcal{P}_{\|\mathbf{x}_{(i)}\| \leq 1}(\bar{\mathbf{x}}^n - \nabla F(\bar{\mathbf{x}}^n))\|_\infty,$
 %\end{equation}
 while the consensus disagreement $D^n$ and the  $\textrm{NMSE}^n$ are defined as in the previous example; in the definition of $\textrm{NMSE}^n$
  the ground truth signal $\mathbf{x}^\star$ is now the leading eigenvector of matrix $\sum_{i=1}^I \mathbf{D}_i^\top \mathbf{D}_i$.

  In Fig.~\ref{fig:DPCA}, we plot  $\log_{10}J^n$ and  $\log_{10}D^n$  [subplot (a)] and the  NMSE [subplot (b)]
   versus the number of agents' message exchanges, averaged over $100 $ Monte-Carlo trials. In each trial, $\boldsymbol{\Sigma}$ is fixed while the $\mathbf{D}_i$'s are randomly generated. Fig.~\ref{fig:DPCA}(a)  clearly shows that SONATA can  find a stationary point efficiently while the gradient algorithm progresses very slowly. More interestingly,   
    Fig.~\ref{fig:DPCA}(b) shows that SONATA always find the leading eigenvector whereas  the gradient algorithm fails to achieve  a small NMSE value.  \vspace{-0.2cm}
  
 %  \begin{figure}
%\begin{tabular}{cc}
%\includegraphics{Figure/DPCA_opt} &\includegraphics{Figure/DPCA_MSE}\vspace{-0.2cm}\\
%\footnotesize{(a)} & \footnotesize{(b)}\vspace{-0.2cm}
%\end{tabular}
%\caption{Distributed PCA Problem \eqref{p: sparse regression}:  SONATA and gradient projection algorithm;  average of $\log_{10}J^n$ and
 %$\log_{10}D^n$ vs. agent's message exchange [subplot (a)]; average of NMSE vs. agent's message exchange [subplot (b)].}\vspace{-0.4cm}\label{fig:DPCA}
%\end{figure}

  \begin{figure}
\begin{tabular}{cc}
\includegraphics{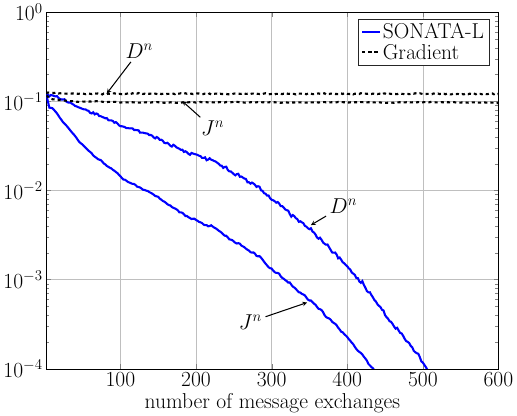} &\includegraphics{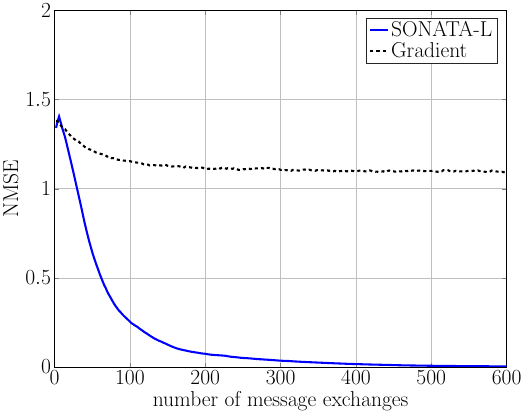}\vspace{-0.2cm}\\
\footnotesize{(a)} & \footnotesize{(b)}\vspace{-0.2cm}
\end{tabular}
\caption{Distributed PCA Problem \eqref{p: sparse regression} on synthetic data set:  SONATA and gradient projection algorithm;  average of $\log_{10}J^n$ and
 $\log_{10}D^n$ vs. agent's message exchange [subplot (a)]; average of NMSE vs. agent's message exchange [subplot (b)].}\vspace{-0.4cm}\label{fig:DPCA}
\end{figure}

\subsubsection{Gene expression data}\vspace{-0.2cm}
This second experiment  tests   SONATA on a real-world data set. Specifically, we used the breast cancer gene expression data set  \cite{bild2006oncogenic}, which consists of  $d = 158$ samples and $m = 12625$ genes per sample.  We first uniformly randomly permute the order the samples and then equally divided the samples among the $I= 30$ agents. To avoid the issue that $d$ is not divisible by $I$, we let the first $I-1$ agents   owing $d_i = \lfloor d/I \rfloor$ samples each, while the $I$-th agent owning the remaining samples. The samples are preprocessed by subtracting the mean from all of them. Note that this can be achieved distributively by running an average consensus algorithm beforehand.

The rest of the setting and tuning of the algorithms are the same of those described in  Sec.~\ref{sec:PCA_synthetic_data}. In Fig.~\ref{fig:DPCA_gene_data}, we plot  $\log_{10}J^n$ and  $\log_{10}D^n$  [subplot (a)] and the  NMSE [subplot (b)]
   versus the number of agents' message exchanges, averaged over $100 $ Monte-Carlo trials. In each trial, samples are randomly partitioned among the agents. From the figure we can see that the behavior of the algorithms on the gene expression data set is similar to that on synthetic data set. Moreover, SONATA  converges  quite fast even though the variable dimension of the real data set we adopted is massive.\vspace{-0.4cm}
   
     \begin{figure}
\begin{tabular}{cc}
\includegraphics{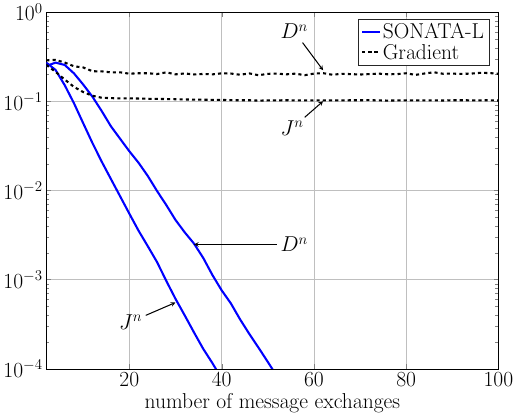} &\includegraphics{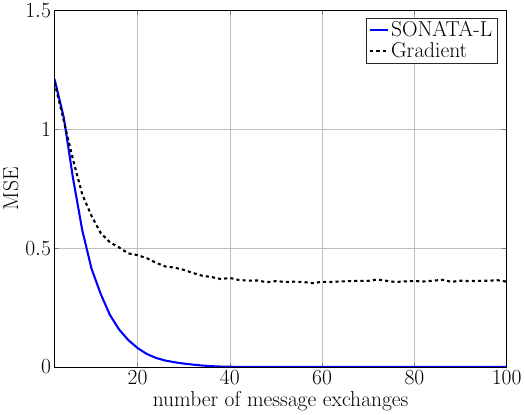}\vspace{-0.2cm}\\
\footnotesize{(a)} & \footnotesize{(b)}\vspace{-0.2cm}
\end{tabular}
\caption{Distributed PCA Problem \eqref{p: sparse regression} {on gene expression data set}:  SONATA and gradient projection algorithm;  average of $\log_{10}J^n$ and
 $\log_{10}D^n$ vs. agent's message exchange [subplot (a)]; average of NMSE vs. agent's message exchange [subplot (b)].}\vspace{-0.4cm}\label{fig:DPCA_gene_data}
\end{figure}